\documentclass[11pt]{article}

\usepackage{xcolor}
\usepackage{amsmath}
\usepackage{mathtools}
\usepackage{amssymb,mathrsfs,euscript}
\usepackage{enumerate}
\usepackage[normalem]{ulem}
\usepackage{cancel}
\usepackage{enumitem}
\usepackage{tikz}
\usepackage{dsfont}
\usepackage{verbatim,authblk,amsthm}
\usepackage[round]{natbib}
\usepackage{float,graphicx,subcaption}

\newcommand*\circled[1]{\tikz[baseline=(char.base)]{
            \node[shape=circle,draw,inner sep=2pt] (char) {#1};}}
            
\newcommand{\inner}[2]{{\left\langle #1, #2 \right\rangle}}            
\newcommand{\norm}[1]{\left\|#1\right\|}

\newcommand{\X}{\mathcal X}
\newcommand{\Y}{\mathcal Y}
\newcommand{\R}{\mathbb R}
\newcommand{\N}{\mathbb N}
\newcommand{\E}{\mathbb E}
\newcommand{\h}{\mathscr H}
\newcommand{\hh}{\boldsymbol{H}}
\newcommand{\Id}{\boldsymbol{I}}
\newcommand{\one}{\boldsymbol{1}}

\newcommand{\B}{\mathcal B}
\newcommand{\M}{\mathcal M}

\newcommand{\id}{\mathfrak I}
\newcommand{\T}{\mathcal T}

\newcommand{\PP}{\mathcal P}
\newcommand{\Ntl}{\mathcal{N}_{2}(\lambda)}
\newcommand{\Ntlh}{\widehat{\mathcal{N}}_{2}(\lambda)}
\newcommand{\Ntlsq}{\mathcal{N}^2_{2}(\lambda)}
\newcommand{\Ntlsqh}{\widehat{\mathcal{N}}^2_{2}(\lambda)}

\newcommand{\Nol}{\mathcal{N}_{1}(\lambda)}
\newcommand{\Nolsq}{\mathcal{N}^2_{1}(\lambda)}
\newcommand{\Cs}{(C_1+C_2)}
\newcommand{\Cl}{C_{\lambda}}
\newcommand{\II}{\mathds{1}}
\newcommand{\K}{\kappa}
\newcommand{\kk}{K}

\newcommand{\op}{\EuScript{L}^\infty(\h)}
\newcommand{\opl}{\EuScript{L}^\infty(\Lp)}
\newcommand{\hs}{\EuScript{L}^2(\h)}
\newcommand{\Lp}{L^{2}(P_0)}
\newcommand{\range}{\text{Ran}} 

\newcommand{\SgL}{\Sigma_{0,\lambda}^{-1/2}}
\newcommand{\SgLh}{\hat{\Sigma}_{0,\lambda}^{-1/2}}
\newcommand{\SL}{\Sigma_{0,\lambda}}
\newcommand{\SLh}{\hat{\Sigma}_{0,\lambda}}

\newcommand{\gShh}{g^{1/2}_{\lambda}(\hat{\Sigma}_0)}

\newcommand{\gT}{g_{\lambda}(\T)}
\newcommand{\gl}{g_{\lambda}}

\newcommand{\U}{u}
\newcommand{\gSL}{g^{1/2}_{\lambda}(\Sigma_0)}

\newcommand{\gSLh}{g^{1/2}_{\lambda}(\hat{\Sigma}_0)}

\newcommand{\stat}{\hat{\eta}^{TS}_{\lambda}}

\newcommand{\htens}{\otimes_{\h}}
\newcommand{\ltens}{\otimes_{\Lp}}

\newcommand{\qq}{q_{1-\alpha}^{\lambda}}
\newcommand{\PQ}{P_0}

\newcommand{\cd} {|\Lambda|}

\renewcommand{\epsilon}{\varepsilon}

\newtheorem{thm}{Theorem}
\newtheorem{rem}{Remark}
\newtheorem{cor}[thm]{Corollary}
\theoremstyle{example} 
\newtheorem{example}{Example}

\begin{document}

\title{Spectral Regularized Kernel Goodness-of-Fit Tests}
\author{Omar Hagrass}
\author{Bharath K. Sriperumbudur}
\author{Bing Li}
\affil{Department of Statistics,  
Pennsylvania State University\\
University Park, PA 16802, USA.\\
\texttt{\{oih3,bks18,bxl9\}}@psu.edu}
\date{}
\maketitle

\begin{abstract}
Maximum mean discrepancy (MMD) has enjoyed a lot of success in many machine learning and statistical applications, including non-parametric hypothesis testing, because of its ability to handle non-Euclidean data. Recently, it has been demonstrated in \cite{Krishna} that the goodness-of-fit test based on MMD is not minimax optimal while a Tikhonov regularized version of it is, for an appropriate choice of the regularization parameter. However, the results in \cite{Krishna} are obtained under the restrictive assumptions of the \emph{mean element} being zero, and the \emph{uniform boundedness condition} on the eigenfunctions of the integral operator. Moreover, the test proposed in \cite{Krishna} is not practical as it is not computable for many kernels. In this paper, we address these shortcomings and extend the results to general spectral regularizers that include Tikhonov regularization.
\end{abstract}
\setlength{\parskip}{4pt}

\section{Introduction}\label{sec:intro}
Given $\mathbb{X}_n:=(X_i)_{i=1}^n \stackrel{i.i.d}{\sim} P$, where $P$ is defined on a measurable space $\X$, a goodness-of-fit test involves testing $H_0 : P=P_0$ against $H_1: P \neq P_0$, where $P_0$ is a fixed known distribution. This is a classical and well-studied problem in statistics for which many tests have been proposed, including the popular ones such as the $\chi^2$-test and Kolmogorov-Smirnoff test \citep{lehmann}. However, many of these tests either rely on strong distributional assumptions or cannot handle non-Euclidean data that naturally arise in many modern applications.

A non-parametric testing framework that has gained a lot of popularity over the last decade is based on the notion of reproducing kernel Hilbert space (RKHS) \citep{Aronszajn} embedding of probability distributions (\citealt{Smola}, \citealt{classifiability}, \citealt{rkhs}). The power of this framework lies in its ability to handle data that is not necessarily Euclidean. 
This framework involves embedding a probability measure $P$ into an RKHS, $\mathscr{H}$ through the corresponding mean element, i.e., 
$$\mu_P=\int_{\X}K(\cdot,x)\,dP(x)\in\mathscr{H},$$
where $K:\X\times\X\rightarrow\mathbb{R}$ is the unique reproducing kernel (r.k.)~associated with $\mathscr{H}$, with $P$ satisfying  $\int_\X \sqrt{K(x,x)}\,dP(x)<\infty$. Using this embedding, a pseudo-metric can be defined on the 
space of probability measures, called the \emph{maximum mean discrepancy} ($\mathrm{MMD}$) (\citealt{gretton12a}, \citealt{NipsGretton}), as \begin{equation}D_{\mathrm{MMD}}(P,Q) = \norm{\mu_P-\mu_Q}_{\h},\nonumber
\end{equation} which has the following variational representation (\citealt{gretton12a}, \citealt{JMLR_metrics}), 
\begin{equation}
D_{\mathrm{MMD}}(P,Q) := \sup_{f \in \h : \norm{f}_{\h}\leq 1} \int_{\X} f(x)\,d(P-Q)(x).\nonumber
\end{equation}
We refer the interested reader to (\citealt{JMLR_metrics}, \citealt{bernouli2016}, \citealt{Carl}) for more details about $D_{\mathrm{MMD}}$.
Thus given some fixed $P_0$, a consistent goodness-of-fit test can be conducted by using the following estimator of $D^2_{\mathrm{MMD}}$ as a test statistic, i.e., 
\begin{align*}
\hat{D}^2_{\mathrm{MMD}}(P,P_0) &:=\frac{1}{n(n-1)}\sum_{i\neq j}\inner{\kk(\cdot,X_i)-\mu_0}{\kk(\cdot,X_j)-\mu_0}_{\h}\\
&=\frac{1}{n(n-1)}\sum_{i\ne j} K(X_i,X_j)-\frac{2}{n}\sum^n_{i=1}\mu_0(X_i)+\Vert \mu_0\Vert^2_\h,
\end{align*}
and using the $1-\alpha$ quantile of the asymptotic null distribution of $\hat{D}^2_{\mathrm{MMD}}(P,P_0)$ as the critical level (\citealp[Theorem 1]{Krishna}), while assuming $\mu_0:=\mu_{P_0}$ and $\Vert \mu_0\Vert^2_{\h}$ are computable in closed form. The asymptotic null distribution of $\hat{D}^2_{\mathrm{MMD}}(P,P_0)$ does not have a simple closed form---the distribution is that of an infinite sum of weighted chi-squared random variables with the weights being the eigenvalues of an integral operator associated with the kernel $K$ w.r.t.~the distribution $P_0$ \citep{Serfling-90}. Assuming $\mu_{0}=0$, recently, \citep{Krishna} showed this test based on $\hat{D}_{\text{MMD}}$ to be not optimal in the minimax sense and  modified it to achieve a minimax optimal test. \citet{MingYuan} constructed an optimal test by using the Gaussian kernel on $\X=\R^d$ (the test and analysis can be extended to translation invariant kernels on $\R^d$ using the ideas in \citealp{MMDagg}) by allowing the bandwidth of the kernel to shrink to zero as $n\rightarrow \infty$---this is in contrast to the $\hat{D}_{\text{MMD}}$ test where the bandwidth or the kernel parameter is fixed and does not depend on $n$. By relaxing the requirement of $\X=\R^d$, 
\cite{Krishna} studied the question of optimality for general domains by proposing a regularized test statistic, 
\begin{equation}
   D^2_{\lambda}(P,P_0) = \sum_{j \geq 1} \frac{\lambda_j}{\lambda_j+\lambda} \left(\E_{P}\phi_j\right)^2, \label{eq:regularized_test}
\end{equation} assuming $\E_{P_0}\phi_j=0$ for all $j$, where $(\lambda_j)_{j\geq1}$ and $(\phi_j)_{j\geq1}$ are the eigenvalues and eigenfunctions of an integral operator associated with the kernel $K$ w.r.t.~the distribution $P_0$, and $\lambda>0$ is the regularization parameter. Under some regularity conditions, they showed the asymptotic null distribution of an appropriately normalized version of \eqref{eq:regularized_test} to be the standard normal distribution, using which a minimax optimal goodness-of-fit test was constructed (\citealt[Theorems 2--4]{Krishna}). However, this test is impractical and limited for two reasons: (i) The test requires knowledge of the eigenvalues and eigenfunctions which are only known for a few $(K, P_0)$ pairs, and (ii) $\E_{P_0}\phi_j=0$ for all $j$ implies that $\mu_0=0$, a condition that is not satisfied by any characteristic translation invariant kernels on $\R^d$, including the Gaussian kernel \citep{JMLR_metrics, JMLR_universal}. To address these issues, in this paper, we follow an operator theoretic approach and construct a generalized version of \eqref{eq:regularized_test} based on the idea of spectral regularization that includes \eqref{eq:regularized_test} as a special case while relaxing the requirement of $\E_{P_0}\phi_j=0$ for all $j$---hence resolving (ii)---, and establish its minimax optimality. Moreover, under an additional assumption of $P_0$ being samplable, i.e., extra samples are available from $P_0$, we propose a practical test (i.e., computable) that is also minimax optimal, thereby resolving the issue mentioned in (i).

Before introducing our contributions, we will first introduce the minimax framework pioneered by \cite{Burnashev} and \cite{Ingester1, Ingester2}  to study the optimality of tests, which is essential to understand our contributions and their connection to the results of \citep{Krishna,MingYuan}. Let $\phi(\mathbb{X}_n)$ be any test that rejects $H_0$ when $\phi=1$ and fails to reject $H_0$ when $\phi=0$. Denote the class of all such asymptotic (\emph{resp.} exact) $\alpha$-level tests to be $\Phi_\alpha$ (\emph{resp.} $\Phi_{n,\alpha}$). The Type-II error of a test $\phi\in \Phi_\alpha$ (\emph{resp.} $\in\Phi_{n,\alpha}$) w.r.t.~$\mathcal{P}_\Delta$ is defined as
$$R_\Delta(\phi)=\sup_{P\in\mathcal{P}_\Delta}\mathbb{E}_{P^n}[1-\phi],$$ 
where \begin{equation}\mathcal{P}_\Delta:=\left\{P \in \mathcal{C}:\rho^2(P,P_0)\ge \Delta\right\},\nonumber
\end{equation} is the class of $\Delta$-separated alternatives in the probability metric (or divergence) $\rho$, with $\Delta$ being referred to as the \emph{separation boundary} or \emph{contiguity radius}. Of course, the interest is in letting $\Delta\rightarrow 0$ as $n \rightarrow \infty$ (i.e., shrinking alternatives) and analyzing $R_\Delta$ for a given test, $\phi$, i.e., whether $R_\Delta(\phi)\rightarrow 0$. In the asymptotic setting, the \emph{minimax separation} or \emph{critical radius} $\Delta^*$ is the fastest possible order at which $\Delta\rightarrow 0$ such that $\lim\inf_{n\rightarrow\infty}\inf_{\phi\in\Phi_\alpha}R_{\Delta^*}(\phi)\rightarrow 0$, i.e., for any $\Delta$ such that $\Delta/\Delta^*\rightarrow\infty$, there is no test $\phi\in\Phi_\alpha$ that is consistent over $\mathcal{P}_\Delta$. A test is \emph{asymptotically minimax optimal} if it is consistent over $\mathcal{P}_\Delta$ with $\Delta \asymp\Delta^*$. On the other hand, in the non-asymptotic setting, the minimax separation $\Delta^*$ is defined as the minimum possible separation, $\Delta$ such that $\inf_{\phi\in\Phi_{n,\alpha}}R_{\Delta}(\phi)\le\delta$, for $0<\delta<1-\alpha$. A test $\phi\in\Phi_{n,\alpha}$ is called \emph{minimax optimal} if $R_\Delta(\phi)\le \delta$ for some $\Delta\asymp\Delta^*$. In other words, there is no other $\alpha$-level test that can achieve the same power with a better separation boundary. 

\cite{Krishna} consider $\mathcal{P}_\Delta$ as
\begin{equation}
\mathcal{P}_\Delta=\left\{P:\frac{dP}{dP_0}-1\in \mathcal{F}(\nu;M), \ \chi^2(P,P_0)=\norm{\frac{dP}{dP_0}-1}^2_{L^2(P_0)} \geq \Delta \right\},
\label{Eq:alternate-krishna}
\end{equation}
where $\nu>0$, and \begin{align*}\mathcal{F}(\nu;M)&:=\left\{f\in L^2(P_0):\,\text{for any } R>0,\,\exists f_R\in\h\,\text{such that }\Vert f_R\Vert_\h\le R,\,\right.\nonumber\\
&\left.\qquad\qquad\qquad\text{and}\,\Vert f-f_R\Vert_{L^2(P_0)}\le MR^{-1/\nu}\right\}.\end{align*}
$\mathcal{P}_\Delta$ in \eqref{Eq:alternate-krishna} denotes the class of alternatives that are $\Delta$-separated from $P_0$ in the $\chi^2$-divergence---alternately, the squared $L^2(P_0)$ norm of the likelihood ratio, $dP/dP_0-1$ is lower bounded by $\Delta$---, while satisfying a smoothness condition. The smoothness condition is imposed on the likelihood ratio and is defined through the rate of approximation of a function in $L^2(P_0)$ by an element in an RKHS ball of radius $R$. The faster the approximation rate---controlled by $\nu$---, the smoother the function being approximated. 
$\mathcal{F}(\nu;M)$ is a subspace of a real interpolation space obtained by interpolating $\h$ and $L^2(P_0)$. Particularly, $\nu=0$ corresponds to an RKHS ball of radius $R$. Note that \eqref{Eq:alternate-krishna} requires $P \ll P_0$ (i.e., $P$ is absolutely continuous w.r.t.~$P_0$) so that the Radon-Nikodym derivative $dP/dP_0$ is well defined. Define
\begin{equation}
\tilde{\PP}_{\Delta}:= \left\{P: \frac{dP}{dP_0}-1 \in \range (L_K^{\frac{1}{2\nu+2}}), 
 \ \chi^2(P,P_0)=\norm{\frac{dP}{dP_0}-1}^2_{L^2(P_0)} \geq \Delta\right\}, \label{Eq:alternative-theta1}
\end{equation}
where $L_K:L^2(P_0)\rightarrow L^2(P_0),\,f\mapsto \int_\mathcal{X} K(\cdot,x)f(x)\,dP_0(x)$ is an integral operator defined by $K$, and $\text{Ran}(A)$ denotes the range space of $A$. It follows from \cite[Theorem 4.1]{Cucker} that $$\range (L_K^{\frac{1}{2\nu+2}})\subset \mathcal{F}\left(\nu;2^{\frac{2\nu+2}{\nu}}\Vert L_K^{-\frac{1}{2\nu+2}}\left(dP/dP_0-1\right)\Vert^{\frac{2\nu+2}{\nu}}_{L^2(P_0)}\right),$$ and if $P_0$ is non-degenerate, then
$$\mathcal{F}(\nu;M)\subset \range (L_K^{\frac{1}{2\nu+2}-\epsilon}),\,\forall\,\epsilon>0,\,\,\text{i.e.,}\,\,\mathcal{F}(\nu;M)\subset \range (L_K^\eta),\,\,\forall\,0\le\eta <\frac{1}{2\nu+2}.$$

In this work, we employ an operator theoretic perspective to the goodness-of-fit test problem involving $\T$ (see Section~\ref{Sec:spec} for details), which is a centered version of the integral operator $L_K$. The centered version is needed to do away with the assumption of $\mu_0=0$, which is assumed in \citealp{Krishna}. Therefore, we choose $\mathcal{P}_\Delta$ similar to the form in \eqref{Eq:alternative-theta1} but with $L_K$ replaced by $\T$. We write it as

\begin{equation}
\PP:=\PP_{\theta,\Delta}:= \left\{P: \frac{dP}{dP_0}-1 \in \range (\T^{\theta}), 
 \ \chi^2(P,P_0)=\norm{\frac{dP}{dP_0}-1}^2_{L^2(P_0)} \geq \Delta\right\}, \label{Eq:alternative-theta}
\end{equation} for $\theta > 0$. Note that $\theta$ and $\nu$ are in inverse proportion to each other and $\theta=\frac{1}{2}$ yields $\range(\T^\theta)=\tilde{\h}$, with $0<\theta<\frac{1}{2}$ yielding interpolation spaces and $\theta=0$ corresponds to $L^2(P_0)$, where $\tilde{\h}$ is the RKHS induced by the centered kernel, $$\bar{K}(x,y)= \inner{K(\cdot,x)-\mu_0}{K(\cdot,y)-\mu_0}_{\h}.$$ The explicit representation of $\range(\T^\theta)$ typically relies on the kernel $K$ and the distribution $P_0$. If the kernel $K$ has a Mercer decomposition with respect to eigenfunctions that constitute an orthonormal basis for $\Lp$, then $\range(\T^\theta)$ comprises functions within the span of these orthonormal basis functions. For instance, in the following examples, we present an explicit representation of $\range(\T^\theta)$ when $P_0$ is a uniform distribution on (i) $[0,1]$, (ii) $\mathbb{S}^2$, a unit sphere, and (iii) when $P_0$ is a standard Gaussian distribution on $\mathbb{R}$. 
In this context, $\range(\T^\theta)$ can be expressed in terms of Fourier basis in Example 1, spherical harmonic basis in Example 2 and Hermite polynomials basis in Example 3.
\begin{example}[Uniform distribution on {$[0,1]$}] \label{Ex: uniform}
 Let $P_0$ be the uniform distribution defined on $[0,1]$ with \begin{equation}
     K(x,y)= a_0+ \sum_{k \neq 0} |k|^{-\beta} e^{\sqrt{-1}2\pi kx}e^{-\sqrt{-1}2\pi ky},\,\,a_0 \geq 0,\,\,\beta>1. \label{eq:kernel-fourier}\end{equation}  \\Then 
 $$\emph{\range}(\T^\theta) = \left\{ \sum_{k \neq 0} a_k e^{\sqrt{-1}2 \pi k x } : \sum_{k \neq 0} a_k^2 k^{2\theta \beta} < \infty \right\}.$$\\ \vspace{2mm} Note that the $s$-order Sobolev space defined on $[0,1]$ is given by
    $$\mathcal{W}^{s,2}:=\left\{f(x)=\sum_{k\in\mathbb{Z}} a_k e^{\sqrt{-1}2\pi kx},\,x\in[0,1]: \sum_{k\in\mathbb{Z}}(1+k^2)^sa^2_k<\infty\right\}.$$ Since  $\sum_{k\neq 0} k^{2\theta \beta}a_k^2 \leq \sum_{k \in\mathbb{Z}} (1+k^2)^{\theta \beta}a_k^2$, it follows that $\mathcal{W}^{s,2}\subset \emph{\range}(\T^\theta)$. This means, if $u:= \frac{dP}{dP_0}-1 \in \mathcal{W}^{s,2}$, then $u \in \emph{\range}(\T^\theta)$ with $\theta=\frac{s}{\beta}.$ 
  An example of a kernel that follows the form in \eqref{eq:kernel-fourier} is the periodic spline kernel, represented as $\tilde{K}(x,y)=\frac{(-1)^{r -1}}{(2r) !}B_{2r}([x-y]),$ where $B_r$ denotes the Bernoulli polynomial, which is generated by the generating function $\frac{te^{tx}}{e^t-1}=\sum_{r=0}^{\infty}B_r(x)\frac{t^r}{r!},$ and $[t]$ denotes the fractional part of $t$. Then using the formula $B_{2r}(x)=\frac{(-1)^{r-1}(2r)!}{(2\pi)^{2r}}\sum_{k\neq 0}^{\infty} |k|^{-2r} e^{\sqrt{-1}2\pi kx},$ it can be demonstrated that $\tilde{K}(x,y) = \sum_{k \neq 0}(2\pi |k|)^{-2r} e^{\sqrt{-1}2\pi kx}e^{-\sqrt{-1}2\pi ky}$ (see \citealp[page 21]{Wahba1990} for details).  
\end{example}
\begin{example}[Uniform distribution on $\mathbb{S}^2$]\label{Ex:uniform2}
    Let $P_0$ be a uniform distribution on $\X=\mathbb{S}^{2},$ where $\mathbb{S}^2$ denotes the unit sphere. Let 
    \begin{equation}
    K(x,y):= \sum_{k=1}^{\infty}\sum_{j=-k}^{k}\lambda_k Y^j_k(\theta_x,\phi_x)Y^j_k(\theta_y,\phi_y). \label{Eq:kernel-sphere}    
    \end{equation}
     where 
    $x =(\sin \theta_x \cos \phi_x,\sin \theta_x \sin\phi_x,\cos \theta_x)$, $y =(\sin \theta_y \cos \phi_y,\sin \theta_y \sin\phi_y,\cos \theta_y)$ with $0<\theta_x,\theta_y<\pi$, $0<\phi_x,\phi_y<2\pi$, and $$Y^j_{k}(\theta,\phi):= \sqrt{\frac{(2k+1)(k-j)!}{4\pi (k+j)!}}p^j_k(\cos \theta)e^{\sqrt{-1}j\phi},$$ with $p^j_k(x) := (-1)^j (1-x^2)^{\frac{j}{2}} \frac{d^j p_k(x)}{dx^j}$ and $p_k(x):= \frac{1}{k!2^k}\frac{d^k(x^2-1)^k}{dx^k}$. Here $(Y^j_k(\theta,\phi))_{j,k}$ denote the spherical harmonics which form an orthonormal basis in $L^2(\mathbb{S}^2)$. If $\sum_{k=1}^{\infty}(2k+1) \lambda_k < \infty,$ then     
    $$ \emph{\range}(\T^\theta) =\left\{\sum_{k=1}^{\infty}\sum_{j=-k}^{k}a_{kj}Y_k^j(\theta_x,\phi_x) : \sum_{k=1}^{\infty}\sum_{j=-k}^{k} a_{kj}^2 \lambda_k^{-2\theta}< \infty\right\}.$$
    Many common kernels take the form in \eqref{Eq:kernel-sphere}. For example, \citet[Theorem 2 and 3]{unibound} provide explicit expressions for the eigenvalues corresponding to Gaussian and polynomial kernels on the sphere. 

\end{example}

\begin{example}[Gaussian distribution with Mehler kernel on $\mathbb{R}$] \label{Ex: Gaussian}
  Let $P_0$ be a standard Gaussian distribution on $\mathbb{R}$ and $K$ be the Mehler kernel, i.e., $$K(x,y):=\frac{1}{\sqrt{1-\rho^2}}\exp\left(-\frac{\rho^2(x^2+y^2)-2 \rho xy}{2(1-\rho^2)}\right),$$ for $0<\rho<1$. Then 
 $$\emph{Ran} (\T^{\theta})=\left\{\sum_{k=1}^{\infty}a_{k}\gamma_k(x) : \sum_{k=1}^{\infty}a_k^2e^{-2k\theta\log \rho} < \infty \right\},$$
    where 
    $\gamma_k(x) = \frac{H_k(x)}{\sqrt{k!}} ,$ and $H_k(x) = (-1)^ke^{x^2/2}\frac{d^k}{dx^k}e^{-x^2/2}$ is the Hermite polynomial. \citet[Theorem 4.6]{Steinwart2012MercersTO} provides an interpretation of $\emph{Ran}(\mathcal{T}^\theta)$ as a real interpolation of $L^2(P_0)$ and $\tilde{\mathscr{H}}$. Therefore, with the kernel being fixed, the influence of $P_0$ on $\emph{Ran}(\mathcal{T}^\theta)$ can be understood as follows. Suppose $P_{0,\sigma_i}:=N(0,\sigma^2_i)$, $i=1,2$. It is easy to verify that $L^2(P_{0,\sigma_1})\subset L^2(P_{0,\sigma_2})$ if $\sigma_2<\sigma_1$, which implies that $\emph{Ran}(\mathcal{T}_1^\theta)\subseteq \emph{Ran}(\mathcal{T}^\theta_2)$, where $\mathcal{T}_i$ is the integral operator defined w.r.t.~$P_{0,\sigma_i}$, $i=1,2$. Based on this intuition, in the context of this example, choosing $P_0$ as a Gaussian distribution with variance larger (\emph{resp.} smaller) than 1 yields a smaller (\emph{resp.} larger) range space than that mentioned above.

\end{example}

With this background, we now present our contributions.

\subsection{Contributions}\label{subsec:contributions}
The main contributions of the paper are as follows:\vspace{1.5mm}\\
\emph{(i)}  First, in Theorem~\ref{thm: MMD}, we show that the test based on $\hat{D}^2_{\text{MMD}}$ (we refer to it as the MMD test) cannot achieve a separation boundary better than $n^{\frac{-2\theta}{2\theta+1}}$ w.r.t.~$\mathcal{P}$ defined in \eqref{Eq:alternative-theta}. This is an extension and generalization of \cite[Theorem 1]{Krishna}, which only shows such a claim for $\theta=\frac{1}{2}$ in an asymptotic setting, assuming $\mu_0=0$ and the \emph{uniform boundedness condition}, $\sup_i \Vert \phi_i\Vert_\infty<\infty,$ where $(\phi_i)_i$ are the eigenfunctions of $\T.$ In contrast, Theorem~\ref{thm: MMD} presents the result both by assuming and not assuming the uniform boundedness condition.
Note that the uniform boundedness condition $\sup_i\norm{\phi_i}_{\infty} < \infty$ is not satisfied in many scenarios (of course, it is satisfied in Example~\ref{Ex: uniform}). For example, as illustrated in \citet[Theorem 5]{unibound}, for $\X=\mathbb{S}^{d-1}$, representing the $d$-dimensional unit sphere, $\sup_i\norm{\phi_i}_{\infty}= \infty$ for all $d\geq3$ when using any kernel of the form $K(x,y)=f(\langle x,y\rangle_2)$, where $x,y\in \X$ and $f$ is continuous (see Example~\ref{Ex:uniform2}). The Gaussian kernel on $\mathbb{S}^{d-1}$ serves as an instance of such a kernel. Moreover, the condition $\mu_0=0$ is not satisfied by any characteristic kernel on general domain $\mathcal{X}$ and therefore excludes popular kernels such as Gaussian, Mat\'{e}rn, and inverse multiquadric on $\mathbb{R}^d$. Relaxing these two assumptions allows a large class of $(K,P_0)$ pairs to be handled by Theorem~\ref{thm: MMD}.
\vspace{1.5mm}\\
\emph{(ii)} Note that the separation boundary of the MMD test depends only on the smoothness of $dP/dP_0-1$, which is determined by $\theta$ but is completely oblivious to the \emph{intrinsic dimensionality} of the RKHS, $\tilde{\mathscr{H}}$, which is controlled by the decay rate of the eigenvalues of $\mathcal{T}$. To this end, by taking into account the intrinsic dimensionality of $\tilde{\mathscr{H}}$, we show in Theorem~\ref{thm:minimax} that the minimax separation w.r.t.~$\mathcal{P}$ is $n^{-\frac{4\theta\beta}{4\theta\beta+1}}$ for $\theta>\frac{1}{2}$ if $\lambda_i\asymp i^{-\beta}$, $\beta>1$, i.e., the eigenvalues of $\mathcal{T}$ decay at a polynomial rate $\beta$, and is $\sqrt{\log n}/n$ if $\lambda_i\asymp e^{-i}$, i.e., exponential decay. These results clearly establish the non-optimality of the MMD-based test. Theorem~\ref{thm:minimax}, which is non-asymptotic, generalizes the asymptotic version of \citep[Theorem 4] {Krishna} without requiring the uniform boundedness condition and also recovers it under the uniform boundedness condition, while not requiring $\mu_0=0$ for both these results. Moreover, even under the uniform boundedness condition, while \citep[Theorem 4] {Krishna} provides a bound on the minimax separation for $\frac{1}{2}>\theta>\frac{1}{2\beta}$, we improve this range in Theorem~\ref{thm:minimax} by showing the 
 minimax separation for $\theta>\frac{1}{4\beta}$. 
\vspace{1.5mm}\\
\emph{(iii)} In Section~\ref{Sec:spec}, we employ an operator theoretic perspective to the regularization idea presented in \cite{Krishna} that allows us to generalize the idea to handle general spectral regularizers, without requiring $\mu_0=0$. More precisely, we propose a statistic of the form $\eta_\lambda(P,P_0):=\norm{\gSL(\mu_P-\mu_{P_0})}_{\h}^2$, which when $g_{\lambda}(x)=(x+\lambda)^{-1}$, i.e., Tikhonov regularization, and $\mu_0=0$ reduces to the regularized statistic in \eqref{eq:regularized_test}. Here $\Sigma_0$ corresponds to the centered covariance operator w.r.t.~$P_0$. 

Assuming $\mu_0$ and $\Sigma_0$ are computable, we propose a spectral regularized test based on $\eta_\lambda$ and provide sufficient conditions on $g_{\lambda}$ for the test to be minimax optimal w.r.t.~$\mathcal{P}$ (see Theorems~\ref{thm:typI-oracle}, \ref{thm:typII-oracle} and Corollaries~\ref{coro:poly-oracle}, \ref{coro:exp-oracle}). Compared to the results in \citep{Krishna}, we provide general sufficient conditions on the separation boundary for any bounded kernel and show the minimax optimality in the non-asymptotic setting for a wider range of $\theta$, both with and without the uniform boundedness condition (see Theorem~\ref{thm:typII-oracle}). However, the drawback of the test is that one needs first to compute the eigenvalues and eigenfunctions of $\Sigma_0$ which is not possible for many $(K,P_0)$ pairs. Thus we refer to this test as the \emph{Oracle test.}\vspace{1.5mm}\\
\emph{(iv)} To address the shortcomings with the Oracle test, in Section~\ref{subsec:test-statistic}, we assume that $P_0$ is samplable, i.e., $P_0$ can be sampled to generate new samples. Based on these samples, we propose a test statistic defined in \eqref{eq:twosamplestat} that involves using the estimators of $\mu_0$ and $\Sigma_0$ in $\eta_\lambda$. We show that such a test statistic can be computed only through matrix operations and by solving a finite-dimensional eigensystem (see Theorem \ref{thm: computation}). We present two approaches to compute the critical level of the test. In Section~\ref{subsec:srct}, we compute the critical level based on a  concentration inequality and refer to the corresponding test as \emph{spectral regularized concentration test} (SRCT), and in Section~\ref{subsec:srpt}, we employ permutation testing (e.g., \citealt{lehmann}, \citealt{Pesarin}, \citealt{permutations}), which we refer to as the \emph{spectral regularized permutation test} (SRPT), leading to a critical level that is easy to compute (see Theorems~\ref{thm:typI-Gamma} and \ref{thm: permutations typeI}). We show that both these tests are minimax optimal w.r.t.~$\mathcal{P}$ (see Theorems~\ref{thm:typII-Gamma} and \ref{thm: permutations typeII}). Note that under these additional samples from $P_0$, a goodness-of-fit test can be seen as a two-sample test, and therefore SRCT and SRPT can be interpreted as two-sample tests. Recently, \cite{twosampletest} developed a \emph{spectral regularized kernel two-sample test} (SR2T) and showed it to be minimax optimal for a suitable class of alternatives. In this work, we show that SRCT and SRPT have better separation rates than those of SR2T for the range of $\theta$, where all these tests are not minimax.
\vspace{1.5mm}\\
\emph{(v)} The minimax optimal separation rate in the proposed tests (SRCT and SRPT) is tightly controlled by the choice of the regularization parameter, $\lambda$, which in turn depends on the unknown parameters, $\theta$ and $\beta$ (in the case of the polynomial decay of the eigenvalues of $\mathcal{T}$). Therefore, 
to make these tests completely data-driven, in Section~\ref{subsec:adaptation}, we present an adaptive version of both tests by aggregating tests over different $\lambda$ (see Theorems~\ref{thm: adp-gamma typeI} and \ref{thm:perm adp typeI}) and show the resulting tests to be minimax optimal up to a $\sqrt{\log n}$ factor in case of the SRCT (see Theorem~\ref{thm: adp-gamma TypeII}) and $\log\log n$ factor in case of SRPT (see Theorem~\ref{thm:perm adp TypeII}). In contrast, \cite[Theorem 5]{Krishna} considers an adaptive and asymptotic version of the Oracle test under $\mu_0=0$ and the uniform boundedness condition, where it only adapts over $\theta$ assuming $\beta$ is known, with $\beta$ being the polynomial decay rate of the eigenvalues of $\mathcal{T}$.
\vspace{1.5mm}\\
\emph{(vi)} Through numerical simulations on benchmark data, in Section~\ref{sec:experiments}, we demonstrate the superior performance of the proposed spectral regularized tests in comparison to the MMD test based on $\hat{D}_{\mathrm{MMD}}(P,P_0)$, Energy test \citep{Energy} based on the energy distance, Kolmogorov-Smirnov (KS) test \citep{KS,Fasano}, and SR2T.
\subsection{Comparison to \citet{twosampletest}} \label{subsection:comparision}
As mentioned in \emph{(iv)} of Section~\ref{subsec:contributions}, the proposed goodness-of-fit tests (SRCT and SRPT) can be seen as two-sample tests because of access to additional samples from $P_0$. Similar to the two-sample test SR2T proposed in \cite{twosampletest}, these tests also employ the spectral regularization approach of \cite{twosampletest} and their analysis uses many technical results developed in \cite{twosampletest}. Therefore, to emphasize the conceptual and technical novelty of our work, in this section, we compare and contrast our setup and results to that of \cite{twosampletest}. \vspace{1.25mm}\\
\emph{(i)} \textbf{Alternative space}: In this paper, the alternative space, $\mathcal{P}_\Delta$ shown in \eqref{Eq:alternative-theta} involves a smoothness condition that is defined with respect to the function $u:=\frac{dP}{dP_0}-1$. In contrast, the smoothness condition in \citet{twosampletest} was defined through $\frac{dP}{dR}-1$, where $R=\frac{P+P_0}{2}$. The separation boundary in this paper is measured in $\chi^2$-distance, i.e., $\chi^2(P,P_0)$ compared to the Hellinger distance between $P$ and $P_0$ (which is topologically equivalent to $\chi^2(P,R)$) as in \cite{twosampletest}. Since the $\chi^2$-divergence dominates the Hellinger distance, the notion of separation considered in this paper is stronger than the one considered in \cite{twosampletest}. 

These changes were made to leverage the knowledge of $P_0$ in the goodness-of-fit problem (which is not available in the two-sample problem), resulting in a better separation boundary than that achieved by the test proposed in \citet{twosampletest}.\vspace{1mm}\\
 \emph{(ii)} \textbf{Estimation of the covariance operator, $\Sigma_0$}: In \cite{twosampletest}, the covariance operator $\Sigma_{0}$ is defined with respect to the average probability measure $R:=\frac{P+P_0}{2}$, which means two sets of samples are required to estimate it and therefore, the estimation error is controlled by the minimum of sizes of two sets of samples. However, in this paper, since we are considering a goodness-of-fit problem where the null distribution $P_0$ is known, we can utilize this knowledge by defining the covariance operator $\Sigma_0$ with respect to $P_0$, which means the estimation error is controlled only by the samples from $P_0$. Since we do not have any budget constraints on sampling from $P_0$, the estimation error of $\Sigma_0$ can be controlled at a desired level for a large enough sample size. Therefore, we investigate the required number of i.i.d.~samples $s$ to be drawn from $P_0$ to estimate $\Sigma_{0}$ to achieve a similar separation boundary as the oracle test, which assumes $\Sigma_{0}$ is exactly known in closed form. Both in this work and \cite{twosampletest}, while the separation rates are determined by the minimum size of the two sets of samples, since the sample size associated with $P_0$ in this work can be chosen to be large enough, the separation rate will be controlled only by the one sample size. Therefore this work yields better separation rates than those in \cite{twosampletest}---also see \emph{(iv)} in Section~\ref{subsec:contributions}---as it should be since a goodness-fit-testing problem is simpler than a two-sample testing problem.\vspace{1mm}\\
   \emph{(iii)} \textbf{Spectral regularized concentration test (SRCT)}: While SRPT proposed in this paper shares in principle the similar ideas of permutation testing as in \cite{twosampletest}, the proposed SRCT involves a concentration inequality based test threshold that was not considered in \citet{twosampletest}. While the analysis of SRPT uses multiple technical results developed in \cite{twosampletest}---of course, with some deviations because of a different alternate space and estimator for the covariance operator---the analysis of SRCT requires establishing new technical results for the estimation error bounds between $\Sigma_0$ and $\hat{\Sigma}_0$ (see Lemmas~\ref{lemma:adaptation Upper} and~\ref{lemma:adaptation lower}), and the estimation error between $\Ntlh$ and $\Ntl$, where $\Ntl := \Vert\SgL\Sigma_{0}\SgL\Vert_{\hs}$ and $\Sigma_{0,\lambda}:=\Sigma_0+\lambda \Id$ (see Lemmas~\ref{lemma: ntlh upper bound} and~\ref{lemma: ntlh lower bound}).

\section{Definitions \& Notation}
For a topological space $\X$, $L^r(\X,\mu)$ denotes the Banach space of $r$-power $(r\geq 1)$ $\mu$-integrable function, where $\mu$ is a finite non-negative Borel measure on $\X$. For $f \in L^r(\X,\mu)=:L^r(\mu)$, $\norm{f}_{L^r(\mu)}:=(\int_{\X}|f|^r\,d\mu)^{1/r}$ denotes the $L^r$-norm of $f$. $\mu^n := \mu \times \stackrel{n}{...} \times \mu$ is the $n$-fold product measure. $\h$ denotes a reproducing kernel Hilbert space with a reproducing kernel $K: \X \times \X \to \R$. $[f]_{\sim}$ denotes the equivalence class associated with $f\in L^r(\X,\mu)$, that is the collection of functions $g \in L^r(\X,\mu)$ such that $\norm{f-g}_{L^r(\mu)}=0$. For two measures $P$ and $Q$, $P \ll Q$ denotes that $P$ is dominated by $Q$ which means, if $Q(A)=0$ for some measurable set $A$, then $P(A)=0$. Let $H_1$ and $H_2$ be abstract Hilbert spaces. $\EuScript{L}(H_1,H_2)$ denotes the space of bounded linear operators from $H_1$ to $H_2$. For $S \in \EuScript{L}(H_1,H_2)$, $S^*$ denotes the adjoint of $S$. $S \in \EuScript{L}(H) := \EuScript{L}(H,H)$ is called self-adjoint if $S^*=S$. For $S \in \EuScript{L}(H)$, $\text{Tr}(S)$, $\norm{S}_{\EuScript{L}^2(H)}$, and $\norm{S}_{\EuScript{L}^{\infty}(H)}$ denote the trace, Hilbert-Schmidt and operator norms of $S$, respectively. For $x,y \in H$, $x \otimes_{H} y$ is an element of the tensor product space of $H \otimes H$ which can also be seen as an operator from $H \to H$ as $(x \otimes_{H}y)z=x\inner{y}{z}_{H}$ for any $z \in H$.

For constants $a$ and $b$, $a \lesssim b$ (resp. $a \gtrsim b$) denotes that there exists a positive constant $c$ (\emph{resp.} $c'$) such that $a\leq cb$ (\emph{resp.} $a \geq c' b)$. $a \asymp b$ denotes that there exists positive constants $c$ and $c'$ such that  $cb \leq a \leq c' b$. We denote $[\ell]$ for $\{1,\ldots,\ell\}$.

\section{Non-optimality of $\hat{D}^2_{\text{MMD}}$ test}\label{Sec:non-optimal}
Assuming $\mu_0=0$, \citep{Krishna} established the non-optimality of the MMD-based goodness-of-fit test. In this section, we extend this result in two directions by not assuming $\mu_0=0$ and by considering the setting of non-asymptotic minimax compared to the asymptotic minimax setting of \cite{Krishna}. The key to achieving these extensions is an operator representation of $D^2_{\text{MMD}}$, which we obtain below. To this end, we make the following assumption throughout the paper.\vspace{1.5mm}\\

$(A_0)$  $(\mathcal{X},\mathcal{B})$ is a second countable (i.e., completely separable) space endowed with Borel $\sigma$-algebra $\mathcal{B}$. $(\h,K)$ is an RKHS of real-valued functions on $\X$ with a continuous reproducing kernel $K$ satisfying
$\sup_{x} K(x,x) \leq \K.$ \vspace{1.5mm}\\

The continuity of $K$ ensures that $K(\cdot,x):\X \to \h$ is Bochner-measurable for all $x \in \X$, which along with the boundedness of $K$ ensures that $\mu_P$ and $\mu_{P_0}$ are well-defined \citep{Dincu}. Also, the separability of $\mathcal{X}$ along with the continuity of $K$ ensures that $\h$ is separable (\citealt[Lemma 4.33]{svm}). Therefore, 
\begin{align}
D^2_{\mathrm{MMD}}(P,P_0) &= \norm{\mu_P-\mu_{P_0}}^2_{\h} 
= \inner{\int_{\X}K(\cdot,x)\, d(P-P_0)(x)}{\int_{\X}K(\cdot,x)\, d(P-P_0)(x)}_{\h}\nonumber\\
&= \inner{\int_{\X}K(\cdot,x)u(x)\, d\PQ(x)}{\int_{\X}K(\cdot,x)u(x)\, d\PQ(x)}_{\h},\label{eq:mmd}
\end{align}
where $u=\frac{dP}{dP_0}-1$. As done in \citep{twosampletest}, by defining $\id : \h \to \Lp$, $f \mapsto [f - \E_{\PQ}f]_{\sim}$, 
where $\E_{\PQ}f=\int_{\X}f(x)\,d\PQ(x)$, it follows from \citep[Proposition C.2]{kpca} that $\id^* : \Lp \to \h$, $f \mapsto \int K(\cdot,x)f(x)\,d\PQ(x)-\mu_{\PQ} \E_{\PQ}f$. Also, it follows from \citep[Proposition C.2]{kpca} that $\T= \Upsilon- (1\ltens1)\Upsilon-\Upsilon(1\ltens1)+(1\ltens1)\Upsilon(1\ltens1)$, where $\Upsilon: \Lp \to \Lp$, $f \mapsto \int K(\cdot,x)f(x)\,d\PQ(x)$ and $\T:=\id \id^* : \Lp \to \Lp$. Note that $\T$ is a trace class operator, and thus compact since $K$ is bounded. Also, $\T$ is self-adjoint and positive, and therefore spectral theorem \citep[Theorems VI.16, VI.17]{Reed} yields that
$$\T = \sum_{i \in I} \lambda_i \Tilde{\phi_i} \ltens \Tilde{\phi_i},$$
where $(\lambda_i)_i \subset \R^+ $ are the eigenvalues and $(\Tilde{\phi}_i)_i$ are the orthonormal system of eigenfunctions (strictly speaking classes of eigenfunctions) of $\T$ that span $\overline{\range(\T)}$ with the index set $I$ being either countable in which case $\lambda_i \to 0$ or finite. In this paper, we assume that the set $I$ is countable, i.e., infinitely many eigenvalues.
Since $\Tilde{\phi_i}$ represents an equivalence class in $\Lp$, by defining $\phi_i:= \frac{\id^* \Tilde{\phi_i}}{\lambda_i}$, it is clear that $\id\phi_i=[\phi_i-\E_{\PQ}\phi_i]_{\sim}=\Tilde{\phi_i}$ and $\phi_i \in \h$. Throughout the paper, $\phi_i$ refers to this definition.

Using these definitions, it follows from \eqref{eq:mmd} that
\begin{align}
D^2_{\mathrm{MMD}}(P,P_0)&=  \inner{\id^*u}{\id^*u}_{\h}
=\inner{\T u}{u}_{\Lp} = \sum_{i\geq 1} \lambda_i \langle u,\Tilde{\phi_i}\rangle^2_{\Lp}.\nonumber
\end{align}
The above expression was already obtained by \citep[p. 6]{Krishna} but  through Mercer's representation of $K$. Here we obtain it alternately through the operator representation, which will turn out to be crucial for the rest of the paper. This representation also highlights the limitation of $D_{\mathrm{MMD}}$ that $D_\mathrm{MMD}$ might not capture the difference between between $P$ and $P_0$ if they differ in the higher Fourier coefficients of $u$, i.e., $\langle u,\tilde{\phi}_i\rangle_{L^2(P_0)}$ for large $i$, since $(\lambda_i)_i$ is a decreasing sequence. 

On the other hand, $\chi^2(P||P_0)=\norm{u}_{\Lp}^2=\sum_{i\geq1}\langle u,\Tilde{\phi_i}\rangle^2_{\Lp}$ if $u\in \text{span}\{\tilde{\phi}_i:i\in I\}$, does not suffer from such an issue. 

The following result shows that the test based on $\hat{D}_{\mathrm{MMD}}^2$ cannot achieve a separation boundary of order better than $n^{\frac{-2\theta}{2\theta+1}}$.

\begin{thm}[Separation boundary of MMD test] \label{thm: MMD}
Let $n\geq 2$ and $$\sup_{P\in \PP}\norm{\T^{-\theta}u}_{\Lp}<\infty.$$  Then for any $\alpha>0$, $\delta>0$, $P_{H_0}\{\hat{D}_{\mathrm{MMD}}^2 \geq \gamma\} \leq \alpha,$
$$\inf_{P\in \PP}P_{H_1}\{\hat{D}_{\mathrm{MMD}}^2  \geq \gamma\} \geq 1-\delta,$$
where $\gamma = \frac{4\kappa}{\sqrt{\alpha}n},$ 
$$\Delta_{n}:=\Delta=c(\alpha,\delta)n^{\frac{-2\theta}{2\theta+1}},$$ and 
$c(\alpha,\delta)\asymp\max\{\alpha^{-1/2},\delta^{-1}\}.$
Furthermore if $\Delta_{n} < d_{\alpha}n^{\frac{-2\theta}{2\theta+1}}$ for some  $d_{\alpha}>0$ and one of the following holds: (i) $\theta \geq \frac{1}{2}$, (ii) $\sup_{i}\norm{\phi_i}_{\infty} < \infty$, $\theta > 0$, then for any decay rate of $(\lambda_i)_i$, there exists $\tilde{k}_{\delta}$ such that for all $n \geq \tilde{k}_{\delta}$,  

$$\inf_{P\in \PP}P_{H_1}\{\hat{D}_{\mathrm{MMD}}^2  \geq \gamma\} < \delta.$$

\end{thm}

\begin{rem} 
Note that the above theorem also holds asymptotically if the testing threshold $\gamma$ is chosen as the $(1-\alpha)$-quantile of the asymptotic distribution of $\hat{D}_{\mathrm{MMD}}^2$ under $H_0$, thereby extending \cite[Theorem 1]{Krishna}, which only considers $\theta=\frac{1}{2}$ but assuming $\mu_0=0$. In fact, the result holds for any threshold that converges in probability to such an asymptotic quantile. 
\end{rem}

By providing the minimax separation rate w.r.t.~$\mathcal{P}$, the following result demonstrates the non-optimality of the MMD test presented in Theorem~\ref{thm: MMD}.
\begin{thm}[Minimax separation boundary] \label{thm:minimax}
If $\lambda_i \asymp i^{-\beta}$, $\beta>1$, then 
there exists $c(\alpha,\delta)$ such that if $$\Delta_{n}  \leq  c(\alpha,\delta) n^{\frac{-4\theta\beta}{4\theta\beta+1}},\,\,0 \leq \delta \leq 1-\alpha,$$ then $R^*_{\Delta_{n}}:= \inf_{\phi \in \Phi_{n,\alpha}} R_{\Delta_{N,M}}(\phi)\geq \delta,$ provided 
one the following holds: (i) $\theta \geq \frac{1}{2}$, (ii) $\sup_{i}\norm{\phi_i}_{\infty} < \infty$, $\theta \geq \frac{1}{4\beta},$ where $R_{\Delta_{n}}(\phi):=\sup_{P\in\PP}\E_{P^n }[1-\phi].$

Suppose $\lambda_i \asymp e^{-\tau i}$, $\tau>0$, $\theta > 0$. Then there exists $c(\alpha,\delta,\theta)$ and $k$ such that if $$\Delta_{n}  \leq  c(\alpha,\delta,\theta)\frac{\sqrt{\log n}}{n},$$ and $n\geq k$, then for any $0 \leq \delta \leq 1-\alpha,$ $R^*_{\Delta_{n}} \geq \delta.$ 

\end{thm}

\begin{rem}
(i) Since $\inf_{\beta>1} \frac{4\theta\beta}{4\theta\beta+1}=\frac{4\theta}{4\theta+1}>\frac{2\theta}{2\theta+1}$ and $1>\frac{2\theta}{2\theta+1}$ for any $\theta>0$, it follows that the separation boundary of MMD is larger than the minimax separation boundary w.r.t.~$\mathcal{P}$ irrespective of the decay rate of the eigenvalues of $\mathcal{T}$.\vspace{1mm}\\
(ii) In the setting of polynomial decay, Theorem~\ref{thm:minimax} generalizes \cite[Theorem 4]{Krishna} in two ways: (a) When the uniform boundedness condition holds, the range of $\theta$ for which the minimax separation rate holds is extended from $\frac{1}{2\beta}<\theta<\frac{1}{2}$ to $\theta>\frac{1}{4\beta}$, and (b) minimax separation is also obtained without assuming the uniform boundedness condition.\vspace{1mm}\\
(iii) The uniform boundedness condition, $\sup_i\norm{\phi_i}_{\infty} < \infty$ does not hold in general. For example,  the Gaussian kernel on $S^{d-1}$, $d\ge 3$, does not satisfy the uniform boundedness condition \citep[Theorem 5]{unibound}, while the Gaussian kernel on $\mathbb{R}^d$ for any $d$ satisfies the uniform boundedness condition~\citep{Steinwart2006AnED}. 

In this paper, we provide results both with and without the uniform boundedness condition to understand its impact on the behavior of the test. Such a condition has also been used in the analysis of the impact of regularization in kernel learning (see \citealt[p. 531]{Mendelson-10}).
\end{rem}

\section{Spectral regularized MMD test}\label{Sec:spec}
In this section, we propose a spectral regularized version of the MMD test and show it to be minimax optimal w.r.t.~$\mathcal{P}$. The proposed test statistic is based on the \emph{spectral regularized discrepancy}, which is defined as
\begin{equation}\eta_{\lambda}(P,P_0) : = \inner{\T\gl(\T)u}{u}_{\Lp},\label{Eq:spec-reg}\end{equation}
where $u=\frac{dP}{dP_0}-1$, $g_\lambda:(0,\infty)\rightarrow (0,\infty)$ is a \emph{spectral regularizer} that satisfies  $\lim_{\lambda\rightarrow 0} xg_\lambda(x)\asymp 1$ (more concrete assumptions on $\gl$ will be introduced later), and 
$$g_{\lambda}(\B) : = \sum_{i\geq 1} g_{\lambda}(\tau_i) (\psi_i \otimes_H \psi_i) + g_{\lambda}(0)\left(\Id - \sum_{i\geq 1} \psi_i \otimes_H \psi_i\right),$$
with $\mathcal{B}$ being any compact, self-adjoint operator defined on a separable Hilbert space, $H$. Here $(\tau_i,\psi_i)_i$ are the eigenvalues and eigenfunctions of $\B$ which enjoys the spectral representation, $\B= \sum_{i}\tau_i \psi_i \otimes_H \psi_i$. The well known \emph{Tikhonov regularizer}, $(\mathcal{B}+\lambda \Id)^{-1}$, which is used in \cite{Krishna}, is obtained as a special case by choosing $g_\lambda(x)=(x+\lambda)^{-1}$.

The key idea in proposing $\eta_\lambda$ is based on the intuition that $\T g_\lambda(\T)\approx \Id$ for sufficiently small $\lambda$ so that $\eta_\lambda(P,P_0)\approx \Vert u\Vert^2_{L^2(P_0)}$, and therefore does not suffer from the limitation of $D^2_{\textrm{MMD}}(P,P_0)$ as aforementioned in Section~\ref{Sec:non-optimal} (see Lemma \ref{lemma: bounds for eta}). 

Using $\eta_\lambda$, in the following, we present details about the construction of the test statistic and the test. First, we provide an alternate representation for $\eta_\lambda$ which is useful to construct the test statistic. Define the \emph{covariance operator},
\begin{align}
 \Sigma_{0}:=\Sigma_{P_0}&=\int_{\X} (K(\cdot,x)-\mu_{P_0}) \htens (K(\cdot,x)-\mu_{P_0})\, dP_0(x)\nonumber\\
 &=\frac{1}{2} \int_{\X \times \X}(K(\cdot,x)-K(\cdot,y))\htens(K(\cdot,x)-K(\cdot,y))\,d\PQ(x)\,d\PQ(y),\nonumber
\end{align}
 which is a positive, self-adjoint, trace-class operator. It can be shown \citep[Proposition C.2]{kpca} that $\Sigma_{0}=\id^*\id:\mathscr{H}\rightarrow\mathscr{H}$. Using this representation in \eqref{Eq:spec-reg} yields
\begin{align}
\eta_\lambda(P,P_0)&=\langle \T g_\lambda(\T)u,u\rangle_{\Lp}=\langle \id\id^*g_\lambda(\id\id^*)u,u\rangle_{\Lp}\stackrel{(\dagger)}{=}\langle \id g_\lambda(\id^*\id)\id^*u,u\rangle_{\Lp}\nonumber\\
&=\langle g_\lambda(\Sigma_{0})\id^*u,\id^*u\rangle_\mathscr{H}=\langle g_\lambda(\Sigma_{0}) (\mu_P-\mu_{P_0}),\mu_P-\mu_{P_0}\rangle_\mathscr{H}\nonumber\\
&=\norm{\gSL(\mu_P-\mu_{P_0})}_{\h}^2,\label{eq:cov-rep}
\end{align}
where $(\dagger)$ follows from (\cite{twosampletest}, Lemma A.8(i) by replacing $\Sigma_{PQ}$ by $\Sigma_0$) that $\T\gl(\T) = \id \gl(\Sigma_{0}) \id^*$. Throughout the paper, we assume that $\gl$ satisfies the following:
\begin{align*}
    &(A_1) \quad \sup_{x \in \Gamma}|x \gl(x)| \leq C_1, &&(A_2) \quad \sup_{x \in \Gamma}|\lambda\gl(x)| \leq C_2, \\ 
    &(A_3) \quad \sup_{\{x\in \Gamma: x\gl(x)<B_3\}} |B_3-x\gl(x)|x^{2\varphi} \leq C_3 \lambda^{2\varphi},
    & &(A_4) \quad \inf_{x \in \Gamma} \gl(x)(x+\lambda) \geq C_4,
\end{align*}
where $\Gamma : = [0,\K]$, $\varphi \in (0,\xi]$ and the constant $\xi$ is called the \emph{qualification} of $\gl$. $C_1$, $C_2$, $C_3$, $B_3$ and $C_4$ are finite positive constants (all independent of $\lambda>0$). Note that these assumptions are quite standard in the inverse problem literature (see e.g., \citealp{Engl.et.al}) and spectral regularized kernel ridge regression \citep{BAUER200752}, except for $(A_3)$, which is replaced by a stronger version---the stronger version of $(A_3)$ takes supremum over whole $\Gamma$. Recently, however, in a two-sample testing scenario, \cite[Section 4.2]{twosampletest} use $(A_3)$. 
The less restrictive assumption $(A_3)$ implies that higher qualifications are possible for the same function $\gl$ in the testing problem compared to the known qualifications in the literature of inverse problems and spectral regularized kernel ridge regression. 
For instance, consider the function $\gl(x)=\frac{1}{x+\lambda}$ corresponding to Tikhonov regularization. In this case, the stronger condition used in literature $\sup_{x\in \Gamma} |1-x\gl(x)|x^{2\varphi} \leq C_3 \lambda^{2\varphi}$ is satisfied only for $\varphi\in(0,\frac{1}{2}]$. However, $(A_3)$ holds for any $\varphi>0$, indicating infinite qualification with no saturation at $\varphi=\frac{1}{2}$ with $B_3 = \frac{1}{2}$ and $C_3=1$.

Define $\Sigma_{0,\lambda} := \Sigma_{0}+\lambda \Id$, 
$$\Nol := \text{Tr}(\SgL\Sigma_{0}\SgL),\,\,\,\text{and}\,\,\,
\Ntl := \norm{\SgL\Sigma_{0}\SgL}_{\hs},$$
which capture the intrinsic dimensionality (or degrees of freedom) of $\h$, with $\Nol$ being quite heavily used in the analysis of kernel ridge regression (e.g., \citealt{Caponnetto-07}). Based on these preliminaries, in the following, we present an Oracle goodness-of-fit test.
\subsection{Oracle Test}
Using the samples $(X_i)_{i=1}^n$, a $U$-statistic estimator of $\eta_\lambda$ defined in \eqref{eq:cov-rep} can be written as
\begin{equation}
\hat{\eta}_{\lambda}= \frac{1}{n(n-1)}\sum_{i\neq j}\inner{\gSL(K(\cdot,X_i)-\mu_0)}{\gSL(K(\cdot,X_j)-\mu_0)}_{\h},    \nonumber
\end{equation}
which when $\mu_0=0$ and $g_\lambda(x)=(x+\lambda)^{-1}$ reduces to the moderated MMD statistic proposed in \cite{Krishna}. The following result provides an $\alpha$-level test based on $\hat{\eta}_\lambda$.
\begin{thm}[Critical region--Oracle]\label{thm:typI-oracle}
Let $n \geq 2$. Suppose $(A_0)$--$(A_2)$ hold. Then for any $\alpha >0$ and any $\lambda>0$,
$$P_{H_0}\{\hat{\eta}_{\lambda} \geq \gamma\} \leq \alpha,$$
where $\gamma = \frac{2\Cs\Ntl}{n\sqrt{\alpha}}.$
\end{thm}
Unfortunately, the test is not practical as the critical value, $\gamma$, and the test statistic depend on the eigenvalues and eigenfunctions of $\Sigma_0$, which are not easy to compute for many $(K, P_0)$ pairs. Therefore, we call the above test the \emph{Oracle test}. The following result analyzes the power of the Oracle test and presents sufficient conditions on the separation boundary to achieve the desired power.

\begin{thm}[Separation boundary--Oracle]\label{thm:typII-oracle}
Suppose $(A_0)$--$(A_3)$. Let $$\sup_{P \in \PP} \norm{\T^{-\theta}u}_{\Lp} < \infty,$$ $\norm{\Sigma_{0}}_{\op} \geq\lambda=d_{\theta}\Delta_{n}^{\frac{1}{2\Tilde{\theta}}}$, $d_{\theta}>0$, where $d_\theta$ is a constant that depends on $\theta$. For any $0<\delta\le 1$, if $\Delta_{n}$ satisfies 
$$\frac{\Delta_{n}^{\frac{2\Tilde{\theta}+1}{2\tilde{\theta}}}}{\mathcal{N}_{2}\left(d_{\theta}\Delta_{n}^{\frac{1}{2\tilde{\theta}}}\right)} \gtrsim \frac{d_{\theta}^{-1}\delta^{-2}}{ n^{2}},\qquad\,
\frac{\Delta_n}{\mathcal{N}_{2}\left(d_{\theta}\Delta_{n}^{\frac{1}{2\tilde{\theta}}}\right)} \gtrsim\frac{(\alpha^{-1/2}+\delta^{-1})}{n}.$$
then 
\begin{equation}\inf_{P \in \PP} P_{H_1}\left\{\hat{\eta}_{\lambda}
\geq \gamma \right\} \geq 1-\delta,\label{Eq:type-2-oracle}\end{equation}
where $\gamma = \frac{2\Cs\Ntl}{n\sqrt{\alpha}}$, and $\Tilde{\theta}=\min(\theta,\xi)$. Furthermore, suppose  $C:=\sup_{i}\norm{\phi_i}_{\infty} < \infty$. Then \eqref{Eq:type-2-oracle} holds when the above  conditions on $\Delta_{n}$ are replaced by
$$\frac{\Delta_{n}}{\mathcal{N}_{1}\left(d_{\theta}\Delta_{n}^{\frac{1}{2\tilde{\theta}}}\right)} \gtrsim \frac{\delta^{-2}}{n^{2}},\qquad\frac{\Delta_{n}}{\mathcal{N}_{2}\left(d_{\theta}\Delta_{n}^{\frac{1}{2\tilde{\theta}}}\right)} \gtrsim \frac{(\alpha^{-1/2}+\delta^{-1})}{ n}.$$

\end{thm}
The following corollaries to Theorem~\ref{thm:typII-oracle} investigate the separation boundary of the test under the polynomial and exponential decay condition on the eigenvalues
of $\Sigma_{0}$.

\begin{cor} [Polynomial decay--Oracle]\label{coro:poly-oracle}
Suppose $\lambda_i \asymp i^{-\beta}$, $\beta>1$. Then for any $\delta>0$, 
$$\inf_{P \in \PP} P_{H_1}\left\{\hat{\eta}_{\lambda}
\geq \gamma \right\} \geq 1-\delta,$$
when
$$\Delta_{n} =
\left\{
	\begin{array}{ll}		c(\alpha,\delta)n^{\frac{-4\tilde{\theta}\beta}{4\Tilde{\theta}\beta+1}},  &  \ \  \Tilde{\theta}> \frac{1}{2}-\frac{1}{4\beta} \\
		c(\alpha,\delta)n^{-\frac{8\Tilde{\theta}\beta}{4\Tilde{\theta}\beta+2\beta+1}}, & \ \  \Tilde{\theta} \le \frac{1}{2}-\frac{1}{4\beta}
	\end{array}
\right.,$$
with $c(\alpha,\delta)\gtrsim(\alpha^{-1/2}+\delta^{-2})$. Furthermore, if $\sup_{i}\norm{\phi_i}_{\infty} < \infty$, then 
$$\Delta_n = c(\alpha,\delta)n^{\frac{-4\tilde{\theta}\beta}{4\Tilde{\theta}\beta+1}}.$$
\end{cor}

\begin{cor}[Exponential decay--Oracle]\label{coro:exp-oracle}
Suppose $\lambda_i \asymp e^{-\tau i}$, $\tau>0$. Then for any $\delta>0$, there exists $k_{\alpha,\delta}$ such that for all $n\geq k_{\alpha,\delta}$, 

$$\inf_{P \in \PP} P_{H_1}\left\{\hat{\eta}_{\lambda}
\geq \gamma \right\} \geq 1-\delta,$$
when
$$\Delta_{n} =
\left\{
	\begin{array}{ll}
		c(\alpha,\delta, \theta)\frac{\sqrt{\log n}}{n}, &  \ \ \Tilde{\theta}> \frac{1}{2} \\
		c(\alpha,\delta,\theta)\left(\frac{\sqrt{\log n}}{n}\right)^{\frac{4\tilde{\theta}}{2\tilde{\theta}+1}}, & \ \ \Tilde{\theta} \le \frac{1}{2}
	\end{array}
\right.,$$
where $c(\alpha,\delta,\theta)\gtrsim \max\left\{\sqrt{\frac{1}{2\tilde{\theta}}},1\right\}(\alpha^{-1/2}+\delta^{-2})$. Furthermore, if $\sup_{i}\norm{\phi_i}_{\infty} < \infty$, then 

$$\Delta_{n} = c(\alpha,\delta,\theta)\frac{\sqrt{\log n}}{n},$$
where $c(\alpha,\delta,\theta) \gtrsim \max\left\{\sqrt{\frac{1}{2\tilde{\theta}}},\frac{1}{2\tilde{\theta}},1\right\}(\alpha^{-1/2}+\delta^{-2})$.
\end{cor}
\begin{rem}
    (i) Observe that larger qualification $\xi$ (defined in Assumption $(A_3)$ in Section 4) corresponds to a smaller separation boundary. Therefore, it is important to work with regularizers with infinite qualification, such as Tikhonov and Showalter. It has to be 
    noted that the Tikhonov regularizer has infinite qualification as per $(A_3)$ but has a qualification of $\frac{1}{2}$ w.r.t.~the stronger version of $(A_3)$.    \vspace{1mm}\\
    (ii) Suppose $g_\lambda$ has infinite qualification, $\xi=\infty$, then $\tilde{\theta}=\theta$. Comparing Corollary~\ref{coro:poly-oracle} (\emph{resp.} Corollary~\ref{coro:exp-oracle}) and Theorem~\ref{thm:minimax} shows that the spectral regularized test based on $\hat{\eta}_{\lambda}$ is minimax optimal w.r.t.~$\mathcal{P}$ in the ranges of $\theta$ as given in Theorem~\ref{thm:minimax} if the eigenvalues of $\T$ decay polynomially (\emph{resp.} exponentially). 

    Outside these ranges of $\theta$, the optimality of the test remains an open question since we do not have a minimax separation covering these ranges of $\theta$.\vspace{1mm}\\
    (iii) Corollary~\ref{coro:poly-oracle} recovers the minimax separation rate in \cite{Krishna} under the uniform boundedness condition but without assuming $\mu_0=0$. Furthermore, it also presents the separation rate for the regularized MMD test without assuming both the uniform boundedness condition and $\mu_0=0$, and shows a phase transition in the separation rate depending on the value of $\tilde{\theta}$.
\label{rem:oracle}
\end{rem}

\subsection{Two-sample statistic} \label{subsec:test-statistic}
The Oracle test statistic requires the knowledge of $\mu_0$ and $\Sigma_0$ for it to be computable. Though $P_0$ is known, $\Sigma_0$ and $\mu_0$ are not known in closed form in general for many $(K,P_0)$ pairs. To address this issue, in this section, we assume that $P_0$ is samplable, i.e., a set of i.i.d.~samples from $P_0$ are available or can be generated. To this end, let us say $m+s$ i.i.d.~samples are available from $P_0$ of which $(Y^0_i)_{i=1}^s \stackrel{i.i.d}{\sim} P_0$ are used to estimate $\Sigma_0$ and $(X^0_i)_{i=1}^m \stackrel{i.i.d}{\sim} P_0$ are used to estimate $\mu_0$, with $(Y^0_i)_{i=1}^s\stackrel{i.i.d.}{=} (X^0_i)_{i=1}^m$. Note that we do not use all $m+s$ samples to estimate both $\mu_0$ and $\Sigma_0$. Instead, we do sample splitting so that the estimators of $\Sigma_0$ and $\mu_0$ are decoupled, which will turn out to be critical for the analysis. Based on this, $\eta_\lambda$ can be estimated as a two-sample $U$-statistic \citep{Hoeffding} as
\begin{equation}
\stat:=\frac{1}{n(n-1)}\frac{1}{m(m-1)}\sum_{1\leq i\neq j \leq n}\sum_{1\leq i'\neq j' \leq N} h(X_i,X_j,X^0_{i'},X^0_{j'}),  \label{eq:twosamplestat}  
\end{equation}
where 
$$h(X_i,X_j,X^0_{i'},X^0_{j'}):= \inner{\gShh(K(\cdot,X_i)-K(\cdot,X^0_{i'}))}{\gShh(K(\cdot,X_j)-K(\cdot,X^0_{j'}))}_{\h},$$
and
\begin{equation*}
\hat{\Sigma}_0 :=\frac{1}{2s(s-1)}\sum_{i\neq j}^{s} (K(\cdot,Y_i^0)-K(\cdot,Y_j^0)) \htens (K(\cdot,Y_i^0)-K(\cdot,Y_j^0)),
\end{equation*}
is a one-sample $U$-statistic estimator of $\Sigma_{0}$ based on $(Y^0_i)_{i=1}^s$. Note that $\hat{\eta}^{TS}_\lambda$ is not exactly a $U$-statistic since it involves $\hat{\Sigma}_{0}$, but conditioned on $(Y^0_i)_{i=1}^s$, one can see that it is exactly a two-sample $U$-statistic. By expanding the inner product in $h$ and writing \eqref{eq:twosamplestat} as 
\begin{eqnarray*}
    \stat&{}={}& \frac{1}{n(n-1)}\sum_{i\neq j}\inner{ g_{\lambda}^{1/2}(\hat{\Sigma}_0)\kk(\cdot,X_i)}{g_{\lambda}^{1/2}(\hat{\Sigma}_0)\kk(\cdot,X_j)}_{\h} \\ 
    &&\qquad+ \frac{1}{m(m-1)}\sum_{i\neq j}\inner{ g_{\lambda}^{1/2}(\hat{\Sigma}_0)\kk(\cdot,X_i^0)}{g_{\lambda}^{1/2}(\hat{\Sigma}_0)\kk(\cdot,X_j^0)}_{\h}\\ 
    &&\qquad\qquad-\frac{2}{nm}\sum_{i,j}\inner{ g_{\lambda}^{1/2}(\hat{\Sigma}_0)\kk(\cdot,X_i)}{g_{\lambda}^{1/2}(\hat{\Sigma}_0)\kk(\cdot,X_j^0)}_{\h},
\end{eqnarray*}
the following result shows that $\stat$ can be computed only through matrix operations and by solving a finite-dimensional eigensystem.

\begin{thm} \label{thm: computation}
Let $(\hat{\lambda}_i, \hat{\alpha_i})_i$ be the eigensystem of $\frac{1}{s} \Tilde{\hh}_s^{1/2}K_s\Tilde{\hh}_s^{1/2}$ where $K_s:=[K(Y_i^0,Y_j^0)]_{i,j \in [s]}$, $\hh_s= \Id_s -\frac{1}{s}\one_s \one_s^\top$, and $\Tilde{\hh}_s=\frac{s}{s-1}\hh_s$. Define $$G := \sum_{i}\left( \frac{g_{\lambda}(\hat{\lambda}_i)-g_{\lambda}(0)}{\hat{\lambda}_i}\right)\hat{\alpha}_i\hat{\alpha}_i^\top.$$ Then
\begin{equation*}
\stat=\frac{1}{n(n-1)}\left(\circled{\emph{\small{1}}}-\circled{\emph{\small{2}}}\right)+\frac{1}{m(m-1)}\left(\circled{\emph{\small{3}}}-\circled{\emph{\small{4}}}\right) - \frac{2}{nm}\circled{\emph{\small{5}}}, 
\end{equation*}
where 
\begin{align*}
&\circled{\emph{\small{1}}}=\one_n^\top\left(g_{\lambda}(0)K_n+\frac{1}{s}K_{ns}\Tilde{\hh}^{1/2}_{s}G\Tilde{\hh}^{1/2}_{s}K_{ns}^\top\right)\one_n,\\
&\circled{\emph{\small{2}}}=\emph{Tr}\left(g_{\lambda}(0)K_n+\frac{1}{s}K_{ns}\Tilde{\hh}^{1/2}_{s}G\Tilde{\hh}^{1/2}_{s}K_{ns}^\top\right),\\
&\circled{\emph{\small{3}}}=\one_m^\top\left(g_{\lambda}(0)K_m+\frac{1}{s}K_{ms}\Tilde{\hh}^{1/2}_{s}G\Tilde{\hh}^{1/2}_{s}K_{ms}^\top\right)\one_m,\\
&\circled{\emph{\small{4}}}=\emph{Tr}\left(g_{\lambda}(0)K_m+\frac{1}{s}K_{ms}\Tilde{\hh}^{1/2}_{s}G\Tilde{\hh}^{1/2}_{s}K_{ms}^\top\right),\quad\text{and}\\
&\circled{\emph{\small{5}}}=\one_m^\top\left(g_{\lambda}(0)K_{mn}+\frac{1}{s}K_{ms}\Tilde{\hh}^{1/2}_{s}G\Tilde{\hh}^{1/2}_{s}K_{ns}^\top\right)\one_n,
\end{align*}
with $K_n:=[K(X_i,X_j)]_{i,j \in [n]}$, $K_m:=[K(X_i^0,X_j^0)]_{i,j \in [m]}$, $K_{ns}:=[K(X_i,Y_j^0)]_{i \in [n],j \in [s]}$,  
$K_{ms}:=[K(X_i^0,Y_j^0)]_{i \in [m],j \in [s]}$, and  $K_{mn}:=[K(X_i^0,X_j)]_{i \in [m],j \in [n]}$.
\end{thm}
Note that in the case of Tikhonov regularization, $G=\frac{-1}{\lambda}(\frac{1}{s}\Tilde{\hh}^{1/2}_sK_s\Tilde{\hh}^{1/2}_s+\lambda \Id_s)^{-1}$. The complexity of computing $\stat$ is given by $O(s^3+m^2+n^2+ns^2+ms^2)$. We would like to mention that since $\hat{\eta}^{TS}_{\lambda}$ is based on two sets of samples, a result very similar to Theorem~\ref{thm: computation} is presented in \cite[Theorem 3]{twosampletest} in the context of two-sample testing.
\subsection{Spectral regularized concentration test (SRCT)}\label{subsec:srct}
By applying Chebyshev inequality to $\stat$ under $H_0$, the following result provides an $\alpha$-level test, which we refer to as SRCT.

Define $\Ntlh : = \norm{\SgLh\hat{\Sigma}_0\SgLh}_{\hs}$.

\begin{thm} [Critical region--SRCT] \label{thm:typI-Gamma}
Let $n \geq 2$ and $m\geq 2$. Suppose $(A_0)$--$(A_2)$ hold. Then for any $\alpha >0,$ $c_1 \geq 65$ and 
$\frac{4c_1\K}{s}\max\{\log\frac{96\K s}{\alpha},\log\frac{12}{\alpha}\} \leq \lambda \leq \norm{\Sigma_0}_{\op}$,
$$P_{H_0}\left\{\stat \geq  \gamma \right\} \leq \alpha,$$
where $\gamma=\frac{12\Cs\Ntlh}{b_1\sqrt{\alpha}}\left( \frac{1}{n}+\frac{1}{m}\right)$, $b_1 = \sqrt{\frac{4}{9}-\frac{16}{3\sqrt{3c_1}}-\frac{32}{9c_1}}.$ 
Furthermore, if $C:=\sup_i\norm{\phi_i}_{\infty} < \infty$, the above bound holds for $4c_1C^2\Nol\log\frac{48\Nol}{\alpha}\leq s$.
\end{thm}
Note that unlike in the Oracle test, the threshold $\gamma$ and the test statistic $\stat$ in the above result is completely data-driven and computable, with $\Ntlh$ being computed based on $(\hat{\lambda}_i)_i$ from Theorem~\ref{thm: computation}. The following result provides sufficient conditions on the separation boundary to achieve the desired power.
 
\begin{thm} [Separation boundary-SRCT]\label{thm:typII-Gamma}
Suppose $(A_0)$--$(A_4)$ and $m\geq n$. Let $$\sup_{P \in \PP} \norm{\T^{-\theta}u}_{\Lp} < \infty,$$ $\norm{\Sigma_{0}}_{\op} \geq\lambda=d_{\theta}\Delta_{n}^{\frac{1}{2\Tilde{\theta}}}$, $d_{\theta}>0$, where $d_\theta$ is a constant that depends on $\theta$. For any $0<\delta\le 1$, if $s \geq 32c_1\K \lambda^{-1}\log(\max\{17920\K^2\lambda^{-1},6\}\delta^{-1})$ and $\Delta_{n}$ satisfies 
$$\frac{\Delta_{n}^{\frac{2\Tilde{\theta}+1}{2\tilde{\theta}}}}{\mathcal{N}_{2}\left(d_{\theta}\Delta_{n}^{\frac{1}{2\tilde{\theta}}}\right)} \gtrsim \frac{d_{\theta}^{-1}\delta^{-2}}{ n^{2}},\qquad\,
\frac{\Delta_n}{\sqrt{\mathcal{N}_{1}\left(d_{\theta}\Delta_{n}^{\frac{1}{2\tilde{\theta}}}\right)}} \gtrsim\frac{(\alpha^{-1/2}+\delta^{-1})}{n},$$
then
\begin{equation}\inf_{P \in \PP} P_{H_1}\left\{\stat
\geq \gamma \right\} \geq 1-4\delta,\label{Eq:type-2}\end{equation}
where $\gamma=\frac{12\Cs\Ntlh}{b_1\sqrt{\alpha}}\left( \frac{1}{n}+\frac{1}{m}\right)$, $b_1 = \sqrt{\frac{4}{9}-\frac{16}{3\sqrt{3c_1}}-\frac{32}{9c_1}},$ $c_1\geq 65$ and $\Tilde{\theta}=\min(\theta,\xi)$. 
Furthermore, suppose  $C:=\sup_{i}\norm{\phi_i}_{\infty} < \infty$. Then \eqref{Eq:type-2} holds if $s\geq 32c_1C^2\Nol\log\frac{32\Nol}{\delta}$ and when the above  conditions on $\Delta_{n}$ are replaced by
$$\frac{\Delta_{n}}{\mathcal{N}_{1}\left(d_{\theta}\Delta_{n}^{\frac{1}{2\tilde{\theta}}}\right)} \gtrsim \frac{\delta^{-2}}{n^{2}},\qquad\frac{\Delta_n}{\sqrt{\mathcal{N}_{1}\left(d_{\theta}\Delta_{n}^{\frac{1}{2\tilde{\theta}}}\right)}} \gtrsim\frac{(\alpha^{-1/2}+\delta^{-1})}{n}.$$
\end{thm}
\begin{rem}\label{rem:srct}
(i) Comparing the conditions on the separation boundary in Theorem \ref{thm:typII-Gamma} to those of Theorem \ref{thm:typII-oracle}, it is easy to verify that the claims in Corollaries \ref{coro:poly-oracle} and \ref{coro:exp-oracle} also hold for $\emph{SRCT}$. Therefore, $\emph{SRCT}$ achieves minimax optimality in the same ranges of $\theta$ as the Oracle test.\vspace{1mm}\\
(ii) In the case of polynomial decay, when $\tilde{\theta} > \frac{1}{2}- \frac{1}{4\beta}$, the condition on $s$---the number of samples needed to estimate the covariance operator $\Sigma_0$---reduces to $s \gtrsim n^{\frac{2\beta}{4\theta\beta+1}}\log n$, which is of sub-linear order and is implied if $s \gtrsim n \log n$. When $\tilde{\theta}\le\frac{1}{2}-\frac{1}{4\beta},$ the condition becomes $s\gtrsim n^{\frac{4\beta}{4\theta\beta+2\beta+1}}\log n$ which is implied for any $\theta$ and $\beta$ if $s \gtrsim n^2 \log n$. Furthermore, under uniform boundedness, the condition on $s$ becomes $s \gtrsim n^{\frac{2}{4\theta\beta+1}} \log n$ which is of sublinear order for $\theta > \frac{1}{4\beta}.$ In case of exponential decay, for $\theta> \frac{1}{2}$, the condition is $s \gtrsim n^{\frac{1}{2\theta}} (\log n)^{1-\frac{1}{4\theta}}$, which is implied for any $\theta >\frac{1}{2}$, if $s \gtrsim n \sqrt{\log n}$. For $\theta < \frac{1}{2}$, the condition is $s \gtrsim n^{\frac{2}{2\theta+1}} (\log n)^{\frac{2\theta}{2\theta+1}}$ which is implied if $s\gtrsim n^2.$ Furthermore, if $\sup_{i}\norm{\phi_i}_{\infty} < \infty$ holds, then the condition is $s \gtrsim (\log n) (\log\log n).$
\end{rem}
\subsection{Spectral regularized permutation test (SRPT)}\label{subsec:srpt}
Instead of using a concentration inequality-based test threshold as in SRCT, in this section, we study the permutation approach to compute the test threshold \citep{lehmann,Pesarin,permutations}. We refer to the resulting test as SRPT. We show that SRPT achieves a minimax optimal separation boundary with a better constant compared to that of SRCT.

Recall that our test statistic defined in Section~\ref{subsec:test-statistic} involves three sets of independent samples,  $(X_i)_{i=1}^n \stackrel{i.i.d.}{\sim} P$, $(X_j^0)_{j=1}^m \stackrel{i.i.d.}{\sim} P_0$, $(Y_i^0)_{i=1}^s \stackrel{i.i.d.}{\sim} P_0$. Define $(U_i)_{i=1}^n:=(X_i)_{i=1}^n$, and $(U_{n+j})_{j=1}^m:=(X_j^0)_{j=1}^m$. Let $\Pi_{n+m}$ be the set of all possible permutations of $\{1,\ldots,n+m\}$ with $\pi \in \Pi_{n+m}$ being a randomly selected permutation from the $D$ possible permutations, where $D :=|\Pi_{n+m}|= (n+m)!$. Define $(X^{\pi}_i)_{i=1}^n := (U_{\pi(i)})_{i=1}^n$ and $(X^{0\pi}_j)_{j=1}^m := (U_{\pi(n+j)})_{j=1}^m$. Let $\hat{\eta}^{\pi}_{\lambda}:=\stat(X^{\pi},X^{0\pi},Y^0)$ be the statistic based on the permuted samples, and $(\pi^i)_{i=1}^B$ be $B$ randomly selected permutations from $\Pi_{n+m}$. For simplicity, define $\hat{\eta}^i_{\lambda}:= \hat{\eta}^{\pi^i}_{\lambda}$ to represent the statistic based on permuted samples w.r.t.~the random permutation $\pi^i$. Given the samples $(X_i)_{i=1}^n$, $(X_j^0)_{j=1}^m$ and $(Y_i^0)_{i=1}^s$, define  $$F_{\lambda}(x):= \frac{1}{D}\sum_{\pi \in \Pi_{n+m}}\II(\hat{\eta}^{\pi}_{\lambda} \leq x)$$ to be the permutation distribution function, and define $$q_{1-\alpha}^{\lambda}:= \inf\{q \in \R: F_{\lambda}(q) \geq 1-\alpha\}.$$ Furthermore, we define the empirical permutation distribution function based on $B$ random permutations as
$$\hat{F}^{B}_{\lambda}(x):= \frac{1}{B+1}\sum_{i=0}^{B}\II(\hat{\eta}^{i}_{\lambda} \leq x),$$ where $\hat{\eta}_{\lambda}^0=\stat$ and define $$\hat{q}_{1-\alpha}^{B,\lambda}:= \inf\{q \in \R: \hat{F}^B_{\lambda}(q) \geq 1-\alpha\}.$$

Based on these notations, the following result presents an $\alpha$-level test with a completely data-driven critical level.

\begin{thm}[Critical region--SRPT]\label{thm: permutations typeI}

For any $0<\alpha\leq 1$, 
$P_{H_0}\{\stat \geq \hat{q}_{1-\alpha}^{B,\lambda} \} \leq \alpha.$
\end{thm}
It is well known that the permutation approach exactly controls the type-I error. This follows from the exchangeability of samples under $H_0$ and the definition of $\qq$. Next, similar to Theorem \ref{thm:typII-Gamma}, the following result provides general conditions under which the power can be controlled. 
\begin{thm}[Separation boundary--SRPT] \label{thm: permutations typeII}
Suppose $(A_0)$--$(A_4)$ hold. Let  $m \geq n,$ 
$$\sup_{P \in \PP} \norm{\T^{-\theta}u}_{\Lp} < \infty,$$ $\norm{\Sigma_{0}}_{\op} \geq\lambda=d_{\theta}\Delta_{n}^{\frac{1}{2\Tilde{\theta}}}$, for some  $d_{\theta}>0$, where $d_\theta$ is a constant that depends on $\theta$. For any $0<\delta\le 1$, if  $n \geq d_3\delta^{-1/2}\log\frac{1}{\alpha}$ for some $d_3>0$, $B\geq \frac{3}{{\alpha}^2}\left(\log2\delta^{-1}+\alpha(1-\alpha)\right)$, $s \geq 280\K \lambda^{-1}\log(17920\K^2\lambda^{-1}\delta^{-1})$ and $\Delta_{n}$ satisfies 
$$\frac{\Delta_{n}^{\frac{2\Tilde{\theta}+1}{2\tilde{\theta}}}}{\mathcal{N}_{2}\left(d_{\theta}\Delta_{n}^{\frac{1}{2\tilde{\theta}}}\right)} \gtrsim \frac{d_{\theta}^{-1}(\delta^{-1}\log(1/\tilde{\alpha}))^2}{ n)^{2}},\qquad\,
\frac{\Delta_{n}}{\mathcal{N}_{2}\left(d_{\theta}\Delta_{n}^{\frac{1}{2\tilde{\theta}}}\right)} \gtrsim\frac{\delta^{-1}\log(1/\tilde{\alpha})}{ n},$$
then 
\begin{equation}\inf_{P \in \PP} P_{H_1}\left\{\stat
\geq \hat{q}_{1-\alpha}^{B,\lambda} \right\} \geq 1-5\delta,\label{Eq:type-2-perm}\end{equation}
where $\Tilde{\theta}=\min(\theta,\xi)$. Furthermore, suppose  $C:=\sup_{i}\norm{\phi_i}_{\infty} < \infty$. Then \eqref{Eq:type-2-perm} holds if $s \geq 136C^2\Nol\log\frac{32\Nol}{\delta}$ and when the above  conditions on $\Delta_{n}$ are replaced by
$$\frac{\Delta_{n}}{\mathcal{N}_{1}\left(d_{\theta}\Delta_{n}^{\frac{1}{2\tilde{\theta}}}\right)} \gtrsim \frac{(\delta^{-1}\log(1/\tilde{\alpha}))^2}{n^{2}},\qquad\frac{\Delta_{n}}{\mathcal{N}_{2}\left(d_{\theta}\Delta_{n}^{\frac{1}{2\tilde{\theta}}}\right)} \gtrsim \frac{\delta^{-1}\log(1/\tilde{\alpha})}{n}.$$
\end{thm}
\begin{cor} \label{coro-perm:poly}
Suppose $\lambda_i \asymp i^{-\beta}$, $\beta>1$. Then for any $\delta>0$, 
$$\inf_{P \in \PP} P_{H_1}\left\{\stat
\geq \hat{q}_{1-\alpha}^{B,\lambda} \right\} \geq 1-5\delta,$$
when
$$\Delta_{n} =
\left\{
	\begin{array}{ll}
c(\alpha,\delta)n^{\frac{-4\tilde{\theta}\beta}{4\Tilde{\theta}\beta+1}},  &  \ \  \Tilde{\theta}> \frac{1}{2}-\frac{1}{4\beta} \\
		c(\alpha,\delta)n^{-\frac{8\Tilde{\theta}\beta}{4\Tilde{\theta}\beta+2\beta+1}}, & \ \  \Tilde{\theta} \le \frac{1}{2}-\frac{1}{4\beta}
	\end{array}
\right.,$$
with $c(\alpha,\delta)\gtrsim \delta^{-2}(\log \frac{1}{\alpha})^2$. Furthermore, if $\sup_{i}\norm{\phi_i}_{\infty} < \infty$, then 
$$\Delta_n = c(\alpha,\delta)n^{\frac{-4\tilde{\theta}\beta}{4\Tilde{\theta}\beta+1}}.$$
\end{cor}

\begin{cor}\label{coro-perm:exp}
Suppose $\lambda_i \asymp e^{-\tau i}$, $\tau>0$. Then for any $\delta>0$, there exists $k_{\alpha,\delta}$ such that for all $n \geq k_{\alpha,\delta}$, 

$$\inf_{(P,Q) \in \PP} P_{H_1}\left\{\stat
\geq \hat{q}_{1-\alpha}^{B,\lambda} \right\} \geq 1-2\delta,$$
when
$$\Delta_{n} =
\left\{
	\begin{array}{ll}
		c(\alpha,\delta, \theta)\frac{\sqrt{\log n}}{n}, &  \ \ \Tilde{\theta}> \frac{1}{2} \\
		c(\alpha,\delta,\theta)\left(\frac{\sqrt{\log n}}{n}\right)^{\frac{4\tilde{\theta}}{2\tilde{\theta}+1}}, & \ \ \Tilde{\theta} \le \frac{1}{2}
	\end{array}
\right.,$$
where $c(\alpha,\delta,\theta)\gtrsim \max\left\{\sqrt{\frac{1}{2\tilde{\theta}}},1\right\}\delta^{-2}(\log \frac{1}{\alpha})^2$. Furthermore, if $\sup_{i}\norm{\phi_i}_{\infty} < \infty$, then 

$$\Delta_{n} = c(\alpha,\delta,\theta)\frac{\sqrt{\log n}}{n},$$
where $c(\alpha,\delta,\theta) \gtrsim \max\left\{\sqrt{\frac{1}{2\tilde{\theta}}},\frac{1}{2\tilde{\theta}},1\right\}\delta^{-2}(\log \frac{1}{\alpha})^2$.
\end{cor}
The above results demonstrate the minimax optimality w.r.t.~$\mathcal{P}$ of the permutation-based test constructed in 
Theorem~\ref{thm: permutations typeI}. Since the conditions on $s$ in Theorem~\ref{thm: permutations typeII} match those of Theorem~\ref{thm:typII-Gamma}, the discussion in Remark~\ref{rem:srct}(ii) also applies for SRPT.
\begin{rem}\label{rem:two-sample}
Recently, \cite{twosampletest} proposed a spectral regularized two-sample test ( \emph{SR2T}) where the test statistic has a close resemblance to $\stat$. Since we are solving a goodness-of-fit test question as a two-sample test, one could simply address it using \emph{SR2T}, and therefore one may wonder about the need for the proposal of \emph{SRCT} and \emph{SRPT}, given their similarity to \emph{SR2T}. While this is a valid question, comparing Corollaries~\ref{coro-perm:poly} and \ref{coro-perm:exp} to that of \citep[Corollaries 6, 7]{twosampletest}, we observe that while all these tests enjoy \emph{minimax} separation rates over the same range of $\tilde{\theta}$, for the range of $\tilde{\theta}$ where the minimaxity of separation rate is not established, the proposed tests have faster convergence rate than that of \emph{SR2T}, thereby demonstrating the advantage of the proposed tests over \emph{SR2T} (see Section \ref{subsection:comparision} for details).
\end{rem}

\subsection{Adaptation}\label{subsec:adaptation}
In the previous sections, we have discussed two ways of constructing a test based on the statistic $\stat$. 

In both these tests, the optimal $\lambda$ to achieve the minimax separation boundary depends on unknown $\theta$ and $\beta$. In this section, we construct a test based on the union (aggregation) of multiple tests constructed for different values of $\lambda$ taking values in a finite set, $\Lambda$. It turns out that the resultant test is guaranteed to be minimax optimal (up to $\log$ factors) for a wide range of $\theta$ (and $\beta$ in the case of polynomially decaying eigenvalues). The aggregation method is quite classical and we employ the technique as used in \cite{twosampletest}.

Define $\Lambda :=\{\lambda_L, 2\lambda_L, ... \,, \lambda_U\},$ where $\lambda_U=2^b\lambda_L$, for $b \in \N$ so that $|\Lambda|=b+1=1+\log_2\frac{\lambda_U}{\lambda_L}$, where $|\Lambda|$ is the cardinality of $\Lambda$. 

Let $\lambda^*$ be the optimal $\lambda$ that yields minimax optimality. The key idea is to choose $\lambda_L$ and $\lambda_U$ such that there is an element in $\Lambda$ that is close to $\lambda^*$ for any $\theta$ (and $\beta$ in the case of polynomially decaying eigenvalues). Define $v^* := \sup\{x \in \Lambda: x \leq \lambda^*\}$. Then it is easy to verify that $v^* \asymp \lambda^*$, i.e., $v^*$ is also an optimal choice for $\lambda$ that belongs to $\Lambda$, since for $\lambda_L\leq \lambda^* \leq \lambda_U$, we have $\frac{\lambda^*}{2}\leq v^* \leq \lambda^*$. Motivated by this, in Theorems~\ref{thm: adp-gamma typeI} and \ref{thm:perm adp typeI}, we construct $\alpha$-level tests that are adaptive versions of SRCT and SRPT, based on the union of corresponding tests over $\lambda \in \Lambda$ that rejects $H_0$ if one of the tests rejects $H_0$. The separation boundary of these tests are analyzed in Theorems~\ref{thm: adp-gamma TypeII} and \ref{thm:perm adp TypeII} under the polynomial and exponential decay rates of the eigenvalues of $\T$. These results show that the adaptive versions achieve the same performance (up to $\log$ factors) as that of the Oracle test, i.e., minimax optimal w.r.t.~$\mathcal{P}$ over the range of $\theta$ mentioned in Theorem~\ref{thm:minimax}, without requiring the knowledge of $\lambda^*$. In contrast, \cite[Theorem 5]{Krishna} considers an adaptive and asymptotic version of the Oracle test under $\mu_0=0$ and the uniform boundedness condition, where it only adapts over $\theta$ assuming $\beta$ is known.

The following results relate to the adaptive version of SRCT.
\begin{thm} [Critical region--adaptation--SRCT] \label{thm: adp-gamma typeI}
Suppose $(A_0)$--$(A_2)$. Then for any $\alpha>0$, 
$\frac{32c_1\K}{s}\max\{\log\frac{96\K s}{\Tilde{\alpha}},\log\frac{12}{\Tilde{\alpha}}\} \leq \lambda_L \leq \lambda_U \leq \norm{\Sigma_0}_{\op}$, where $\Tilde{\alpha} = \frac{\alpha}{|\Lambda|}$, $c_1 \geq 65$.
$$P_{H_0}\left\{\sup_{\lambda \in \Lambda}\frac{\stat}{\Ntlh} \geq \gamma \right\} \leq \alpha,$$
where $\gamma=\frac{12\Cs}{b_1\sqrt{\Tilde{\alpha}}}\left( \frac{1}{n}+\frac{1}{m}\right)$, $b_1 = \sqrt{\frac{4}{9}-\frac{16}{3\sqrt{3c_1}}-\frac{32}{9c_1}}.$ 
Furthermore if $C:=\sup_i\norm{\phi_i}_{\infty} < \infty$, the above bound holds for $4c_1C^2\mathcal{N}_{1}(\lambda_L)\log\frac{48\mathcal{N}_{1}(\lambda_L)}{\delta}\leq s$.
\end{thm}

\begin{thm} [Separation boundary--adaptation--SRCT] \label{thm: adp-gamma TypeII}
Suppose 
$(A_0)$--$(A_4)$ hold, $\Tilde{\theta}=\min(\theta,\xi)$, $\Tilde{\xi}=\max(\xi,\frac{1}{4}),$  $\sup_{\theta>0}\sup_{P \in \PP} \norm{\T^{-\theta}u}_{\Lp} < \infty$, $\norm{\Sigma_0}_{\op}\geq \lambda_U$, $\theta_l>0,$ $m \geq n$, and $0<\alpha\leq1$. Then for any $\delta >0$ and $\gamma$ defined as in Theorem \ref{thm: adp-gamma typeI},
$$\inf_{\theta \geq \theta_l}\inf_{P \in \PP} P_{H_1}\left\{\sup_{\lambda \in \Lambda} \frac{\stat}{\Ntlh} \geq \gamma \right\} \geq 1-4\delta,$$
provided one of the following cases holds:
\begin{enumerate}[label=(\roman*)]

\item $\lambda_i \asymp i^{-\beta},$ $1<\beta \leq \beta_U$, $\lambda_L =r_1 n^\frac{-4\beta_U}{1+2\beta_U}$, $\lambda_U=r_2\left(\frac{n}{\sqrt{\log n}} \right)^\frac{-2}{4\Tilde{\xi}+1}$ for $r_1,r_2>0$,   
$s \geq 32c_1\K \lambda_L^{-1}\log(\max\{17920\K^2\lambda_L^{-1},6\}\delta^{-1})$, and $$\Delta_n=c(\alpha,\delta)\max\left\{ n^\frac{-8\Tilde{\theta}\beta}{1+2\beta+4\Tilde{\theta}\beta}, \left( \frac{n}{\sqrt{\log n}}\right)^\frac{-4\Tilde{\theta}\beta}{4\Tilde{\theta}\beta+1}\right\}.$$ 

Furthermore if $C:=\sup_i \norm{\phi_i}^2_{\infty} < \infty$, then the above conditions on $\lambda_L,$ $\lambda_U$ and $s$ can be replaced by $\lambda_L =r_3 \left(\frac{n}{\sqrt{\log n}} \right)^\frac{-2\beta_U}{4\theta_l\beta_U+1}$, $\lambda_U=r_4\left(\frac{n}{\sqrt{\log n}} \right)^\frac{-2}{4\xi+1}$ for $r_3,r_4>0$,

$s\geq 32c_1C^2\mathcal{N}_{1}(\lambda_L)\log\frac{32\mathcal{N}_{1}(\lambda_L)}{\delta}$, and
$$\Delta_n=c(\alpha,\delta)\left( \frac{n}{\sqrt{\log n}}\right)^\frac{-4\Tilde{\theta}\beta}{4\Tilde{\theta}\beta+1},$$ 
where $c(\alpha,\delta)\gtrsim(\alpha^{-1/2}+\delta^{-2}).$

\item  $\lambda_i \asymp e^{-\tau i}$, $\tau>0$, $\lambda_L=r_5 \left(\frac{n}{\sqrt{\log n}}\right)^{-2},$ $ \lambda_U =r_6 \left(\frac{n}{\log n}\right)^{-1/2\tilde{\xi}}$ for some $r_5,r_6>0,$ $s \geq 32c_1\K \lambda_L^{-1}\log(\max\{17920\K^2\lambda_L^{-1},6\}\delta^{-1})$, and $$\Delta_{n} \gtrsim 
c(\alpha,\delta,\theta)\max\left\{ \left(\frac{n}{\log n}\right)^{-1},\left(\frac{ n}{\sqrt{\log n}}\right)^{-\frac{4\tilde{\theta}}{2\tilde{\theta}+1}} \right\},$$
where $c(\alpha,\delta,\theta)\gtrsim \max\left\{\sqrt{\frac{1}{2\tilde{\theta}}},1\right\}(\alpha^{-1/2}+\delta^{-2}).$ 
Furthermore if $C:=\sup_i \norm{\phi_i}^2_{\infty} < \infty$, then the above conditions on $\lambda_L,$ $\lambda_U$ and $s$ can be replaced by $\lambda_L =r_7\left(\frac{n}{\log \log n}\right)^{-1/2\theta_l},$ $\lambda_U=r_8\left(\frac{n}{\log \log n}\right)^{-1/2\xi}$ for some $r_7,r_8>0$, 
$s\geq 32c_1C^2\mathcal{N}_{1}(\lambda_L)\log\frac{32\mathcal{N}_{1}(\lambda_L)}{\delta}$ and 
 $$\Delta_{n} = c(\alpha,\delta,\theta)\frac{\log n}{n},$$
where $c(\alpha,\delta,\theta) \gtrsim \max\left\{\sqrt{\frac{1}{2\tilde{\theta}}},\frac{1}{2\tilde{\theta}},1\right\}(\alpha^{-1/2}+\delta^{-2})$.  
\end{enumerate}
\end{thm}

The following results handle the  adapted version of SRPT, which show that the adapted test is minimax optimal w.r.t.~$\mathcal{P}$ up to a $\log\log n$ factor.

\begin{thm}[Critical region--adaptation--SRPT] \label{thm:perm adp typeI}
For any $0<\alpha\leq 1$, 
$$P_{H_0}\left\{\bigcup_{\lambda \in \Lambda}\stat \geq \hat{q}_{1-\frac{\alpha}{\cd}}^{B,\lambda} \right\} \leq \alpha.$$
\end{thm}

\begin{thm}[Separation boundary--adaptation--SRPT] \label{thm:perm adp TypeII}
Suppose 
$(A_0)$--$(A_4)$ hold, $\Tilde{\theta}=\min(\theta,\xi)$, $\Tilde{\xi}=\max(\xi,\frac{1}{4})$,  $\sup_{\theta>0}\sup_{P \in \PP} \norm{\T^{-\theta}u}_{\Lp} < \infty$, and $m \geq n$. Then for any $\delta>0$, $B\geq \frac{3}{{\tilde{\alpha}}^2}\left(\log2\delta^{-1}+\tilde{\alpha}(1-\tilde{\alpha})\right)$, $0<\alpha\leq e^{-1}$, $\tilde{\alpha}=\frac{\alpha}{\cd}$, $\theta_l>0$, we have 

$$\inf_{\theta>\theta_l}\inf_{P \in \PP} P_{H_1}\left\{\bigcup_{\lambda \in \Lambda}\stat \geq \hat{q}_{1-\frac{\alpha}{\cd}}^{B,\lambda} \right\} \geq 1-5\delta,$$

provided one of the following cases holds:
\begin{enumerate}[label=(\roman*)]

\item $\lambda_i \asymp i^{-\beta},$ $1<\beta \leq \beta_U$, $\lambda_L =r_1\left(\frac{n}{\log \log n} \right)^\frac{-4\beta_U}{1+2\beta_U}$, $\lambda_U=r_2\left(\frac{n}{\log \log n} \right)^\frac{-2}{4\Tilde{\xi}+1}$ for $r_1,r_2>0$,   
$s \geq 32c_1\K \lambda_L^{-1}\log(\max\{17920\K^2\lambda_L^{-1},6\}\delta^{-1})$, and $$\Delta_n=c(\alpha,\delta)\max\left\{\left( \frac{n}{\log \log n}\right)^\frac{-8\Tilde{\theta}\beta}{1+2\beta+4\Tilde{\theta}\beta}, \left( \frac{n}{\log \log n}\right)^\frac{-4\Tilde{\theta}\beta}{4\Tilde{\theta}\beta+1}\right\}.$$ 

Furthermore if $C:=\sup_i \norm{\phi_i}^2_{\infty} < \infty$, then the above conditions on $\lambda_L,$ $\lambda_U$ and $s$ can be replaced by $\lambda_L =r_3 \left(\frac{n}{\log \log n} \right)^\frac{-2\beta_U}{4\theta_l\beta_U+1}$, $\lambda_U=r_4\left(\frac{n}{\log \log n} \right)^\frac{-2}{4\xi+1}$ for $r_3,r_4>0$,

$s\geq 32c_1C^2\mathcal{N}_{1}(\lambda_L)\log\frac{32\mathcal{N}_{1}(\lambda_L)}{\delta}$, and
$$\Delta_n=c(\alpha,\delta)\left( \frac{n}{\log \log n}\right)^\frac{-4\Tilde{\theta}\beta}{4\Tilde{\theta}\beta+1},$$
where $c(\alpha,\delta)\gtrsim \delta^{-2}(\log \frac{1}{\alpha})^2.$

\item  $\lambda_i \asymp e^{-\tau i}$, $\tau>0$, $\lambda_L=r_5 \left(\frac{n}{\sqrt{\log n}\log \log n}\right)^{-2},$ $ \lambda_U =r_6 \left(\frac{n}{\sqrt{\log n}\log \log n}\right)^{-1/2\tilde{\xi}}$ for some $r_5,r_6>0,$ $s \geq 32c_1\K \lambda_L^{-1}\log(\max\{17920\K^2\lambda_L^{-1},6\}\delta^{-1})$, and $$\Delta_{n} \gtrsim 
c(\alpha,\delta,\theta)\max\left\{ \left(\frac{n}{\sqrt{\log n}\log \log n}\right)^{-1},\left(\frac{ n}{\sqrt{\log n}\log \log n}\right)^{-\frac{4\tilde{\theta}}{2\tilde{\theta}+1}} \right\},$$
where $c(\alpha,\delta,\theta)\gtrsim \max\left\{\sqrt{\frac{1}{2\tilde{\theta}}},1\right\}\delta^{-2}(\log \frac{1}{\alpha})^2.$
Furthermore if $C:=\sup_i \norm{\phi_i}^2_{\infty} < \infty$, then the above conditions on $\lambda_L,$ $\lambda_U$ and $s$ can be replaced by $$\lambda_L =r_7\left(\frac{n}{\sqrt{\log n}\log\log n}\right)^{-1/2\theta_l},\,\, \lambda_U=r_8\left(\frac{n}{\sqrt{\log n}\log\log n}\right)^{-1/2\xi}$$ for some $r_7,r_8>0$, 
$s\geq 4c_1C^2\mathcal{N}_{1}(\lambda_L)\log\frac{8\mathcal{N}_{1}(\lambda_L)}{\delta}$ and 
 $$\Delta_{n} = c(\alpha,\delta,\theta)\frac{\sqrt{\log n}\log\log n}{n},$$
where $c(\alpha,\delta,\theta) \gtrsim \max\left\{\sqrt{\frac{1}{2\tilde{\theta}}},\frac{1}{2\tilde{\theta}},1\right\}\delta^{-2}(\log \frac{1}{\alpha})^2$. 
\end{enumerate}
\end{thm}

The discussion so far has dealt with adapting to unknown $\theta$ and $\beta$ associated with a given kernel. The natural question is how to choose the kernel, for example, suppose the kernel is a Gaussian kernel, then what is the right choice of bandwidth? This is an important question because it is not easy to characterize the class of kernels that satisfy the range space and eigenvalue decay conditions for a given $P_0$. This question can be addressed by starting with a family of kernels, $\mathcal{K}$ and constructing an adaptive test by taking the union of tests jointly over $\lambda\in \Lambda$ and $K\in\mathcal{K}$, so that the resulting test is jointly adaptive over $\lambda$ and the kernel class $\mathcal{K}$. This idea has been explored recently in \cite[Section 4.5]{twosampletest} to construct a minimax optimal (up to a log factor) test that is jointly adaptive to both $\lambda$ and $\mathcal{K}$ ($\mathcal{K}$ is assumed to be finite). Since the same idea can be explored for SRCT and SRPT to create kernel adaptive tests that yield results that are similar to those of Theorems~\ref{thm: adp-gamma typeI}--\ref{thm:perm adp TypeII} along with their proofs, we skip the details here and encourage the reader to refer to \cite{twosampletest}.

\section{Experiments} \label{sec:experiments}

In this section, we investigate the empirical performance of the proposed regularized goodness-of-fit tests, SRCT and SRPT with adaptation to $\lambda$ and the kernel. Note that SRCT and SRPT are approximations to the Oracle test, since the latter is not easy to compute in general. In Section~\ref{subsec:oracle}, using a periodic spline kernel, we compare the performance of SRPT to the moderated MMD (M3D) test (i.e., Oracle test) of \cite{Krishna}, which requires the knowledge of the eigenvalues and eigenfunctions of the kernel with respect to $P_0$. Since SRCT and SRPT can be treated as two-sample tests, in Sections~\ref{subsec:gaussian}--\ref{subsec:directional}, we compare their performance to other popular two-sample tests in the literature such as 
adaptive MMD test (MMDAgg) \citep{MMDagg}, Energy test \citep{Energy}, Kolmogorov-Smirnov (KS) test \citep{KS,Fasano} and the spectral regularized two sample test (SR2T) proposed in \cite{twosampletest} with Showalter regularization. For these experiments we used Gaussian kernel, defined as $K(x,y)=\exp\left(-\frac{\norm{x-y}_{2}^2}{2h}\right)$, where $h$ is the bandwidth. For our tests, we construct adaptive versions by taking the union of tests jointly over $\lambda \in \Lambda$ and $h \in W$. Let $\hat{\eta}_{\lambda,h}$ be the test statistic based on $\lambda$ and bandwidth $h$. We reject $H_0$ if $\hat{\eta}_{\lambda,h} \geq \hat{q}_{1-\frac{\alpha}{\cd|W|}}^{B,\lambda,h}$ for any $(\lambda,h) \in \Lambda \times W$. We performed such a test for $\Lambda :=\{\lambda_L, 2\lambda_L, ... \,, \lambda_U\},$ and $W:=\{w_L h_m, 2w_L h_m, ... \,, w_Uh_m\}$, where $h_m:= \text{median}\{\norm{q-q'}_2^2: q,q' \in X \cup X^0\}$, $X:=(X_1,\ldots,X_n)$ and $X^0:=(X^0_1,\ldots,X^0_m)$. In our experiments, we set $\lambda_L=10^{-6}$, $\lambda_U=5,$ $w_L=0.01$ and $w_U=100.$ All tests are repeated 200 times and the average power is reported. For all experiments, we set $\alpha=0.05$. For the tests SRPT and SR2T, we set the number of permutations to $B=60$ and the number of samples used to estimate the covariance operator to $s=100$.

\subsection{Periodic spline kernel \& perturbed uniform distribution: Oracle test}\label{subsec:oracle}
In this section, we compare the power of SRPT to that of M3D \cite{Krishna}. To be able to compute the M3D test, we use the periodic spline kernel, defined as $K(x,y)=\frac{(-1)^{r -1}}{(2r) !}B_{2r}([x-y]),$ where $B_r$ is the Bernoulli polynomial and $[t]$ is the fractional part of $t$. We set $r=1$ and consider testing uniformity on the unit interval $\X=[0,1]$. Under this setting, the eigenvalues and eigenfunctions of $K$ are known explicitly (see \citealp[Section 5]{Krishna} for details) so that the test statistic can be exactly computed. We examine testing the null hypothesis of uniform distribution against perturbed uniform distribution (see \citealp[Section 5.1.1]{twosampletest} for details), where the perturbed uniform distribution is indexed by a parameter $P$ that characterizes the degree of perturbation. The larger the $P$ is, the associated distribution is closer to uniform, implying that it becomes more difficult to distinguish between the null and the alternative. Figure~\ref{fig:P-spine} shows the power of SRPT in comparison to M3D for varying sample sizes $n$. SRPT$(m=n)$ and SRPT$(m=3n)$ refer to our proposed permutation test while setting $m=n$, and $m=3n$ respectively (recall that $m$ is the number of samples from $P_0$ used to estimate the mean function $\mu_{P_0}$). We can observe that SRPT with enough samples from $P_0$ can yield power almost as good as M3D (Oracle test), while not requiring the exact eigenvalues and eigenfunctions of $\mathcal{T}$. We also observe that $s$ (the number of samples used to estimate the covariance operator, $\Sigma_0$) does not have much significance on the power and the choice of $s=100$ seems to be good enough to accurately estimate $\Sigma_0$.

Other than this experiment, unfortunately, we are not able to replicate any other experiment from \citep{Krishna} since no details about the parameter settings of the null and the alternative distributions are provided (i.e., if $P_0$ is normal, its mean and variance are not mentioned). Moreover, the exact details about the computation of the eigenvalues and eigenfunctions of $\mathcal{T}$ are not provided.

\begin{rem}

Theorem~\ref{thm:perm adp TypeII} states that choosing any $m\geq n$ should be enough to achieve the same separation boundary up to constants as the Oracle test. Using $m=3n$ as compared to $m=n$ will theoretically yield the same separation boundary in terms of $n$ but with a better constant closer to that of the Oracle test. To demonstrate this point, for the rest of the experiments, we used $m=3n.$ 
\end{rem}

\begin{figure}[t]
\centering
\includegraphics[scale=0.5]{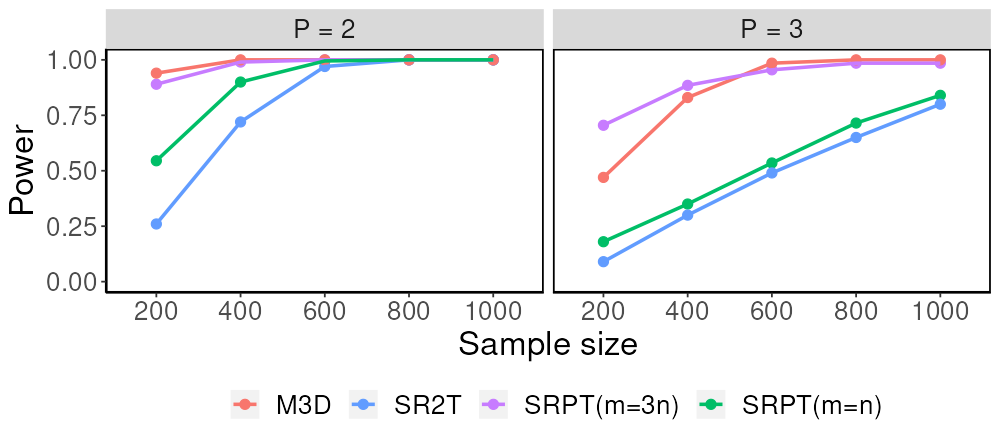}
\caption{Oracle test (M3D) and SRPT to test for uniformity using periodic spline kernel on $[0,1]$. $P$ denotes the degree of perturbation where large $P$ makes the alternative distribution (i.e., the perturbed uniform distribution) to be closer to the null (uniform distribution).}\vspace{-6mm}
\label{fig:P-spine}
\end{figure}

\subsection{Gaussian distribution}\label{subsec:gaussian}
In this section, we examine the Gaussian location shift and covariance scale problems, where the observed samples are generated from a Gaussian distribution with a shifted mean or scaled covariance matrix (by scaling the diagonal elements of the identity matrix). The goal is to test the null hypothesis of standard Gaussian distribution. Figure \ref{fig:Gaussian_shift}(a) shows the power for different mean shifts and different dimensions from which we note that the Energy test gives the best power closely followed by the SRPT test. Figure~\ref{fig:Gaussian_shift}(b) shows the result for different choices of $s$ for both SRCT and SRPT tests with Showalter regularization. We can see that SRPT is not very sensitive to the choice of $s$ as opposed to SRCT, which seems to give higher power for lower values of $s$, however with the cost of a worse Type-I error (shown at mean shift = 0). We can see that the choice of $s=100$ controls both power and Type-I error for SRCT and for this choice of $s=100$, the permutation test SRPT yields a higher power while still controlling the Type-I error. Similarly, Figure~\ref{fig:Gaussian_scale} shows the power for different scaling factors with different dimensions and different choices of $s$, demonstrating similar results.

\begin{figure}[t]
\begin{minipage}[b]{0.03\linewidth}
\end{minipage}
\begin{minipage}[b]{0.96\linewidth}
\centering
\includegraphics[scale=0.54]{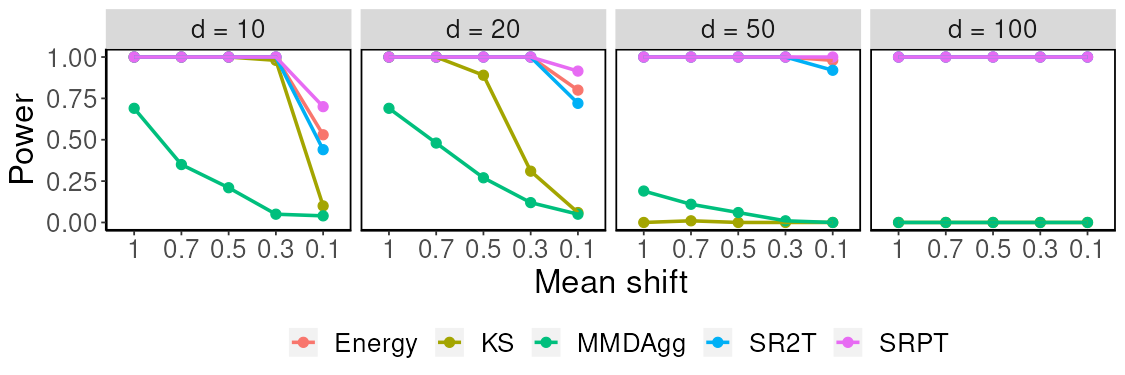}
\end{minipage}
\begin{minipage}[b]{0.03\linewidth}
\end{minipage}
\begin{minipage}[b]{0.96\linewidth}
\centering
\includegraphics[scale=0.5]{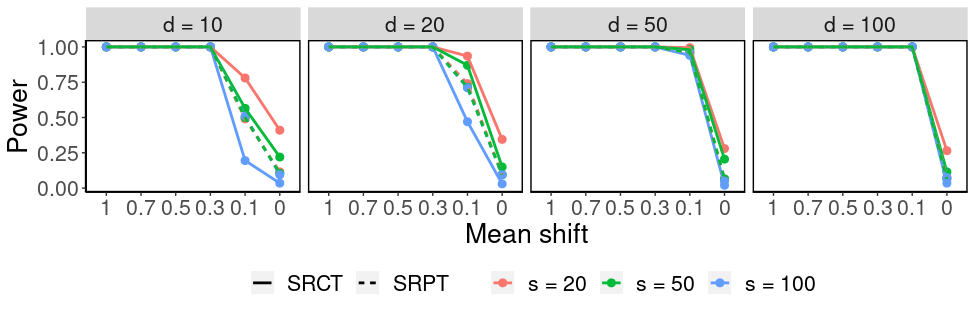}
\end{minipage}
\vspace{-4mm}
\caption{Power for Gaussian shift experiments with different $d$ and $s$ using $n=200$.}\label{fig:Gaussian_shift}
\vspace{1mm}
\end{figure}

\begin{figure}[t]
\begin{minipage}[b]{0.03\linewidth}
\end{minipage}
\begin{minipage}[b]{0.96\linewidth}
\centering
\includegraphics[scale=0.5]{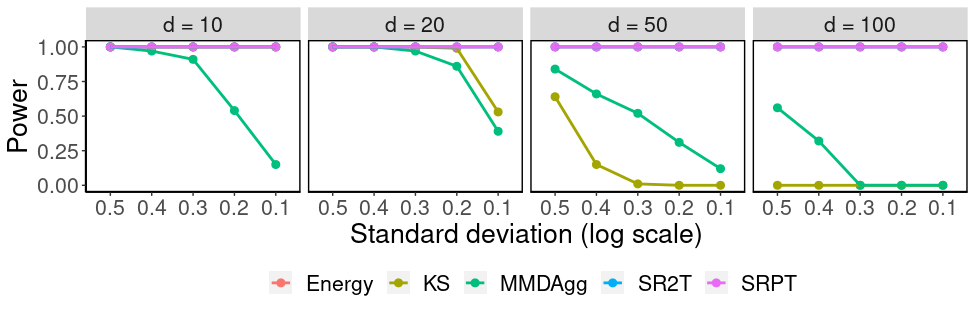}
\end{minipage}
\begin{minipage}[b]{0.03\linewidth}
\end{minipage}
\begin{minipage}[b]{0.96\linewidth}
\centering
\includegraphics[scale=0.5]{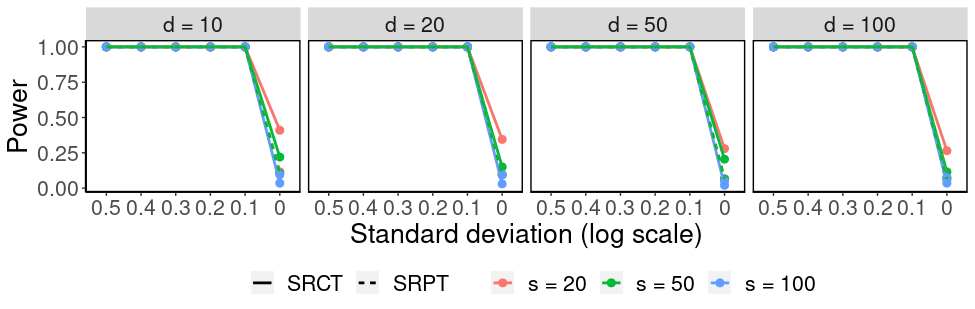}
\end{minipage}
\vspace{-4mm}
\caption{Power for Gaussian covariance scale experiments with different $d$ and $s$ using $n=200.$} \label{fig:Gaussian_scale}
\vspace{-6mm}
\end{figure}

\subsection{Perturbed uniform distribution}\label{subsec:uniform}
In this part, we examine testing the null hypothesis of uniform distribution against perturbed uniform distribution for different values of perturbation, $P$ (see \citealp[Section 5.1.1]{twosampletest} for details). Figure \ref{fig:pert_uni}(a) shows the result for $d \in \{1,2\}$ for different perturbations, wherein we can see that the highest power is achieved by SRPT. Figure \ref{fig:pert_uni}(b) shows the power for SRCT and SRPT for different choices of $s$, with $P=0$ corresponding to no perturbations and thus showing Type-I error. Similar to the observation from the previous section, SRPT is not very sensitive to the choice of $s$, while SRCT is sensitive to $s$, with $s=50$ and $s=200$ being the reasonable choices, respectively for $d=1$ and $d=2$ that controls both power and Type-I error.

\begin{figure}[t]
\begin{minipage}[b]{0.03\linewidth}
\end{minipage}
\begin{minipage}[b]{0.96\linewidth}
\centering
\includegraphics[scale=0.48]{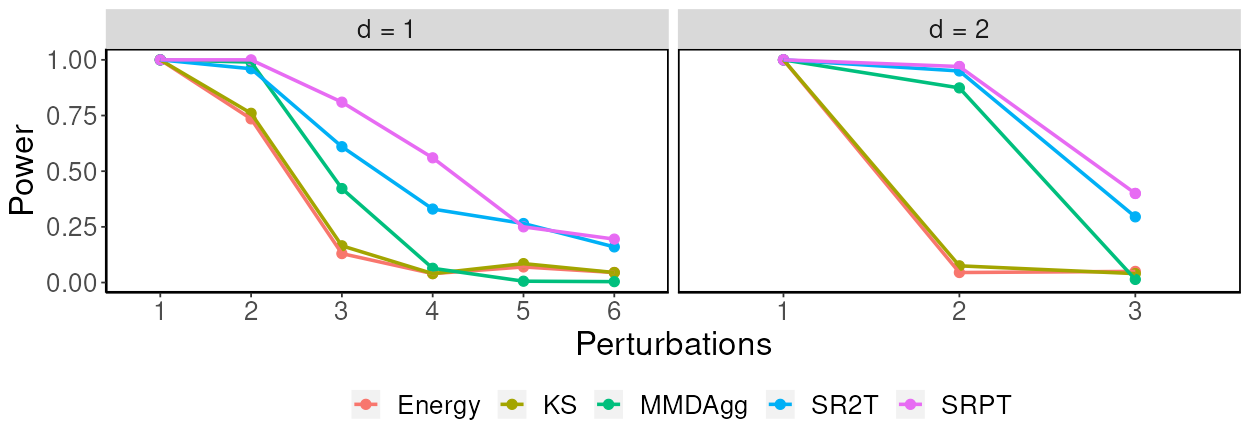}
\end{minipage}
\begin{minipage}[b]{0.03\linewidth}
\end{minipage}
\begin{minipage}[b]{0.96\linewidth}
\centering
\includegraphics[scale=0.5]{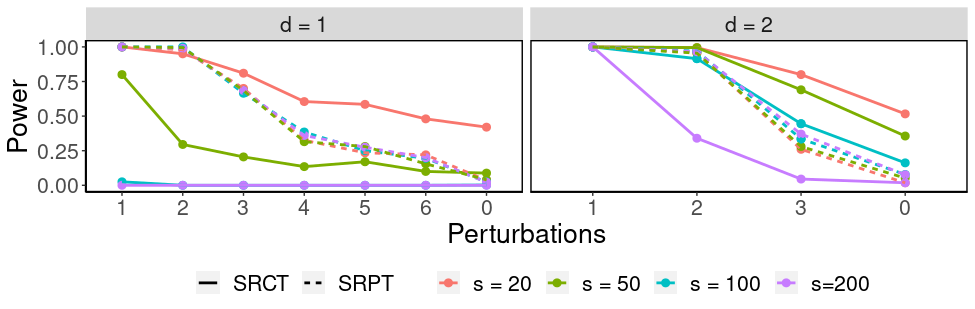}
\end{minipage}
\vspace{-4mm}
\caption{Power for perturbed uniform distributions for $d=1$ ($n=500$) and $d=2$ ($n=2000$).}\label{fig:pert_uni}
\vspace{-6mm}
\end{figure}

\subsection{Directional data}\label{subsec:directional}
In this section, we investigate two experiments involving directional domains, where we focus on testing for a multivariate von Mises-Fisher distribution, which serves as the Gaussian analog on the unit sphere defined by the density $f(x,\mu,k)=\frac{k^{d/2-1}}{2\pi^{d/2}I_{d/2-1}(k)}\exp(k\mu^Tx),\,x\in \mathbb{S}^{d-1},$ with $k\geq 0$ being the concentration parameter, $\mu$ being the mean parameter and $I$ being the modified Bessel function. Figure~\ref{fig:von-mises}(a) shows the results for testing von Mises-Fisher distribution against spherical uniform distribution ($k=0$) for different concentration parameters. We can see from Figure~\ref{fig:von-mises}(a) that the best power is achieved by the Energy test followed closely by SRPT. Figure~\ref{fig:von-mises}(b)
shows that SRPT is less sensitive to the choice of $s$ as opposed to SRCT which achieves its best power at $s=100$ while still controlling the Type-I error. In the second experiment, we explore a mixture of two multivariate Watson distributions, representing axially symmetric distributions on a sphere, given by $f(x,\mu,k)=\frac{\Gamma(d/2)}{2\pi^{d/2}M(1/2,d/2.\kappa)}\exp(k(\mu^Tx)^2)$, $x\in \mathbb{S}^{d-1},$ where $k\geq 0$ is the concentration parameter, $\mu$ is the mean parameter and $M$ is Kummer's confluent hypergeometric function. Using equal weights we drew 500 samples from a mixture of two Watson distributions with similar concentration parameter $k$ and mean parameter $\mu_1 ,\mu_2$ respectively, where $\mu_1=(1, \ldots ,1) \in \R^d$ and $\mu_2=(-1,1 \ldots ,1) \in \R^d$. Figure~\ref{fig:watson}(a) shows the power against spherical uniform distribution for different concentration parameters. We can see that SRPT outperforms all the other methods. Figure~\ref{fig:watson}(b) shows how the power and Type-I error are affected by the choice of $s$, which similar to the previous sections shows that SRPT is not very sensitive to $s$, while SRCT achieves its best power while still controlling for Type-I error at $s=100.$ 
\begin{figure}[t]
\begin{minipage}[b]{0.03\linewidth}
\end{minipage}
\begin{minipage}[b]{0.96\linewidth}
\centering
\includegraphics[scale=0.49]{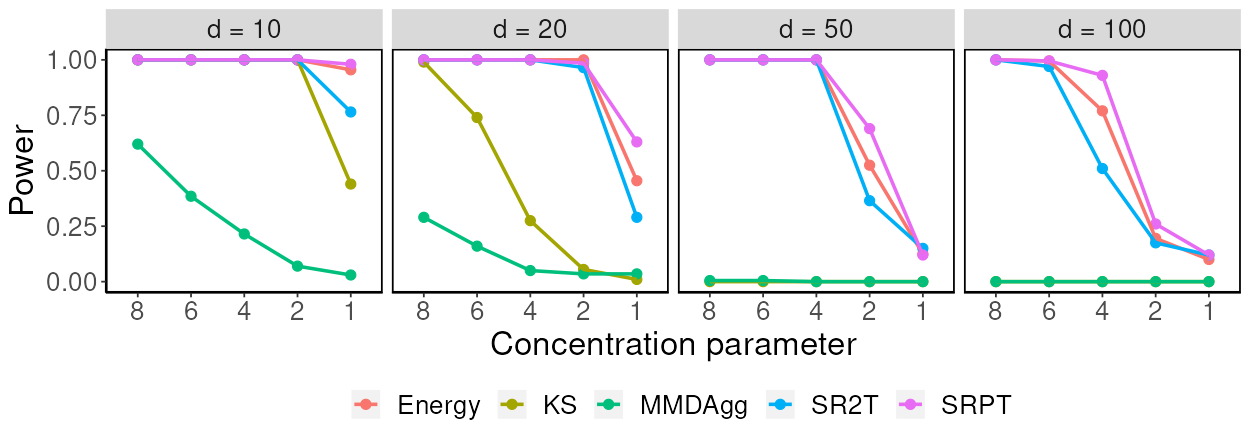}
\end{minipage}
\begin{minipage}[b]{0.03\linewidth}
\end{minipage}
\begin{minipage}[b]{0.96\linewidth}
\centering
\includegraphics[scale=0.5]{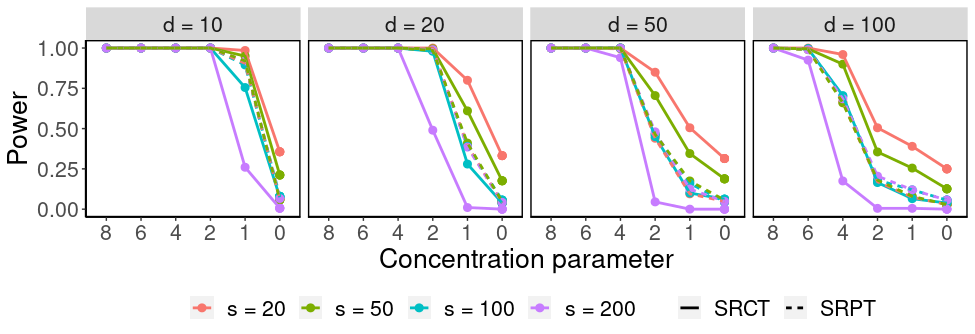}
\end{minipage}
\vspace{-4mm}
\caption{Power for von Mises-Fisher distribution with different concentration parameter $k$ and $s$ using $n=500.$} \label{fig:von-mises}
\vspace{-6mm}
\end{figure}

\begin{figure}[t]
\begin{minipage}[b]{0.03\linewidth}
\end{minipage}
\begin{minipage}[b]{0.96\linewidth}
\centering
\includegraphics[scale=0.57]{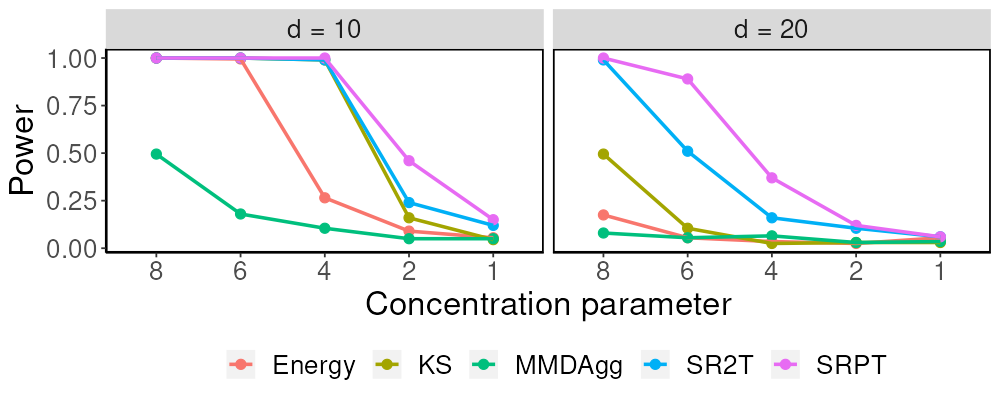}
\end{minipage}
\begin{minipage}[b]{0.03\linewidth}
\end{minipage}
\begin{minipage}[b]{0.96\linewidth}
\centering
\includegraphics[scale=0.5]{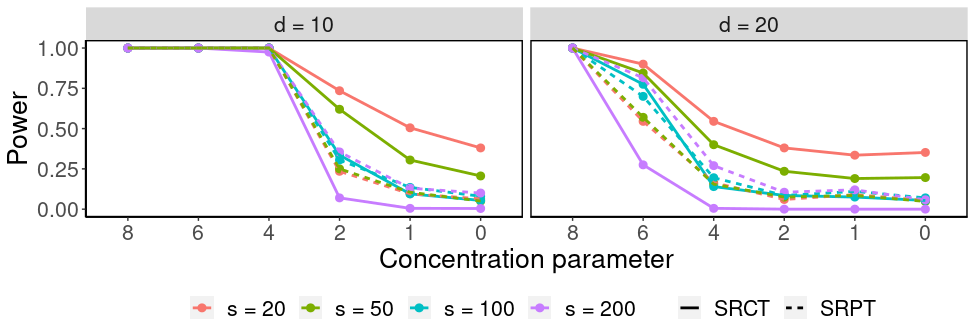}
\end{minipage}
\vspace{-4mm}
\caption{Power for mixture of Watson distributions with different concentration parameter $k$ and $s$ using $n=500.$} \label{fig:watson}
\vspace{-6mm}
\end{figure}

\section{Discussion}\label{sec:discussion}
To summarize, in this work, we have extended and generalized the theoretical properties of the Oracle test proposed by \cite{Krishna} by employing a general spectral regularization approach, wherein we obtained sufficient conditions for the separation boundary under weaker assumptions and for a wider range of alternatives. Under the assumption that we have access to samples from $P_0$, we addressed the problem of the practicality of the Oracle test by proposing two completely data-driven tests (SRCT and SRPT) that adapt to the choice of the kernel, the eigenvalue decay rate and the smoothness of the likelihood ratio, while still being minimax optimal (up to logarithmic factors) 
w.r.t.~$\PP$. Through numerical experiments, we established the superior performance of the proposed spectral regularized tests over the MMD-based test and the closely related two-sample test proposed in \cite{twosampletest}.

However, there are still some open questions for future consideration: (i) Improving the computational complexity of the proposed tests using approximation schemes like random Fourier features \citep{Rahimi-08a}, Nystr\"{o}m method (e.g., \citealt{Williams-01,Drineas-05}) or sketching~\citep{Yang-17}, and studying the computational vs.~statistical trade-off for the approximate test. (ii) The proposed test requires access to i.i.d.~samples from $P_0$, which might not be easy to generate or easily available. As an alternative, instead of regularizing w.r.t.~$\hat{\Sigma}_{P_0}$ (i.e., the empirical covariance operator estimated based on the samples from $P_0$), one can regularize with respect to $\hat{\Sigma}_P$ (i.e., the empirical covariance operator estimated from the available samples drawn from $P$). However, even this approach is not practical unless the mean element $\mu_0$ is computable, which need not be the case. To completely get around this issue, we can consider applying the idea of spectral regularization to Kernel Stein Discrepancy (KSD) \citep{chwialkowski16,liub16} which does not require computing any integrals with respect to $P_0$, and study its minimax optimality.
\section{Proofs}
In this section, we present the proofs of the main results of the paper.
\subsection{Proof of Theorem \ref{thm: MMD}}
Define $b(x)= \kk(\cdot,x)-\mu_P$ and $a(x)=b(x)+(\mu_P-\mu_0)=\kk(\cdot,x)-\mu_0$,. Thus we can write, 
\begin{align*}
    \hat{D}_{\mathrm{MMD}}^2 &= \frac{1}{n(n-1)}\sum_{i\neq j}\inner{a(X_i)}{a(X_j)}_{\h}\\
    &=\frac{1}{n(n-1)}\sum_{i\neq j}\inner{b(X_i)}{b(X_j)}_{\h}+\frac{2}{n} \sum_{i=1}^{n}\inner{b(X_i)}{(\mu_P-\mu_0)}_{\h}+D_{\mathrm{MMD}}^2 \\
&    = I_1+I_2+D_{\mathrm{MMD}}^2,
\end{align*} where 
\begin{align*}
   I_1=\frac{1}{n(n-1)}\sum_{i\neq j}\inner{b(X_i)}{b(X_j)}_{\h},\qquad\text{and}\qquad\,\,\,\,  I_2=\frac{2}{n} \sum_{i=1}^{n}\inner{b(X_i)}{(\mu_P-\mu_0)}_{\h}.
\end{align*}
so that $\E[ (\hat{D}_{\mathrm{MMD}}^2-D_{\mathrm{MMD}}^2)^2] = \E_{P}(I_1+I_2)^2 \leq 2\E_P(I_1^2)+2\E_P(I_2^2).$ Next, following similar ideas as in the proofs of \cite[Lemmas A.4, A.5]{twosampletest}, we can bound $I_1$ and $I_2$ as 
\begin{align*}
      \E\left(I_1^2\right)  \leq \frac{4}{n^2} \norm{\Sigma_P}_{\hs}^2,\qquad\text{and}\qquad
        \E\left(I_2^2 \right)
         \leq\frac{4}{n} \norm{\Sigma_P}_{\op}\norm{\mu_P-\mu_0}_{\h}^2,
\end{align*}
respectively. Combining these bounds yields that 
\begin{equation}
    \E[ (\hat{D}_{\mathrm{MMD}}^2-D_{\mathrm{MMD}}^2)^2] 
    \lesssim \frac{1}{n^2}+ \frac{D_{\mathrm{MMD}}^2}{n}. \label{Eq:MMD_var1} 
\end{equation}
When $P=P_0$, we have $D_{\mathrm{MMD}}^2=0$. Therefore under $H_0$,
\begin{equation}
\E[ (\hat{D}_{\mathrm{MMD}}^2)^2] \leq \frac{4\norm{\Sigma_{0}}_{\hs}^2}{n^2} \stackrel{(*)}{\leq} \frac{16\kappa^2}{n^2},\label{Eq:MMD_var2}
\end{equation}
where in $(*)$ we used where in $(*)$ we used $\norm{\Sigma_{0}}_{\hs}^2 \leq 4\kappa^2$. Thus using \eqref{Eq:MMD_var2} and Chebyshev's inequality yields 
$$P_{H_0}\{\hat{D}_{\mathrm{MMD}}^2 \geq \gamma\} \leq \alpha,$$
where $\gamma = \frac{4\kappa}{\sqrt{\alpha}n}.$

Next, we use the bound in \eqref{Eq:MMD_var1} to bound the power.
Let $\gamma_1= \frac{1}{\sqrt{\delta}n}+\frac{\sqrt{D_{\mathrm{MMD}}^2}}{\sqrt{\delta}\sqrt{n}}$. Then 
\begin{align*}
    P_{H_1}\{\hat{D}_{\mathrm{MMD}}^2  \geq \gamma\} & \stackrel{(*)}{\geq} P_{H_1}\{\hat{D}_{\mathrm{MMD}}^2 > D_{\mathrm{MMD}}^2-\gamma_1\} \\ &\geq P_{H_1}\{|\hat{D}_{\mathrm{MMD}}^2 - D_{\mathrm{MMD}}^2|\leq \gamma_1\} \stackrel{(**)}{\geq} 1- \delta,
\end{align*}
where $(*)$ holds when $D_{\mathrm{MMD}}^2 \geq \gamma+\gamma_1$, which is implied if $D_{\mathrm{MMD}}^2 \gtrsim \frac{1}{n}$, which in turn is implied if $\norm{u}_{\Lp}^2 \gtrsim n^{\frac{-2\theta}{2\theta+1}},$ where the last implication follows from \cite[Lemma A.19]{twosampletest}. $(**)$ follows from \eqref{Eq:MMD_var1} and an application of Chebyshev's inequality. The desired result, therefore, holds by taking infimum over $P \in \PP$.

Finally, we will show that we cannot achieve a rate better than $n^{\frac{-2\theta}{2\theta+1}}$ over $\PP$. Recall that $\T = \sum_{i \in I} \lambda_i \Tilde{\phi_i} \ltens \Tilde{\phi_i}$. Let $\bar{\phi_i} =\phi_i-\E_{P_0}\phi_i$, where $\phi_i=\frac{\id^*\Tilde{\phi_i}}{\lambda_i}$. Then $\id\bar{\phi_i}=\id\phi_i=\frac{\T\tilde{\phi}_i}{\lambda_i}=\Tilde{\phi_i}.$ Assuming $\lambda_i=h(i)$, where $h$ is an invertible, continuous function (for example $h=i^{-\beta}$ and $h=e^{-\tau i}$ correspond to polynomial and exponential decays respectively), let $k=\lfloor h^{-1}(n^{\frac{-1}{2\theta+1}})\rfloor$, hence $\lambda_k=n^{\frac{-1}{2\theta+1}}$. Define $$f:=b\bar{\phi}_k,$$ where  $b<\sqrt{\frac{4\kappa}{4\sqrt{\alpha}}}$. Then $\norm{f}_{\Lp}^2 = b^2 \lesssim 1,$ and thus $f \in \Lp$. Define 
$$\tilde{u}:=\T^{\theta}f=b\lambda_k^{\theta}\tilde{\phi}_k,\quad\text{and}\quad u:=b\lambda_k^{\theta}\bar{\phi}_k.$$ 
Note that $\E_{P_0} u = b \lambda^\theta_k\E_{P_0}\bar{\phi}_k =0$. Since $\id u=\Tilde{u}$, we have $u \in [\Tilde{u}]_\sim \in \range(\T^{\theta}),$ $\norm{u}_{\Lp}^2=b^2\lambda_k^{2\theta} > \Delta_{n}.$ Next we bound $|u(x)|$ in the following two cases.\\

\underline{\emph{Case I:}}  $\theta\ge\frac{1}{2}$ and $\sup_{i} \norm{\phi_i}_{\infty}$ \emph{is not finite.}\vspace{1.5mm}\\
Note that
$$|u(x)| =b\lambda_k^{\theta}\left|\inner{k(\cdot, x)-\mu_0}{ \phi_k}_{\h}\right|  
    \leq b\lambda_k^{\theta}\norm{k(\cdot, x)-\mu_0}_{\h}\norm{ \phi_k}_{\h}  
    \stackrel{(*)}{\leq} 2b\sqrt{\kappa} \lambda_k^{\theta-\frac{1}{2}} \stackrel{(\dag)}{\leq} 1,$$
where in $(*)$ we used $\norm{\phi_k}^2_{\h} =\lambda_k^{-2}\inner{\id^*\tilde{\phi}_k}{\id^*\tilde{\phi}_k}=\lambda_k^{-1}.$ In $(\dag)$ we used $\theta>\frac{1}{2}.$\\

\underline{\emph{Case II:}} $\sup_{i}\norm{\phi_i}_{\infty} < \infty$.\vspace{1.5mm}\\
In this case,
$$|u(x)| \leq 2b\sup_k\norm{\phi_k}_{\infty} \lambda_k^{\theta}\leq 1,$$ for $n$ large enough.  This implies that we can find $P \in \PP$ such that $\frac{dP}{dP_0}=u+1.$ Then for such $P$, we have $D^2_{\mathrm{MMD}}=\norm{\T^{1/2}u}_{\Lp}^2=b^2\lambda_k^{2\theta+1}=\frac{b^2}{n} < \frac{4 \kappa}{\sqrt{\alpha}n}=\gamma$. Therefore there exists some $\epsilon>0$ such that $D^2_{\mathrm{MMD}} < \gamma -\epsilon.$ Hence, we have
\begin{align*}
    P_{H_1}\{\hat{D}_{\mathrm{MMD}}^2 \geq \gamma\} &< P_{H_1}\{\hat{D}_{\mathrm{MMD}}^2\geq D_{\mathrm{MMD}}^2 + \epsilon\}  < P_{H_1}\{|\hat{D}_{\mathrm{MMD}}^2- D_{\mathrm{MMD}}^2| \geq \epsilon\} \\ & \stackrel{(*)}{\lesssim}\frac{1}{\epsilon^2n^2} \leq \delta, 
\end{align*}
where we used \eqref{Eq:MMD_var1} along with Chebyshev's inequality in $(*)$, and the last inequality holds for $n> \frac{1}{\epsilon \sqrt{\delta}}.$
\subsection{Proof of Theorem \ref{thm:minimax}}
As shown in \cite[Lemma G.1]{permutations}, in order to show that a separation boundary $\Delta_n$ will imply $R^*_{\Delta_{n}} \geq \delta$, it is sufficient to find a set of distributions $\{P_k\}_{k=1}^L \subset \PP,$ such that 
\begin{equation}\label{Eq:E.1}
    \E_{P_0^n}\left[\left(\frac{1}{L}\sum_{k=1}^L\frac{dP_k^n}{dP_0^n}\right)^2\right] \leq 1+4(1-\alpha-\delta)^2.
\end{equation}
Then the proof follows the same ideas as used in the proof of \cite[Theorem 2]{twosampletest} as shown briefly below. 

Recall $\T = \sum_{i \in I} \lambda_i \Tilde{\phi_i} \ltens \Tilde{\phi_i}$. Let $\bar{\phi_i} =\phi_i-\E_{P_0}\phi_i$, where $\phi_i=\frac{\id^*\Tilde{\phi_i}}{\lambda_i}$. Then $\id\bar{\phi_i}=\id\phi_i=\frac{\T\tilde{\phi}_i}{\lambda_i}=\Tilde{\phi_i}$.\vspace{.5mm}\\

\underline{\emph{Polynomial decay (Case I):}} $\lambda_i \asymp i^{-\beta}$, $\beta>1$, $\sup_{i}\norm{\phi_i}_{\infty} < \infty$ and $\theta\ge\frac{1}{4\beta}$.\vspace{1.5mm}\\
Let $$B_n=\left\lfloor\frac{\Delta_n^{-1/2\theta\beta}}{16(\sup_{i}\norm{\phi_i}_{\infty})^2}\right\rfloor,$$ $C_n=\lfloor \sqrt{B_n} \rfloor$ and $a_n=\sqrt{\frac{\Delta_n}{C_n}}$. For $k \in \{1,\dots,L\},$ define 
$$u_{n,k}:=a_N\sum_{i=1}^{B_n}\epsilon_{ki} \bar{\phi_i},$$
where $\epsilon_k := \{\epsilon_{k1}, \epsilon_{k2},\ldots, \epsilon_{kB_n}\}\in \{0,1\}^{B_n}$ such that $\sum_{i=1}^{B_n}\epsilon_{ki}=C_n$, thus $L={B_n \choose C_n}$.
Then it can be shown (see the proof of \citealt[Theorem 2]{twosampletest}) that we can find $P_k \in \PP$ such that $\frac{dP_k}{dP_0}=u_{n,k}+1,$ and that \eqref{Eq:E.1} holds for $\theta>\frac{1}{4\beta}$ when $\Delta_{n}  \leq  c(\alpha,\delta) n^{\frac{-4\theta\beta}{4\theta\beta+1}}$ for some $c(\alpha,\delta).$ \vspace{.5mm}\\

\underline{\emph{Polynomial decay (Case II):}} $\lambda_i \asymp i^{-\beta}$, $\beta>1$, $\theta\ge\frac{1}{2}$ and $\sup_{i} \norm{\phi_i}_{\infty}$ \emph{is not finite.}\vspace{1.5mm}\\
Since $\lambda_i \asymp i^{-\beta}$, $\beta>1$, there exists constants $\underbar{A}>0$ and $\bar{A}>0$ such that $\underbar{A}i^{-\beta} \leq \lambda_i \leq \bar{A} i^{-\beta}$. Let $B_n=\lfloor\left(\frac{\underbar{A}\Delta_n^{-1}}{4\K}\right)^{\frac{1}{2\theta\beta}}\rfloor$, $C_n=\lfloor \sqrt{B_n} \rfloor$ and $a_n=\sqrt{\frac{\Delta_n}{C_n}}$. Then similar to \emph{Case I,}  it can be shown that (see the proof of \citealt[Theorem 2]{twosampletest}) we can find $P_k \in \PP$ such that $\frac{dP_k}{dP_0}=u_{n,k}+1,$ and that \eqref{Eq:E.1} holds for $\theta>\frac{1}{2}$ when $\Delta_{n}  \leq  c(\alpha,\delta) n^{\frac{-4\theta\beta}{4\theta\beta+1}}$ for some $c(\alpha,\delta).$ \vspace{.5mm}\\

\underline{\emph{Exponential decay:}} $\lambda_i \asymp e^{-\tau i}$, $\tau>0$.\vspace{1.5mm}\\
Since $\lambda_i \asymp e^{-\tau i}$, $\tau>0$, there exists constants $\underbar{A}>0$ and $\bar{A}>0$ such that $\underbar{A}e^{-\tau i} \leq \lambda_i \leq \bar{A} e^{-\tau i}$. Let $B_n=\lfloor (2\tau\max\{\theta,\frac{1}{2}\})^{-1}\log(\frac{\underbar{A}}{4\kappa\Delta_n}) \rfloor$, $C_n=\lfloor \sqrt{B_n} \rfloor$ and $a_n=\sqrt{\frac{\Delta_n}{C_n}}$, where $\theta>0$. Then similar to the previous cases it can shown that \eqref{Eq:E.1} holds when $\Delta_{n}  \leq  c(\alpha,\delta,\theta)\frac{(\log n)^{b}}{n} $ for any $b<\frac{1}{2}.$ Thus the desired bound holds by taking supremum over $b < \frac{1}{2}.$
\subsection{Proof of Theorem \ref{thm:typI-oracle}}
By defining $\B : =\gSL\SL^{1/2}$ and $a(x)= \B\Sigma_{0,\lambda}^{-1/2}(\kk(\cdot,x)-\mu_0)$, we have 
$$\hat{\eta}_{\lambda}=\frac{1}{n(n-1)}\sum_{i\neq j}\inner{a(X_i)}{a(X_j)}_\h.$$ 
By replacing $\Sigma_{PQ}$ with $\Sigma_0$ in the proof of 
\cite[Lemma A.4]{twosampletest}, we have 
$$\E_{P_0}(\hat{\eta}_{\lambda}^2)\leq \frac{4}{n^2}\norm{\B}^4_{\op}\Ntlsq.$$
The result therefore follows by applying Chebyshev's inequality and noting from \cite[Lemma A.8(ii)]{twosampletest} that $\norm{\B}^2_{\op} \leq \Cs$.

\subsection{Proof of Theorem \ref{thm:typII-oracle}} 

Define $\B : =\gSL\SL^{1/2}$, $b(x)= \B\Sigma_{0,\lambda}^{-1/2}(\kk(\cdot,x)-\mu_P)$, and $a(x)=b(x)+\B\Sigma_{0,\lambda}^{-1/2}(\mu_P-\mu_0)= \B\Sigma_{0,\lambda}^{-1/2}(\kk(\cdot,x)-\mu_0)$. 
 
Thus we can write, 
\begin{align*}
    \hat{\eta}_{\lambda} &= \frac{1}{n(n-1)}\sum_{i\neq j}\inner{b(X_i)}{b(X_j)}_{\h}+\frac{2}{n} \sum_{i=1}^{n}\inner{b(X_i)}{\B\SL^{-1/2}(\mu_P-\mu_0)}_{\h}+\eta_{\lambda}\\
    &= I_1+I_2+ \eta_{\lambda},
\end{align*}
where
\begin{align*}
   I_1=\frac{1}{n(n-1)}\sum_{i\neq j}\inner{b(X_i)}{b(X_j)}_{\h},\qquad\text{and}\qquad  I_2=\frac{2}{n} \sum_{i=1}^{n}\inner{b(X_i)}{\B\SL^{-1/2}(\mu_P-\mu_0)}_{\h}.
\end{align*}
Thus, $\text{Var}_{P}[\hat{\eta}_{\lambda}]=\E_{P}(I_1+I_2)^2\le 2\E_P(I_1^2)+2\E_P(I_2^2).$ Next we bound $\E_P(I_1^2)$ and $\E_P(I_2^2)$ using \cite[Lemmas A.4 and A.5]{twosampletest} by replacing $\Sigma_{PQ}$ with $\Sigma_0$, which yields
\begin{align*}
    \text{Var}_P[\hat{\eta}_{\lambda}] &\leq \frac{8}{n^2}\norm{\B}_{\op}^4\norm{\Sigma_{0,\lambda}^{-1/2}\Sigma_P\SgL}_{\hs}^2 \\ & \qquad+ \frac{8}{n}\norm{\B}_{\op}^4\norm{\SgL\Sigma_P\SgL}_{\op}\norm{\SgL(\mu_P-\mu_0)}_{\h}^2 \\
    &\stackrel{(*)}{\leq} \frac{8}{n^2}\norm{\B}_{\op}^4(4 \Cl \norm{u}_{\Lp}^2  + 2\Ntlsq) \\
    & \qquad+\frac{8}{n}\norm{\B}_{\op}^4(2 \sqrt{\Cl} \norm{u}_{\Lp} + 1)\norm{u}_{\Lp}^2,
\end{align*}
where $\Cl$ as defined in Lemma \ref{lemma: bound hs and op} and in $(*)$, we used Lemmas \ref{lemma: bounds for eta} and \ref{lemma: bound hs and op}. Then we can easily deduce by using Chebyshev's inequality that $P_{H_1}\{\hat{\eta}_{\lambda} \geq \gamma\} \geq 1-\delta$ holds if $\eta_{\lambda} \geq \gamma + \sqrt{\frac{\text{Var}_{P}[\hat{\eta}_{\lambda}]}{\delta}}$. Using Lemma \ref{lemma: bounds for eta}, we have $\eta_{\lambda} \geq \frac{B_3}{4}\norm{\U}_{\Lp}^2$ under the assumptions $u \in \text{\range}(\T^{\theta})$, and 
\begin{equation}\norm{\U}_{\Lp}^2 \geq  \frac{4C_3}{3B_3} \norm{\T}_{\op}^{2\max(\theta-\xi,0)}\lambda^{2 \Tilde{\theta}} \norm{\T^{-\theta}\U}_{\Lp}^2.\label{Eq:lower}\end{equation}
Note that $u \in \text{\range}(\T^{\theta})$ is guaranteed since $P \in \PP$ and 
\begin{equation}
\norm{u}_{\Lp}^2 \geq c_4\lambda^{2 \Tilde{\theta}}\label{Eq:verify-1}  
\end{equation} guarantees \eqref{Eq:lower}
since $\norm{\T}_{\op}=\norm{\Sigma_{0}}_{\op}\leq 2\kappa$ and $c_1:=\sup_{P \in \PP} \norm{\T^{-\theta}u}_{\Lp}<\infty$, where 
$c_4=\frac{4c_1^{2}C_3(2\kappa)^{2\max(\theta-\xi,0)}}{3B_3}.$ Thus, 
\begin{equation}
    \frac{B_3}{2}\norm{u}_{\Lp}^2 \geq \gamma + \sqrt{\frac{\text{Var}_{P}[\hat{\eta}_{\lambda}]}{\delta}} \label{Eq:verify-2}
\end{equation}
guarantees $\eta_{\lambda} \geq \gamma + \sqrt{\frac{\text{Var}_{P}[\hat{\eta}_{\lambda}]}{\delta}}$. Hence it remains to verify \eqref{Eq:verify-1} and \eqref{Eq:verify-2}. Using $\norm{u}^2_{\Lp} \geq \Delta_{n}$, it is easy to see that \eqref{Eq:verify-1} is implied when $\lambda=(c_4^{-1}\Delta_{n})^{1/2\tilde{\theta}}$. Using $\norm{\B}_{\op}^4\leq \Cs^2$, which follows from \citep[Lemma A.8 \emph{(ii)}]{twosampletest} by replacing $\Sigma_{PQ}$ with $\Sigma_0$, and substituting the expressions of $\gamma$ and $\text{Var}_{P}[\hat{\eta}_{\lambda}]$ in \eqref{Eq:verify-2}, we can verify that \eqref{Eq:verify-2} is implied if $\Delta_{n} \geq \frac{r_1\Ntl}{n\sqrt{\alpha}}$, $\Delta_{n} \geq \frac{r_2\Cl}{\delta^2 n^2}$ and $\Delta_{n}\geq\frac{r_3\Ntl}{\delta n}$ for some constants $r_1,r_2,r_3>0$.  Hence the desired result follows by taking infimum over $P \in \PP$.

\subsection{Proof of Corollary \ref{coro:poly-oracle}}
When $\lambda_i \asymp i^{-\beta}$, we have $\Ntl \leq \norm{\SgL\Sigma_{0}\SgL}_{\op}^{1/2}\mathcal{N}^{1/2}_{1}(\lambda) \lesssim \lambda^{-1/2\beta}$ (see \citealt[Lemma B.9]{kpca}). Using this bound  in the conditions mentioned in Theorem~\ref{thm:typII-oracle}, ensures that these conditions on the separation boundary hold if 
\begin{equation} \label{eq:oracle-poly-1}
\Delta_{n} \gtrsim \max\left\{\left(\frac{n}{\alpha^{-1/2}+\delta^{-1}}\right)^{-\frac{4\Tilde{\theta}\beta}{4\Tilde{\theta}\beta+1}}, (\delta n)^{-\frac{8\Tilde{\theta}\beta}{4\Tilde{\theta}\beta+2\beta+1}}\right\},
\end{equation}
which in turn is implied if 
$$\Delta_{n} =
\left\{
	\begin{array}{ll}
		c(\alpha,\delta)n^{\frac{-4\tilde{\theta}\beta}{4\Tilde{\theta}\beta+1}},  &  \ \  \Tilde{\theta}> \frac{1}{2}-\frac{1}{4\beta} \\
		c(\alpha,\delta) n^{-\frac{8\Tilde{\theta}\beta}{4\Tilde{\theta}\beta+2\beta+1}}, & \ \  \Tilde{\theta} \leq \frac{1}{2}-\frac{1}{4\beta}
	\end{array}
\right.,$$ 
where $c(\alpha,\delta)\gtrsim(\alpha^{-1/2}+\delta^{-2})$ and we used that $\tilde{\theta}> \frac{1}{2}-\frac{1}{4\beta} \Leftrightarrow \frac{4\tilde{\theta}\beta}{4\tilde{\theta}\beta+1} <  \frac{8\tilde{\theta}\beta}{4\tilde{\theta}\beta+2\beta+1}.$ On the other hand when $C:=\sup_i\norm{\phi_i}_{\infty} < \infty$, we obtain the corresponding condition as 
\begin{equation} \label{eq:oracle-poly-2}
\Delta_{n} \gtrsim \max\left\{\left(\frac{n}{\alpha^{-1/2}+\delta^{-1}}\right)^{-\frac{4\Tilde{\theta}\beta}{4\Tilde{\theta}\beta+1}},(\delta n)^{-\frac{4\Tilde{\theta}\beta}{2\Tilde{\theta}\beta+1}}\right\},    
\end{equation}
which is implied if 
$$\Delta_n = c(\alpha,\delta)n^{\frac{-4\tilde{\theta}\beta}{4\Tilde{\theta}\beta+1}}.$$

\subsection{Proof of Corollary \ref{coro:exp-oracle}}
When $\lambda_i \asymp e^{-\tau i}$, we have $\Ntl \leq \norm{\SgL\Sigma_{0}\SgL}_{\op}^{1/2}\mathcal{N}^{1/2}_{1}(\lambda) \lesssim \sqrt{\log\frac{1}{\lambda}}$ (see \citealt[Lemma B.9]{kpca}). Thus, substituting this in the conditions from Theorem \ref{thm:typII-oracle} and using $n \geq \max\{e^2,\alpha^{-1/2}+\delta^{-1}\}$, we can write the separation boundary as 
\begin{equation} \label{eq:oracle-exp-1}
    \Delta_{n} \gtrsim 
\max\left\{ \left(\frac{\sqrt{2\tilde{\theta}}(\alpha^{-1/2}+\delta^{-1})^{-1}n}{\sqrt{\log n}}\right)^{-1},\left(\frac{\delta n}{\sqrt{\log n}}\right)^{-\frac{4\tilde{\theta}}{2\tilde{\theta}+1}} \right\},
\end{equation}
which is implied if 
$$\Delta_{n} =
\left\{
	\begin{array}{ll}
		c(\alpha,\delta, \theta)\frac{\sqrt{\log n}}{n}, &  \ \ \Tilde{\theta}> \frac{1}{2} \\
		c(\alpha,\delta,\theta)\left(\frac{\sqrt{\log n}}{n}\right)^{\frac{4\tilde{\theta}}{2\tilde{\theta}+1}}, & \ \ \Tilde{\theta} \le \frac{1}{2}
	\end{array}
\right.,$$
where $c(\alpha,\delta,\theta)\gtrsim \max\left\{\sqrt{\frac{1}{2\tilde{\theta}}},1\right\}(\alpha^{-1/2}+\delta^{-2})$ and we used that $\tilde{\theta}> \frac{1}{2} \Leftrightarrow 1 <  \frac{4\tilde{\theta}}{2\tilde{\theta}+1}.$

On the other hand when $C:=\sup_i\norm{\phi_i}_{\infty} < \infty$, we obtain
\begin{equation}
    \Delta_{n} \gtrsim 
\max\left\{ \sqrt{\frac{1}{2\tilde{\theta}}}\left(\frac{(\alpha^{-1/2}+\delta^{-1})^{-1}n}{\sqrt{\log n}}\right)^{-1},\frac{1}{2\tilde{\theta}}\left(\frac{\delta n}{\sqrt{\log n}}\right)^{-2}\right\},\nonumber
\end{equation}
which in turn is implied if 
$$\Delta_{n} = c(\alpha,\delta,\theta)\frac{\sqrt{\log n}}{n},$$
where $c(\alpha,\delta,\theta) \gtrsim \max\left\{\sqrt{\frac{1}{2\tilde{\theta}}},\frac{1}{2\tilde{\theta}},1\right\}(\alpha^{-1/2}+\delta^{-2})$. 

\subsection{Proof of Theorem \ref{thm: computation}}
The proof is exactly similar to that of \citep[Theorem 3]{twosampletest} by replacing $\Sigma_{PQ}$ with $\Sigma_0.$

\subsection{Proof of Theorem~\ref{thm:typI-Gamma}}
Since  $\E(\stat | (Y^0_i)_{i=1}^{s} )=0$, an application of Chebyshev's inequality via Lemma \ref{Lemma: bounding expectations} yields,  
\begin{equation*} 
        P_{H_0}\left\{|\stat| \geq \frac{\sqrt{6}\Cs\norm{\M}_{\op}^{2}\Ntl}{\sqrt{\delta}}\left(\frac{1}{n}+\frac{1}{m}\right) \Big| (Y^0_i)_{i=1}^{s}\right\} \leq \delta.
\end{equation*} 
Let $\gamma_1 : = \frac{2\sqrt{6}\Cs\Ntl}{\sqrt{\delta}}\left(\frac{1}{n}+\frac{1}{m}\right)$, and $\gamma_2 : = \frac{\sqrt{6}\Cs\norm{\M}_{\op}^{2}\Ntl}{\sqrt{\delta}}\left(\frac{1}{n}+\frac{1}{m}\right).$ Then 

\begin{align*}
  P_{H_0}\{\stat \leq \gamma_1\} & \geq P_{H_0}\{ \{\stat \leq \gamma_2\} \ \cap \{\gamma_2 \leq \gamma_1\}\} \\
  & \geq 1- P_{H_0}\{\stat \geq \gamma_2\} - P_{H_0}\{\gamma_2 \geq \gamma_1\} 
  \stackrel{(*)}{\geq} 1-3\delta,
\end{align*}
where $(*)$ follows using
\begin{align*}
P_{H_0}\{\stat \geq \gamma_2\} \leq P_{H_0}\{|\stat| \geq \gamma_2\} = \E_{P_0^m}\left[P_{H_0}\{|\stat| \geq \gamma_2| (Y^0_i)_{i=1}^{s}\}\right] \leq \delta,
\end{align*}
and 
\begin{align*}
    P_{H_0}\{\gamma_2 \geq \gamma_1\} = P_{H_0}\{\norm{\M}_{\op}^2 \geq 2\} \stackrel{(\dag)}{\leq} 2\delta, 
\end{align*}
where $(\dag)$ follows from \cite[Lemma B.2\emph{(ii)}]{kpca}, under the condition that $\frac{140\K}{s}\log \frac{16\K s}{\delta} \leq \lambda \leq \norm{\Sigma_0}_{\op}$. When $C:=\sup_i\norm{\phi_i}_{\infty} < \infty$, using \cite[Lemma A.17]{twosampletest}, we can obtain an improved condition on $\lambda$ satisfying $136C^2\Nol\log\frac{8\Nol}{\delta} \leq s$ and $\lambda \leq \norm{\Sigma_{0}}_{\op}.$ Thus setting $\delta = \frac{\alpha}{6},$ yields that 
\begin{equation} \label{Eq:gamma1}
    P_{H_0}\left\{\stat \geq \frac{12\Cs\Ntl}{\sqrt{\alpha}}\left( \frac{1}{n}+\frac{1}{m}\right)\right\} \leq \frac{\alpha}{2}.
\end{equation}
Finally, the desired result follows by writing
\begin{align*}
    &P_{H_0}\left\{\stat \leq \frac{12\Cs\Ntlh}{b_1\sqrt{\alpha}}\left( \frac{1}{n}+\frac{1}{m}\right)\right\} \\ 
    & \geq P_{H_0}\left\{\left\{\stat \leq \frac{12\Cs\Ntl}{\sqrt{\alpha}}\left( \frac{1}{n}+\frac{1}{m}\right)\right\} \ \cap \{\Ntlh \geq b_1 \Ntl\}\right\}\\
    & \geq 1-P_{H_0}\left\{\stat \geq \frac{12\Cs\Ntl}{\sqrt{\alpha}}\left( \frac{1}{n}+\frac{1}{m}\right)\right\} - P_{H_0}\{\Ntlh \leq b_1 \Ntl\} \\
    & \stackrel{(\ddag)}{\geq} 1- \alpha,
\end{align*}
where $(\ddag)$ follows using \eqref{Eq:gamma1} and Lemma \ref{lemma: ntlh lower bound} under the condition that $$\frac{4c_1\K}{s}\max\{\log\frac{96\K s}{\alpha},\log\frac{12}{\alpha}\} \leq \lambda \leq \norm{\Sigma_0}_{\op}.$$ The above condition can be replaced with $4c_1C^2\Nol\log\frac{48\Nol}{\alpha}\leq s$ if $C:=\sup_i\norm{\phi_i}_{\infty} < \infty$.

\subsection{Proof of Theorem \ref{thm:typII-Gamma}}
Let $\M = \SLh^{-1/2}\SL^{1/2}$ , $\gamma_1 =\frac{1}{\sqrt{\delta}} \left(\frac{\sqrt{\Cl} \norm{\U}_{\Lp}+\Ntl}{n}+\frac{\Cl^{1/4}\norm{\U}_{\Lp}^{3/2}+\norm{\U}_{\Lp}}{\sqrt{n}}\right),$ where $\Cl$ is defined in Lemma \ref{lemma: bound hs and op}. Then Lemma \ref{Lemma: bounding expectations} implies $\tilde{C}\norm{\M}_{\op}^2\gamma_1\geq\sqrt{\frac{\text{Var}(\stat|(Y_i^0)_{i=1}^s)}{\delta}}$ for some constant $\Tilde{C}>0$. By \citep[Lemma A.1]{twosampletest}, if \begin{equation}P\left\{\gamma \geq \zeta - \tilde{C}\norm{\M}_{\op}^2\gamma_1\right\} \leq \delta,\label{Eq:gamma2}\end{equation} for any $P \in \PP$, then we obtain $P\{\stat \geq \gamma\} \geq 1-2\delta.$ The result follows by taking the infimum over $P \in \PP$. Therefore, it remains to verify \eqref{Eq:gamma2}, which we do below. Define $d_{\lambda} : = \Ntl+\left(\frac{\sqrt{2}}{c_1}+\frac{1}{\sqrt{2c_1}}\right) \sqrt{\Nol},$ $\gamma_3 :=\frac{12d_{\lambda}\Cs}{b_1\sqrt{\alpha}}\left(\frac{1}{n}+\frac{1}{m}\right)$ and $c_2 :=B_3C_4 \Cs^{-1}$. Consider
\begin{align*}
    &P_{H_1}\left\{\gamma \leq \zeta - \Tilde{C}\norm{\M}_{\op}^2\gamma_1\right\} \\
    &\stackrel{(**)}{\geq} P_{H_1}\left\{\left\{\norm{\M}_{\op}^2\gamma_3 \leq c_2\norm{\M^{-1}}^{-2}_{\op} \norm{u}_{\Lp}^2 - \Tilde{C}\norm{\M}_{\op}^2\gamma_1\right\}\right.\\
    &\qquad\qquad\qquad\left.\cap \left\{\gamma \leq \norm{\M}_{\op}^2\gamma_3 \right\}\right\}\\ 
    & \geq  1- P_{H_1}\left\{\frac{\norm{\M}^2_{\op}\norm{\M^{-1}}^{2}_{\op}(\Tilde{C}\gamma_1+\gamma_3)}{c_2 \norm{u}_{\Lp}^2 }\geq 1\right\} - P_{H_1}\left\{\gamma \geq \norm{\M}_{\op}^2\gamma_3\right\}\\
    &\stackrel{(*)}{\geq} P_{H_1}\left\{\frac{\norm{\M}^2_{\op}\norm{\M^{-1}}^{2}_{\op}(\Tilde{C}\gamma_1+\gamma_3)}{c_2 \norm{u}_{\Lp}^2 }\leq 1\right\} -\delta \\
    & \stackrel{(\dag)}{\geq}1- P\left\{\left\{\norm{\M^{-1}}^{2}_{\op} \leq \frac{3}{2}\right\} \ \cap \left\{\norm{\M}^{2}_{\op} \leq 2\right\}\right\}-\delta\\
    & \geq 1-P\left\{\norm{\M^{-1}}^{2}_{\op} \geq \frac{3}{2}\right\} - P\left\{\norm{\M}^{2}_{\op} \geq 2\right\}-\delta \\
    & \stackrel{(\ddag)}{\geq} 1-2\delta,
\end{align*}
where $(**)$ follows by using $\zeta \geq  c_2 \norm{\M^{-1}}^{-2}_{\op} \norm{u}_{\Lp}^2$, which is obtained by combining (\citealt[Lemma A.11]{twosampletest} by replacing $\Sigma_{PQ}$ and $\mu_Q$ with $\Sigma_0$ and $\mu_0$, respectively) with Lemma \ref{lemma: bounds for eta}  under the assumptions of $u \in \text{\range}(\T^{\theta})$, and \eqref{Eq:lower}. Note that $u \in \text{\range}(\T^{\theta})$ is guaranteed since $P \in \PP$ and \eqref{Eq:verify-1} guarantees \eqref{Eq:lower} as discussed in the proof of Theorem \ref{thm:typII-oracle}. $(*)$ follows by Lemma \ref{lemma: ntlh upper bound} under the condition $s \geq 32c_1\K \lambda^{-1}\log(\max\{17920\K^2\lambda^{-1},6\}\delta^{-1})$ and $\norm{\Sigma_0}_{\op} \geq \lambda$ (when $C:=\sup_i\norm{\phi_i}_{\infty} < \infty$, the condition can be replaced by $s\geq 32c_1C^2\Nol\log\frac{6}{\delta}$). $(\dag)$ follows when
\begin{equation} 
 \norm{u}_{\Lp}^2 \geq \frac{3(\Tilde{C}\gamma_1+\gamma_3)}{c_2}.  
\label{eq:ver1}
\end{equation}
$(\ddag)$ follows from \citep[Lemma B.2\emph{(ii)}]{kpca} under the assumption that 
\begin{equation} 
    \frac{140\K}{s}\log \frac{64\K s}{\delta} \leq \lambda \leq \norm{\Sigma_0}_{\op},
  \label{eq:ver2}  
\end{equation}
which is implied by $s \geq 280\K \lambda^{-1}\log(17920\K^2\lambda^{-1}\delta^{-1}).$
When $C:=\sup_i\norm{\phi_i}_{\infty} < \infty$, $(\ddag)$ follows from \citep[Lemma A.17]{twosampletest} by replacing \eqref{eq:ver2} with 
\begin{equation}
136C^2\Nol\log\frac{32\Nol}{\delta} \leq s, \qquad \text{and}\qquad\lambda \leq \norm{\Sigma_{0}}_{\op}.\label{eq:ver2-2}
\end{equation}
Thus it remains only to verify \eqref{eq:ver1}. Using $m\geq n$ and $$\Ntl \leq \norm{\SgL\Sigma_{0}\SgL}_{\op}^{1/2}\mathcal{N}^{1/2}_{1}(\lambda),$$ it can be checked that \eqref{eq:ver1} is implied by $\Delta_{n} \geq \frac{r_1\mathcal{N}^{1/2}_{1}(\lambda)}{n\sqrt{\alpha}}$, $\Delta_{n} \geq \frac{r_2\Cl}{\delta^2 n^2}$ and $\Delta_{n}\geq\frac{r_3\mathcal{N}^{1/2}_{1}(\lambda)}{\delta n}$ for some constants $r_1,r_2,r_3>0$.

\subsection{Proof of Theorem \ref{thm: permutations typeI}}
The proof is exactly similar to that of \cite[Theorem 8]{twosampletest}.

\subsection{Proof of Theorem \ref{thm: permutations typeII}}
Let $\M = \SLh^{-1/2}\SL^{1/2}$ and $\gamma_1 =\frac{1}{\sqrt{\delta}} \left(\frac{\sqrt{\Cl} \norm{\U}_{\Lp}+\Ntl}{n}+\frac{\Cl^{1/4}\norm{\U}_{\Lp}^{3/2}+\norm{\U}_{\Lp}}{\sqrt{n}}\right).$
Then following the proof of Theorem 9 in  \citep{twosampletest}, Lemma \ref{lemma:bound quantile} along with $m\geq n$ yields that the power will be controlled to the desired level when $\Delta_{n} \geq \frac{r_1\Cl(\log(1/\alpha))^2}{\delta^2n^2}$ and $\Delta_{n}\geq\frac{r_2\Ntl\log(1/\alpha)}{\delta n}$ for some constants $r_1,r_2>0$, and under the condition \eqref{eq:ver2} which can be replaced by \eqref{eq:ver2-2} when $C:=\sup_i\norm{\phi_i}_{\infty} < \infty$.

\subsection{Proof of Corollary \ref{coro-perm:poly}}
The proof is similar to that of Corollary~\ref{coro:poly-oracle}. Since $\lambda_i \asymp i^{-\beta}$, we have $\Ntl 
\lesssim \lambda^{-1/2\beta}.$ By using this bound in the conditions of Theorem \ref{thm: permutations typeII}, we obtain that the conditions on $\Delta_{n}$ hold if 

\begin{equation} \label{eq:perm-poly-1}
  \Delta_{n} \gtrsim \max\left\{\left(\frac{\delta n}{\log(1/\alpha)}\right)^{-\frac{4\Tilde{\theta}\beta}{4\Tilde{\theta}\beta+1}}, \left(\frac{\delta n}{\log(1/\alpha)}\right)^{-\frac{8\Tilde{\theta}\beta}{4\Tilde{\theta}\beta+2\beta+1}}\right\}.  
\end{equation}
By exactly using the same arguments as in the proof of Corollary~\ref{coro:poly-oracle}, it is easy to verify that the above condition on $\Delta_n$ is implied if 
$$\Delta_{n} =
\left\{
	\begin{array}{ll}
		c(\alpha,\delta)n^{\frac{-4\tilde{\theta}\beta}{4\Tilde{\theta}\beta+1}},  &  \ \  \Tilde{\theta}> \frac{1}{2}-\frac{1}{4\beta} \\
		c(\alpha,\delta) n^{-\frac{8\Tilde{\theta}\beta}{4\Tilde{\theta}\beta+2\beta+1}}, & \ \  \Tilde{\theta} \leq \frac{1}{2}-\frac{1}{4\beta}
	\end{array}
\right.,$$ 
where $c(\alpha,\delta)\gtrsim \delta^{-2}(\log \frac{1}{\alpha})^2.$ On the other hand when $C:=\sup_i\norm{\phi_i}_{\infty} < \infty$, we obtain the corresponding condition as 
\begin{equation}\Delta_{n} \gtrsim \max\left\{\left(\frac{\delta n}{\log(1/\alpha)}\right)^{-\frac{4\Tilde{\theta}\beta}{4\Tilde{\theta}\beta+1}},\left(\frac{\delta n}{\log(1/\alpha)}\right)^{-\frac{4\Tilde{\theta}\beta}{2\Tilde{\theta}\beta+1}} \right\}, \label{eq:perm-poly-2}\end{equation}
which is implied if 
$$\Delta_n = c(\alpha,\delta)n^{\frac{-4\tilde{\theta}\beta}{4\Tilde{\theta}\beta+1}}.$$

\subsection{Proof of Corollary \ref{coro-perm:exp}}
The proof is similar to that of Corollary~\ref{coro:exp-oracle}. When $\lambda_i \asymp e^{-\tau i}$, we have $\Ntl \lesssim \sqrt{\log\frac{1}{\lambda}}$. Thus substituting this in the conditions from Theorem \ref{thm: permutations typeII} and assuming that  $$n \geq \max\{e^2,\delta^{-1}(\log 1/\alpha)\},$$ we can write the separation boundary as

\begin{equation} \label{eq:perm-exp-1}
    \Delta_{n} \gtrsim 
\max\left\{ \left(\frac{\sqrt{2\tilde{\theta}}\delta n}{\log(1/\alpha)\sqrt{\log n}}\right)^{-1},\left(\frac{\delta n\log(1/\alpha)^{-1}}{\sqrt{\log n}}\right)^{-\frac{4\tilde{\theta}}{2\tilde{\theta}+1}} \right\},
\end{equation}
which is implied if $$\Delta_{n} =
\left\{
	\begin{array}{ll}
		c(\alpha,\delta, \theta)\frac{\sqrt{\log n}}{n}, &  \ \ \Tilde{\theta}> \frac{1}{2} \\
		c(\alpha,\delta,\theta)\left(\frac{\sqrt{\log n}}{n}\right)^{\frac{4\tilde{\theta}}{2\tilde{\theta}+1}}, & \ \ \Tilde{\theta} \le \frac{1}{2}
	\end{array}
\right.,$$
where $c(\alpha,\delta,\theta)\gtrsim \max\left\{\sqrt{\frac{1}{2\tilde{\theta}}},1\right\}\delta^{-2}(\log \frac{1}{\alpha})^2$.

On the other hand when $C:=\sup_i\norm{\phi_i}_{\infty} < \infty$, we obtain
\begin{equation} \label{eq:perm-exp-2}
    \Delta_{n} \gtrsim 
\max\left\{ \left(\frac{\sqrt{2\tilde{\theta}}\delta n}{\log(1/\alpha)\sqrt{\log n}}\right)^{-1},\frac{1}{2\tilde{\theta}}\left(\frac{\delta n}{\log(1/\alpha)\sqrt{\log n}}\right)^{-2}\right\},
\end{equation}
which in turn is implied if 
$$\Delta_{n} = c(\alpha,\delta,\theta)\frac{\sqrt{\log n}}{n},$$
where $c(\alpha,\delta,\theta) \gtrsim \max\left\{\sqrt{\frac{1}{2\tilde{\theta}}},\frac{1}{2\tilde{\theta}},1\right\}\delta^{-2}(\log \frac{1}{\alpha})^2$. 
\subsection{Proof of Theorem \ref{thm: adp-gamma typeI}}
First note that the following two events,
$$A:=\bigcup_{\lambda \in \Lambda} \left\{\stat \geq \gamma(\Tilde{\alpha},\lambda)\right\},$$ and  
$$B := \sup_{\lambda \in \Lambda} \frac{\stat}{\Ntlh} \geq \frac{12\Cs}{b_1\sqrt{\Tilde{\alpha}}}\left( \frac{1}{n}+\frac{1}{m}\right)$$ are equivalent, where $\gamma(\alpha,\lambda) = \frac{12\Cs\Ntlh}{b_1\sqrt{\alpha}}\left( \frac{1}{n}+\frac{1}{m}\right)$. The proof therefore follows from Theorem \ref{thm:typI-Gamma} and \cite[Lemma A.16]{twosampletest}.

\subsection{Proof Theorem \ref{thm: adp-gamma TypeII}}

The same steps as in the proof of Theorem \ref{thm:typII-Gamma} will follow, with the only difference being $\alpha$ is replaced by $\frac{\alpha}{\cd}$, where $\cd = 1+\log_2\frac{\lambda_U}{\lambda_L}\lesssim \log(n).$

For the case of $\lambda_i \asymp i^{-\beta}$, we can deduce from the proof of Corollary~\ref{coro:poly-oracle} (see~\eqref{eq:oracle-poly-1}) that when $\lambda = d_3^{-1/2\tilde{\theta}} \Delta_{N,M}^{1/2\Tilde{\theta}}$ for some $d_3>0$, then 
\begin{align*}
    P_{H_1}\left\{\stat \geq \tilde{\gamma} \right\} \geq 1-4\delta,
\end{align*}
where $\tilde{\gamma}= \frac{12\Ntlh\sqrt{\cd}\Cs}{b_1\sqrt{\alpha}}\left( \frac{1}{n}+\frac{1}{m}\right)$
and the condition on the separation boundary becomes
$$\Delta_{n} \gtrsim \max\left\{\left(\frac{n}{\tilde{\alpha}^{-1/2}+\delta^{-1}}\right)^{-\frac{4\Tilde{\theta}\beta}{4\Tilde{\theta}\beta+1}}, (\delta n)^{-\frac{8\Tilde{\theta}\beta}{4\Tilde{\theta}\beta+2\beta+1}}\right\},$$
where $\tilde{\alpha}=\frac{\alpha}{\cd}$. In turn, this is implied if
$$\Delta_{n} =c(\alpha,\delta) \max\left\{\left(\frac{n}{\sqrt{\log n}}\right)^{-\frac{4\Tilde{\theta}\beta}{4\Tilde{\theta}\beta+1}}, n^{-\frac{8\Tilde{\theta}\beta}{4\Tilde{\theta}\beta+2\beta+1}}\right\},$$
where $c(\alpha,\delta)\gtrsim(\alpha^{-1/2}+\delta^{-2})$. Note that the optimal choice of $\lambda$ is given by $$\lambda=\lambda^*:=d_3^{-1/2\tilde{\theta}}c(\alpha,\delta,\theta)^{1/2\tilde{\theta}} \max\left\{\left(\frac{n}{\sqrt{\log n}}\right)^{-\frac{2\beta}{4\Tilde{\theta}\beta+1}}, n^{-\frac{4\beta}{4\Tilde{\theta}\beta+2\beta+1}}\right\}.$$
Thus it can be verified that for any $\theta$ and $\beta$, the optimal lambda can be bounded as $$r_1 n^\frac{-4\beta_U}{1+2\beta_U}\leq \lambda \leq r_2 \left(\frac{n}{\sqrt{\log n}} \right)^\frac{-2}{4\Tilde{\xi}+1}$$ for some constants $r_1$,$r_2>0$.
\par 
$\blacklozenge$ Define $v^* := \sup\{x \in \Lambda: x \leq \lambda^*\}$. From the definition of $\Lambda$, it is easy to see that $\lambda_L\leq \lambda^* \leq \lambda_U$ and $\frac{\lambda^*}{2}\leq v^* \leq \lambda^*$. Thus $v^*\in\Lambda$ is an optimal choice of $\lambda$ that will yield the same form of the separation boundary up to constants. Therefore, by \cite[Lemma A.16]{twosampletest}, for any $\theta$ and any $P$ in $\PP$, we have 
\begin{align*}
P_{H_1}\left\{\sup_{\lambda \in \Lambda} \frac{\stat}{\Ntlh} \geq \gamma \right\} \geq 1-4\delta.
\end{align*}
Thus the desired result holds by taking the infimum over $P \in \PP$ and $\theta$. $\blacklozenge$

When $\lambda_i \asymp i^{-\beta}$ and $C:=\sup_i\norm{\phi_i}_{\infty} < \infty$, then using \eqref{eq:oracle-poly-2}, the conditions on the separation boundary becomes

$$\Delta_n = c(\alpha,\delta)\left(\frac{n}{\sqrt{\log n}}\right)^{\frac{-4\tilde{\theta}\beta}{4\Tilde{\theta}\beta+1}},$$
where $c(\alpha,\delta)\gtrsim(\alpha^{-1/2}+\delta^{-2}).$ This yields the optimal $\lambda$ to be 
$$\lambda^*:=d_3^{-1/2\tilde{\theta}}c(\alpha,\delta,\theta)^{1/2\tilde{\theta}}\left(\frac{n}{\sqrt{\log n}}\right)^{\frac{-2\beta}{4\Tilde{\theta}\beta+1}}.$$ Then as in the previous case we deduce that for any $\theta$ and $\beta$, $$r_3\left(\frac{n}{\sqrt{\log n}} \right)^\frac{-2\beta_U}{4\theta_l\beta_U+1}\leq\lambda \leq r_4\left(\frac{n}{\sqrt{\log n}} \right)^\frac{-2}{4\xi+1}$$ for some constants $r_3,r_4>0.$ The claim therefore follows by using the argument mentioned between $\blacklozenge$ and $\blacklozenge$.

For the case $\lambda_i \asymp e^{-\tau i}$, $\tau>0$, the condition on the separation boundary from \eqref{eq:oracle-exp-1} becomes 
$$\Delta_{n} \gtrsim 
c(\alpha,\delta,\theta)\max\left\{ \left(\frac{n}{\log n}\right)^{-1},\left(\frac{ n}{\sqrt{\log n}}\right)^{-\frac{4\tilde{\theta}}{2\tilde{\theta}+1}} \right\},$$
where $c(\alpha,\delta,\theta)\gtrsim \max\left\{\sqrt{\frac{1}{2\tilde{\theta}}},1\right\}(\alpha^{-1/2}+\delta^{-2}).$ 
Thus $$\lambda^* = d_3^{-1/2\tilde{\theta}}c(\alpha,\delta,\theta)^{1/2\tilde{\theta}}\max\left\{ \left(\frac{n}{\log n}\right)^{-1/2\tilde{\theta}},\left(\frac{ n}{\sqrt{\log n}}\right)^{-\frac{2}{2\tilde{\theta}+1}} \right\},$$ which can be bounded as $r_5 \left(\frac{n}{\sqrt{\log n}}\right)^{-2} \leq \lambda \leq r_6 \left(\frac{n}{\log n}\right)^{-1/2\tilde{\xi}}$ for some $r_5,r_6>0.$
 Furthermore  when $C:=\sup_i\norm{\phi_i}_{\infty} < \infty$, the condition on the separation boundary becomes $$\Delta_{n} = c(\alpha,\delta,\theta)\frac{\log n}{n},$$
where $c(\alpha,\delta,\theta) \gtrsim \max\left\{\sqrt{\frac{1}{2\tilde{\theta}}},\frac{1}{2\tilde{\theta}},1\right\}(\alpha^{-1/2}+\delta^{-2})$. Thus 
$$\lambda^*=d_3^{-1/2\tilde{\theta}}c(\alpha,\delta,\theta)^{1/2\tilde{\theta}}\left(\frac{n}{\log n}\right)^{-1/2\tilde{\theta}},$$ which can be bounded by $r_7\left(\frac{n}{\log n}\right)^{-1/2\theta_l}\leq \lambda \leq r_8\left(\frac{n}{\log n}\right)^{-1/2\xi}.$
The claim, therefore, follows by using the same argument as mentioned in the polynomial decay case.

\subsection{Proof of Theorem~\ref{thm:perm adp typeI}}
The proof follows from Theorem \ref{thm: permutations typeI} and \cite[Lemma A.16]{twosampletest} by using $\frac{\alpha}{\cd}$ in the place of $\alpha$.

\subsection{Proof of Theorem \ref{thm:perm adp TypeII}}

The proof follows from Theorem \ref{thm: permutations typeII} using $\frac{\alpha}{\cd}$ instead of $\alpha$ in the expressions \eqref{eq:perm-poly-1}, \eqref{eq:perm-poly-2}, \eqref{eq:perm-exp-1} and \eqref{eq:perm-exp-2}, and then bounding the expressions for the resulting optimal $\lambda^*$ using the ideas similar to that used in the proof of Theorem \ref{thm: adp-gamma TypeII}.

\section*{Acknowledgments}
OH and BKS are partially supported by National Science Foundation (NSF) CAREER
award DMS-1945396. BL is supported by NSF grant DMS-2210775.

\bibliographystyle{plainnat} 
\bibliography{main} 

\setcounter{section}{0}
\renewcommand\thesection{\Alph{section}}
\newtheorem{thmm}{Theorem}[section]
\newtheorem{lem}{Lemma}[section]
\section{Technical results}
In the following, we present technical results that are used to prove the main results of the paper.  Unless
specified otherwise, the notation used in this section matches that of the main paper.
\begin{lem} \label{lemma:bound U-statistic3}
Let $(X_i)_{i=1}^n \stackrel{i.i.d}{\sim} Q , \ (Y_i)_{i=1}^m \stackrel{i.i.d}{\sim} P$ and $\B: \h \to \h$ be a bounded operator. Define
$$I=\frac{2}{nm}\sum_{i,j}\inner{ a(X_i)}{b(Y_j)}_{\h},$$ where $a(x)= \B\Sigma_{0,\lambda}^{-1/2}(\kk(\cdot,x)-\mu_Q)$, $b(x)= \B\Sigma_{0,\lambda}^{-1/2}(\kk(\cdot,x)-\mu_P)$, $\mu_Q = \int_{\X} \kk(\cdot,x)\,dQ(x)$ and $\mu_P = \int_{\Y} \kk(\cdot,y)\,dP(y)$. Then
\begin{enumerate}[label=(\roman*)]
\item $\E \inner{ a(X_i)}{b(Y_j)}_{\h}^2 \leq \norm{\B}_{\op}^4\norm{\SgL\Sigma_Q\SgL}_{\hs}\norm{\SgL\Sigma_P\SgL}_{\hs}$;
\item $\E\left(I^2\right) \leq  \frac{4}{nm}\norm{\B}_{\op}^4\norm{\SgL\Sigma_Q\SgL}_{\hs}\norm{\SgL\Sigma_P\SgL}_{\hs}.$
\end{enumerate}
\end{lem}
\begin{proof}
\emph{(i)} Note that
\begin{eqnarray*} 
    \E \inner{ a(X_i)}{b(Y_j)}_{\h}^2 
    &{} ={}& \E \inner{ a(X_i)\htens a(X_i)}{b(Y_j) \htens b(Y_j)}_{\hs} \\
    &{} ={}&  \inner{ \B\SgL\Sigma_Q\SgL\B^*}{\B\SgL\Sigma_P\SgL\B^*}_{\hs} \\
    &{} \leq{}& \norm{\B}_{\op}^4\norm{\SgL\Sigma_Q\SgL}_{\hs}\norm{\SgL\Sigma_P\SgL}_{\hs}.
\end{eqnarray*}
\\
\emph{(ii)} follows by noting that
\begin{align*}
    \E\left(I^2\right) \stackrel{(\dag)}{=} \frac{4}{n^2m^2}\sum_{i,j} \E\inner{ a(X_i)}{b(Y_j)}_{\h}^2
\end{align*}
where $(\dag)$ follows from \citep[Lemma A.3 \emph{(i)}]{twosampletest}, and the result follows from $(i)$.
\end{proof}

\begin{lem} \label{lemma: bounds for eta}
Let $u=\frac{dP}{d\PQ}-1  \in \Lp$ and $\eta= \norm{\gSL(\mu_0-\mu_P)}_{\h}^{2}$, where $g_{\lambda}$ satisfies $(A_1)$--$(A_4)$. Then
$$\eta \leq C_1 \norm{\U}_{\Lp}^{2}.$$ 
Furthermore, if $u \in \emph{\range}(\T^{\theta})$, $\theta>0$ and $$\norm{\U}_{\Lp}^2 \geq  \frac{4C_3}{3B_3} \norm{\T}_{\opl}^{2\max(\theta-\xi,0)}\lambda^{2 \Tilde{\theta}} \norm{\T^{-\theta}\U}_{\Lp}^2,$$  where $\Tilde{\theta}= \min(\theta,\xi)$, then, $$\eta \geq \frac{B_3}{4} \norm{\U}_{\Lp}^2.$$
\end{lem}
\begin{proof}
    The proof uses the same approach as in the proof of \citep[Lemma A.7]{twosampletest} by noting that $\eta=\inner{\T\gT \U}{\U}_{\Lp}$ and involves  replacing $\mu_Q$, $R$ and $\Sigma_{PQ}$ with $\mu_0$, $P_0$ and $\Sigma_0$, respectively.
\end{proof}

\begin{lem} \label{lemma: bound hs and op}
Define $\Nol := \emph{Tr}(\SgL\Sigma_{0} \SgL)$,
$\Ntl := \norm{\SgL\Sigma_{0}\SgL}_{\hs}$, and $u := \frac{dP}{d\PQ}-1 \in L^2(P_0)$. Then the following hold:\vspace{-2mm}
\begin{enumerate}[label=(\roman*)]
    \item $\norm{\SgL\Sigma_P\SgL}_{\hs}^2 \leq 4 \Cl \norm{u}_{\Lp}^2  + 2\Ntlsq;$\vspace{-1mm}
    \item 
    $\norm{\SgL\Sigma_P\SgL}_{\op} \leq 2 \sqrt{\Cl} \norm{u}_{\Lp} + 1,$
    \end{enumerate}\vspace{-1mm}
where $$\Cl =
\left\{
	\begin{array}{ll}
		\Nol \sup_{i} \norm{\phi_i}^2_{\infty},   &  \ \ \sup_i \norm{\phi_i}^2_{\infty} < \infty \\
		\frac{2\Ntl}{\lambda} \sup_x \norm{K(\cdot,x)}_{\h}^2,  & \ \  \text{otherwise}
	\end{array}
\right..$$    
\end{lem}
\begin{proof}
The proof is similar to that of \citep[Lemma A.9]{twosampletest} and involves replacing $R$ with $P_0$ and $\Sigma_{PQ}$ with $\Sigma_0$. 

\end{proof}
\begin{lem} \label{lemma:adaptation Upper}
For any $0<\delta<\frac{1}{2}$,
$$P_0^s\left\{(Y_i)_{i=1}^s: \norm{\SgL(\hat{\Sigma}_0-\Sigma_0)\SgL}_{\hs}\leq \frac{32\K\log\frac{3}{\delta}}{\lambda s}+\sqrt{\frac{16\K\Nol\log\frac{2}{\delta}}{\lambda s}}\right\} \geq 1-2\delta.$$
Furthermore, suppose  $C:=\sup_{i}\norm{\phi_i}_{\infty} < \infty$. Then 
\begin{align*}
&P_0^s\left\{(Y_i)_{i=1}^s: \norm{\SgL(\hat{\Sigma}_0-\Sigma_0)\SgL}_{\hs}\leq \frac{32 C^2 \Nol\log\frac{3}{\delta}}{s}+\sqrt{\frac{16C^2\Nolsq \log\frac{2}{\delta}}{s}}\right\}\\
&\qquad\qquad\qquad\qquad\geq 1-2\delta.\end{align*}
\end{lem}
\begin{proof}
Define $s(x):=K(\cdot,x)$, $A(x,y):=\frac{1}{\sqrt{2}}(s(x)-s(y))$, $U(x,y) := \SgL A(x,y)$, and $Z(x,y)= U(x,y) \htens U(x,y)$. Then
$$\SgL(\hat{\Sigma}_0-\Sigma_0)\SgL = \frac{1}{s(s-1)}\sum_{i\neq j}Z(Y_i,Y_j)-\E(Z(X,Y)). $$
Also,
$$\sup_{x,y}\norm{Z(x,y)}_{\hs} = \sup_{x,y} \norm{U(x,y)}_{\h}^2 = \frac{1}{2} \sup_{x,y}\norm{\SgL(s(x)-s(y))}_{\h}^2 \leq \frac{2\K}{\lambda}.$$
Define $\zeta(x) := \E_{Y}[Z(x,Y)]$. Then 
\begin{align*}
    &\E \norm{\zeta(X)-\Sigma_0}_{\hs}^2 \leq \E \norm{\zeta(X)}_{\hs}^2 = \E \norm{\SgL \E_{Y}[A(X,Y)\htens A(X,Y)]\SgL}_{\hs}^2 \\
    & = \E \text{Tr} \left(\SgL \E_{Y}[A(X,Y)\htens A(X,Y)] \SL^{-1} \E_{Y}[A(X,Y)\htens A(X,Y)] \SgL \right) \\
    &\le \sup_{x} \norm{\zeta(x)}_{\op}\text{Tr}(\SgL\Sigma_0\SgL) \\
    & \leq \sup_{x,y} \norm{U(x,y)}_{\h}^2\Nol  \leq \frac{2\K\Nol}{\lambda}. 
\end{align*}
When $C:=\sup_{i}\norm{\phi_i}_{\infty} < \infty$, we can use the same approach as in the proof of \cite[Lemma A.17]{twosampletest} to show that $\sup_{x,y} \norm{U(x,y)}_{\h}^2 \leq 2C^2 \Nol$ which in turn yields that 
$$\sup_{x,y}\norm{Z(x,y)}_{\hs}  \leq 2C^2 \Nol,$$
and 
$$\E \norm{\zeta(X)-\Sigma_0}_{\hs}^2 \leq 2C^2 \Nolsq.$$
Then the result follows from \citep[Theorem D.3\emph{(ii)}]{kpca}.
\end{proof}

\begin{lem} \label{lemma:adaptation lower}
Let $I =\inner{\SgL\Sigma_0\SgL}{\SgL(\hat{\Sigma}_0-\Sigma_0)\SgL}_{\hs}$. Then for any $\delta>0$,   
\begin{align*}
P_0^s\left\{(Y_i)_{i=1}^s:|I| \leq \frac{4\K\log\frac{2}{\delta}}{\lambda s}+\sqrt{\frac{12\K\Ntlsq\log\frac{2}{\delta}}{\lambda s}}\right\} \geq 1-\delta.
\end{align*}
Furthermore, suppose  $C:=\sup_{i}\norm{\phi_i}_{\infty} < \infty$. Then 
\begin{align*}
P_0^s\left\{(Y_i)_{i=1}^s:|I| \leq \frac{4C^2\Nol\log\frac{2}{\delta}}{ s}+\sqrt{\frac{12C^2\Nol\Ntlsq\log\frac{2}{\delta}}{ s}}\right\} \geq 1-\delta.
\end{align*}
\end{lem}
\begin{proof}
Define $s(x):=K(\cdot,x)$, $A(x,y):=\frac{1}{\sqrt{2}}(s(x)-s(y))$, $U(x,y) := \SgL A(x,y)$, $Z(x,y):= U(x,y) \htens U(x,y)$, $B:=\E Z(X,Y)= \SgL\Sigma_0\SgL$, and $$\Tilde{Z}(x,y):= \inner{B}{Z(x,y)}_{\hs}.$$ Then
$$I = \frac{1}{s(s-1)}\sum_{i\neq j}\Tilde{Z}(Y_i,Y_j)-\E\Tilde{Z}(X,Y).$$
Moreover,
\begin{align*}
&\sup_{x,y}|\Tilde{Z}(x,y)|=\sup_{x,y}|\inner{B}{Z(x,y)}_{\hs}|\\
& \qquad = \sup_{x,y}|\text{Tr}(B (U(x,y) \htens U(x,y)))| \leq \norm{B}_{\op}\sup_{x,y}\norm{U(x,y)}_{\h}^2 \leq \frac{2\K}{\lambda},
\end{align*}
and
\begin{align*}
    \E\Tilde{Z}^{2}(X,Y) \stackrel{(*)}{\leq} \sup_{x,y}|\Tilde{Z}(x,y)|\E \inner{B}{Z(X,Y)}_{\hs} \leq \frac{2\K\Ntlsq}{\lambda}, 
\end{align*}
where $(*)$ follows by using $\Tilde{Z}(x,y) \geq 0$, which can be shown by writing,
\begin{align*}
    &\Tilde{Z}(x,y)=\inner{B}{Z(x,y)} = \text{Tr}(B(U(x,y)\htens U(x,y))) \\
    &=\text{Tr}\left(\SgL \Sigma_0 \SgL [U(x,y)\htens U(x,y)]\right)\\
    &=\text{Tr}(\Sigma_0^{1/2}\SgL (U(x,y)\htens U(x,y)) \SgL \Sigma_0^{1/2}) \\
    &= \norm{\Sigma_0^{1/2}\SgL U(x,y)}_{\h}^2 \geq 0.
\end{align*}
When $C:=\sup_{i}\norm{\phi_i}_{\infty} < \infty$, we can use the same approach as in \citep[Lemma A.17]{twosampletest} to show that $\sup_{x,y} \norm{U(x,y)}_{\h}^2 \leq 2C^2 \Nol$ which in turn yields that
$$\sup_{x,y}|\Tilde{Z}(x,y)| \leq 2C^2 \Nol,$$ and 
$$\E\Tilde{Z}^{2}(X,Y)  \leq 2C^2\Nol\Ntlsq.$$
Thus the result follows by using Hoeffding's inequality as stated in \citep[Theorem 4.1.8]{pena}.
\end{proof}

\begin{lem}\label{lemma: ntlh upper bound}
For any $c_1>0$, $\delta>0$ and $\frac{32c_1\K}{s}\log\frac{3}{\delta} \leq \lambda \leq \norm{\Sigma_0}_{\op}$, we have 
\begin{align*}
    P_0^s\left\{(Y_i)_{i=0}^s: \Ntlh \leq \norm{\M}_{\op}^2\left(\Ntl+\left(\frac{\sqrt{2}}{c_1}+\frac{1}{\sqrt{2c_1}}\right) \sqrt{\Nol}\right)  \right\} \geq 1-2\delta.
\end{align*}
Furthermore if $C:=\sup_i\norm{\phi_i}_{\infty} < \infty$, the above bound holds for $32c_1C^2\Nol\log\frac{3}{\delta} \leq s$ and $\lambda \leq \norm{\Sigma_0}_{\op}.$
\end{lem}
\begin{proof}
Let $\M = \SLh^{-1/2}\SL^{1/2}$. Then
\begin{align*}
    \Ntlh &\leq \norm{\M}_{\op}^2 \norm{\SgL\hat{\Sigma}_0\SgL}_{\hs}\\
    & \leq \norm{\M}_{\op}^2 \left(\Ntl + \norm{\SgL(\hat{\Sigma}_0-\Sigma_0)\SgL}_{\hs} \right).
\end{align*}
From Lemma \ref{lemma:adaptation Upper} and the assumption that $\frac{32c_1\K}{s}\log\frac{3}{\delta} \leq \lambda$, we have with probability at least $1-2\delta$ that,
\begin{align*}
  \norm{\SgL(\hat{\Sigma}_0-\Sigma_0)\SgL}_{\hs}&\leq \frac{1}{c_1}+ \frac{1}{\sqrt{2c_1}}\sqrt{\Nol} \\ 
  & \stackrel{(*)}{\leq} \left(\frac{\sqrt{2}}{c_1}+ \frac{1}{\sqrt{2c_1}}\right)\sqrt{\Nol},
\end{align*}
where in $(*)$ we used $\sqrt{\Nol}= \sqrt{\sum_{i}\frac{\lambda_i}{\lambda_i+\lambda}} \geq \sqrt{\frac{\norm{\Sigma_0}_{\op}}{\norm{\Sigma_0}_{\op}+\lambda}} \stackrel{(\dag)}{\geq} \frac{1}{\sqrt{2}}$, and $(\dag)$ follows from $\lambda \leq \norm{\Sigma_0}_{\op}$. Similarly when $C:=\sup_i\norm{\phi_i}_{\infty} < \infty$, the same bound holds by Lemma \ref{lemma:adaptation Upper} for $32c_1C^2\Nol\log\frac{3}{\delta} \leq s.$
\end{proof}
\begin{lem} \label{lemma: ntlh lower bound}
For any $c_1>0$, $\delta>0$, and $\max\{\frac{140\K}{s}\log \frac{16\K s}{\delta},\frac{4c_1\K}{s}\log\frac{2}{\delta}\}\leq \lambda \leq \norm{\Sigma_0}_{\op}$, we have
$$P_0^s\left\{(Y_i)_{i=0}^s: \Ntlsqh \geq \left(\frac{4}{9}-\frac{16}{3\sqrt{3c_1}}-\frac{32}{9c_1}\right)\Ntlsq  \right\} \geq 1-3\delta.$$  
Furthermore if $C:=\sup_i\norm{\phi_i}_{\infty} < \infty$, the above bound holds for $\lambda \leq \norm{\Sigma_0}_{\op},$ and \\ $C^2\Nol\max\{4c_1\log\frac{2}{\delta},136\log \frac{8\Nol}{\delta}\} \leq s.$
\end{lem}
\begin{proof}
Let $\M := \SLh^{-1/2}\SL^{1/2}$ and $I :=\inner{\SgL\Sigma_0\SgL}{\SgL(\hat{\Sigma}_0-\Sigma_0)\SgL}_{\hs}$. Then 
\begin{align*}
    \Ntlsqh &\geq \frac{1}{\norm{\M^{-1}}_{\op}^4}\norm{\SgL\hat{\Sigma}_0\SgL}_{\hs}^2 \\
    & = \frac{1}{\norm{\M^{-1}}_{\op}^4}\left(\norm{\SgL \Sigma_0\SgL}_{\hs}^2+2I+\norm{\SgL(\hat{\Sigma}_0-\Sigma_0)\SgL}_{\hs}^2\right) \\
    &  \geq \frac{1}{\norm{\M^{-1}}_{\op}^4}(\Ntlsq-2|I|).
\end{align*}
Then from Lemma \ref{lemma:adaptation lower} and the assumption $\frac{4\K c_1}{s}\log\frac{2}{\delta} \leq \lambda$ (similarly when $C:=\sup_i\norm{\phi_i}_{\infty} < \infty$ the same bound holds by Lemma \ref{lemma:adaptation lower} for $32c_1C^2\Nol\log\frac{3}{\delta} \leq s$), we have with probability at least $1-\delta$,
$$|I| \leq \frac{1}{c_1}+\sqrt{\frac{3}{c_1}}\Ntl \leq \left( \sqrt{\frac{3}{c_1}}+\frac{2}{c_1}\right)\Ntl\leq2\left( \sqrt{\frac{3}{c_1}}+\frac{2}{c_1}\right)\Ntlsq,$$ where in the last two inequalities we used $\Ntl \geq \frac{\norm{\Sigma_0}_{\op}}{\norm{\Sigma_0}_{\op}+\lambda} \geq \frac{1}{2}$ when $\lambda \leq \norm{\Sigma_0}_{\op}$.

Define
$$S_1:=\left\{(Y_i)_{i=1}^s:|I| \leq 2\left( \sqrt{\frac{3}{c_1}}+\frac{2}{c_1}\right)\Ntlsq \right\}$$
and
$$S_2 := \left\{(Y_i)_{i=1}^s:\norm{\M^{-1}}_{\op}^4 \leq \frac{9}{4} \right\}.$$ Then, 
\begin{align*}
    P_0^s\left\{(Y_i)_{i=0}^s: \Ntlsqh \geq \left(\frac{4}{9}-\frac{16}{3\sqrt{3c_1}}-\frac{32}{9c_1}\right)\Ntlsq  \right\} &\geq P(S_1 \cap S_2)\\
    & \geq 1-P(S_1^{'})-P(S_2^{'})\geq 1-3\delta,
\end{align*}
where we used \citep[Lemma B.2 \emph{(iii)}]{kpca} in the last inequality, with $S^{'}$ being the complement of set $S$. Similarly, when $C:=\sup_i\norm{\phi_i}_{\infty} < \infty$ the same bound holds using \citep[Lemma A.17 \emph{(iii)}]{twosampletest}.
\end{proof}

\begin{lem} \label{Lemma: bounding expectations}
Let $\zeta = \norm{\gSLh(\mu_P-\mu_0)}_{\h}^2$, $\M = \SLh^{-1/2}\SL^{1/2}$, and $m \geq n$ . Then
\begin{align*}
    &\E \left[(\stat-\zeta)^2 |(Y_i^0)_{i=1}^{s}\right] \\
    & \quad \leq \tilde{C} \norm{\M}_{\op}^4 \left\{\frac{\Cl \norm{\U}_{\Lp}^2+\Ntlsq}{n^2}+\frac{\sqrt{\Cl}\norm{\U}_{\Lp}^3+\norm{\U}_{\Lp}^2}{n}\right\},
\end{align*}
where $\Cl$ is defined in Lemma~\ref{lemma: bound hs and op} and $\tilde{C}$ is a constant that depends only on $C_1$ and $C_2$. Furthermore, if $P=P_0$, then
\begin{align*}
    &\E\left[\left(\stat\right)^{2}|(Y_i^0)_{i=1}^{s}\right] \leq 6 \Cs^2\norm{\M}_{\op}^4\Ntlsq\left(\frac{1}{m^2}+\frac{1}{n^2}\right).
\end{align*}
\end{lem}
\begin{proof}
Define $a(x)= \B\SgL(\kk(\cdot,x)-\mu_P)$, and $b(x)=\B\SgL(\kk(\cdot,x)-\mu_0)$, where $\B= \gSLh \SL^{1/2}$. Then replacing $\Sigma_{PQ}$ by $\Sigma_0$ and $\mu_Q$ by $\mu_0$ in the proof of \cite[Lemma A.12]{twosampletest}, it can be shown that
\begin{multline*}
    \stat-\zeta = \underbrace{\frac{1}{n(n-1)} \sum_{i\neq j} \left \langle a(X_i),a(X_j)\right \rangle_{\h}}_{\footnotesize{\circled{1}}} + \underbrace{ \frac{1}{m(m-1)}\sum_{i \neq j} \inner{b(X_i^0)}{b(X_j^0)}_{\h}}_{\footnotesize{\circled{2}}}  \\   +\underbrace{\frac{2}{m}\sum_{i=1}^{m} \inner{b(X_i^0)}{\B\SgL(\mu_0-\mu_P)}_{\h}}_{\footnotesize{\circled{3}}} - \underbrace{\frac{2}{nm}\sum_{i,j} \left \langle a(X_i),b(X_j^0)\right \rangle_{\h}}_{\footnotesize{\circled{4}}} \\
    - \underbrace{\frac{2}{n}\sum_{i=1}^{n} \inner{a(X_i)}{\B\SgL(\mu_0-\mu_P)}_{\h}}_{\footnotesize{\circled{5}}},
\end{multline*}
and 
\begin{eqnarray*}
    \norm{\B}_{\op} &{} \leq{}& \Cs^{1/2} \norm{\M}_{\op}.
\end{eqnarray*}
Next, we bound each of these terms using Lemmas  \ref{lemma:bound U-statistic3}, \ref{lemma: bound hs and op} and \citep[Lemma A.4, Lemma A.5]{twosampletest}. It follows from Lemma~\ref{lemma: bound hs and op}\emph{(i)} and \citep[Lemma A.4\emph{(ii)}]{twosampletest} that
\begin{align*}
        \E\left(\footnotesize{\circled{1}}^2 |(Y_i^0)_{i=1}^{s}\right) & \leq \frac{4}{n^2}\|\B\|_{\op}^4\norm{\Sigma_{0,\lambda}^{-1/2}\Sigma_P\SgL}_{\hs}^2 \\
        &\leq \frac{4}{n^2}\Cs^2\norm{\M}_{\op}^4 \left(4\Cl \norm{\U}_{\Lp}^{2}+2 \Ntlsq\right),
\end{align*}
and 
\begin{align*}
        \E\left(\footnotesize{\circled{2}}^2 |(Y_i^0)_{i=1}^{s}\right)
        &\leq \frac{4}{m^2}\Cs^2\norm{\M}_{\op}^4\Ntlsq.
\end{align*}
Using Lemma~\ref{lemma: bound hs and op}\emph{(ii)} and \citep[Lemma A.5]{twosampletest}, we obtain
\begin{align*}
        \E\left(\footnotesize{\circled{3}}^2 |(Y_i^0)_{i=1}^{s}\right) & \leq \frac{4}{m}\|\SgL\Sigma_0\SgL\|_{\op}\|\B\|_{\op}^4\|\SgL(\mu_P-\mu_0)\|_{\h}^2 \\ & \leq \frac{4}{m}\Cs^2\norm{\M}_{\op}^4\|\SgL(\mu_P-\mu_0)\|_{\h}^2 \\
        & \stackrel{(*)}{\leq} \frac{4}{m}\Cs^2\norm{\M}_{\op}^4 \norm{\U}_{\Lp}^2,
\end{align*}
and 
\begin{align*}
     \E(\footnotesize{\circled{5}}^2 |(Y_i^0)_{i=1}^{s} ) &\leq \frac{4}{n} \|\SgL\Sigma_P\SgL\|_{\op}  \|\B\|_{\op}^4\|\SgL(\mu_0-\mu_P)\|_{\h}^2 \\
     &  \leq \frac{4}{n} (1+2\sqrt{\Cl}\norm{\U}_{\Lp})  \Cs^2 \norm{\M}_{\op}^4 \|\SgL(\mu_0-\mu_P)\|_{\h}^2  \\
     & \stackrel{(*)}{\leq} \frac{4}{n} (1+2\sqrt{\Cl}\norm{\U}_{\Lp})  \Cs^2 \norm{\M}_{\op}^4 \norm{\U}_{\Lp}^2,
\end{align*}
where $(*)$ follows from using $g_\lambda(x)=(x+\lambda)^{-1}$ with $C_1=1$ in Lemma \ref{lemma: bounds for eta}. For term \circled{\footnotesize{4}}, using Lemmas \ref{lemma:bound U-statistic3} and  \ref{lemma: bound hs and op}, we have
\begin{align*}
     \E(\circled{4}^2 |(Y_i^0)_{i=1}^{s} ) &\leq  \frac{4}{nm}\|\B\|_{\op}^4\|\SgL\Sigma_P\SgL\|_{\hs}\|\SgL\Sigma_0\SgL\|_{\hs} \\ 
     & \leq \frac{4}{nm}\Cs^2 \norm{\M}_{\op}^4 (2\sqrt{\Cl} \norm{\U}_{\Lp}+\sqrt{2}\Ntl) \Ntl \\
     & \leq \frac{4}{nm}\Cs^2 \norm{\M}_{\op}^4 (2\sqrt{\Cl}\Ntl \norm{\U}_{\Lp}+\sqrt{2}\Ntlsq).
\end{align*}
Combining these bounds with the fact that $\sqrt{ab} \leq \frac{a}{2}+\frac{b}{2}$, and that $(\sum_{i=1}^k a_k)^2 \leq k \sum_{i=1}^k a_k^2$ for any $a,b,a_k \in \R$, $k \in \N$ yields that 
\begin{align*}
    &\E\left[ (\stat-\zeta)^2 |(Y_i^0)_{i=1}^{s} \right] \\
    & \quad \lesssim \norm{\M}_{\op}^4 \left(\frac{\Cl \norm{\U}_{\Lp}^2+\Ntlsq}{n^2}+\frac{\sqrt{\Cl}\norm{\U}_{\Lp}^3+\norm{\U}_{\Lp}^2}{n}\right.\\
&\qquad\qquad\qquad\qquad\qquad\left.    + \frac{\Ntlsq}{m^2}+\frac{\norm{u}_{\Lp}^2}{m}\right) \\
    & \quad \lesssim \norm{\M}_{\op}^4 \left(\frac{\Cl \norm{\U}_{\Lp}^2+\Ntlsq}{n^2}+\frac{\sqrt{\Cl}\norm{\U}_{\Lp}^3+\norm{\U}_{\Lp}^2}{n}\right),
\end{align*}
where in the last inequality we used $m \geq n$. 

When $P=P_0$, and using the same lemmas as above, we have 
\begin{align*}
        \E\left(\footnotesize{\circled{1}}^2 |(Y_i^0)_{i=1}^{s}\right) &\leq \frac{4}{n^2}\Cs^2\norm{\M}_{\op}^4\Ntlsq,\\
        \E\left(\footnotesize{\circled{2}}^2 |(Y_i^0)_{i=1}^{s}\right) &\leq \frac{4}{m^2}\Cs^2\norm{\M}_{\op}^4\Ntlsq,\\
        \E\left(\footnotesize{\circled{4}}^2 |(Y_i^0)_{i=1}^{s} \right) &\leq \frac{4}{nm}\Cs^2\norm{\M}_{\op}^4 \Ntlsq,
\end{align*}
and $\footnotesize{\circled{3}} = \footnotesize{\circled{5}}= 0$. Therefore, 
\begin{align*}
    \E[(\stat)^2 | (Y_i^0)_{i=1}^{s}] &= \E \left[\left(\circled{\footnotesize{1}} + \circled{\footnotesize{2}} + \circled{\footnotesize{4}}\right)^2 | (Y_i^0)_{i=1}^{s}\right]  
     \stackrel{(*)}{=} \E\left(\circled{\footnotesize{1}}^2 + \circled{\footnotesize{2}}^2 + \circled{\footnotesize{4}}^2|(Y_i^0)_{i=1}^{s} \right) \\ & \stackrel{(\dag)}{\leq} \Cs^2\norm{\M}_{\op}^4\Ntlsq\left(\frac{6}{m^2}+\frac{6}{n^2}\right),
\end{align*}
where $(*)$ follows by noting that $\E\left(\circled{\footnotesize{1}}\cdot\circled{\footnotesize{2}}\right)=\E\left(\circled{\footnotesize{1}}\cdot\circled{\footnotesize{4}}\right)=\E\left(\circled{\footnotesize{2}}\cdot\circled{\footnotesize{4}}\right) = 0$ under the assumption $P=P_0$, and $(\dag)$ follows using  $\sqrt{ab} \leq \frac{a}{2}+\frac{b}{2}.$
\end{proof}

\begin{lem} \label{lemma:bound quantile}
For $0<\alpha\leq e^{-1}$, $\delta>0$ and $m \geq n$, there exists a constant $ C_5>0$ such that
\begin{align*}
    P_{H_1}(\qq \leq C_5\gamma ) \geq 1-\delta ,
\end{align*}
where $$\gamma = \frac{\norm{\M}_{\op}^2\log\frac{1}{\alpha}}{\sqrt{\delta}n}\left(\sqrt{\Cl}\norm{u}_{\Lp}+\Ntl+\Cl^{1/4}\norm{u}^{3/2}_{\Lp}+\norm{u}_{\Lp}\right)+\frac{\zeta\log\frac{1}{\alpha}}{\sqrt{\delta}n},$$ $\zeta = \norm{\gSLh(\mu_0-\mu_P)}_{\h}^2$, and $\Cl$ is defined in Lemma~\ref{lemma: bound hs and op}.
\end{lem}
\begin{proof}
The proof is similar to that of \cite[Lemma A.15]{twosampletest} and involves replacing $\Sigma_{PQ}$ with $\Sigma_0$, $R$ with $P_0$, and $\mu_Q$ with $\mu_0$. Then the desired result follows by using $m \geq n$.
\end{proof}
\end{document}